\documentclass[11pt]{amsart}

\makeatletter
\newcommand{\RemoveAlgoNumber}{\renewcommand{\fnum@algocf}{\AlCapSty{\AlCapFnt\algorithmcfname}}}
\newcommand{\RevertAlgoNumber}{\algocf@resetfnum}
\makeatother
\usepackage{amsmath}
\usepackage{amsfonts}
\usepackage{amssymb}
\usepackage{amsthm}
\usepackage{blkarray}
\usepackage{bm}
\usepackage{bbm}
\usepackage{enumerate}
\usepackage{epsfig}
\usepackage[T1]{fontenc}
\usepackage{graphicx}
\usepackage{mathrsfs}
\usepackage{nccmath}
\usepackage{url}
\usepackage[usenames,dvipsnames]{xcolor}
\usepackage[colorlinks=true,linkcolor=NavyBlue,urlcolor=RoyalBlue,citecolor=PineGreen,bookmarks=true,bookmarksopen=true,bookmarksopenlevel=2,unicode=true,linktocpage]{hyperref}
\usepackage[all]{xy}
\usepackage{shuffle}

\theoremstyle{plain}
\newtheorem*{theorem*}{Theorem}
\newtheorem*{question*}{Question}
\newtheorem{theorem}{Theorem}[section]
\newtheorem{lemma}[theorem]{Lemma}
\newtheorem{proposition}[theorem]{Proposition}

\newtheorem{corollary}[theorem]{Corollary}

\theoremstyle{definition}
\newtheorem{definition}[theorem]{Definition}
\newtheorem{example}[theorem]{Example}

\theoremstyle{remark}

\renewcommand{\1}{\mathbbm{1}}

\def\adj{{\mathrm{adj}}}
\newcommand{\bH}{\mathbb{H}}
\newcommand{\build}[3]{\mathrel{\mathop{\kern 0pt#1}\limits_{#2}^{#3}}}
\renewcommand{\C}{\mathbb{C}}
\def\codet{\text{codet}}
\renewcommand{\d}{{\; {\rm d}}}
\newcommand{\dd}{{\rm d}}
\newcommand{\DLP}{{\rm DLP}}
\newcommand{\E}{\mathbb{E}}
\newcommand{\End}{\mathrm{End}}
\def\epsilon{\varepsilon}
\newcommand{\ext}{
  \mathchoice{\raisebox{1pt}{$\displaystyle\medop\bigwedge$}}
             {\raisebox{1pt}{$\bigwedge$}}
             {\raisebox{0.5pt}{$\scriptstyle\bigwedge$}}
             {\raisebox{0.2pt}{$\scriptscriptstyle\bigwedge$}}}
\def\g{{\mathfrak g}}
\newcommand{\GL}{{\mathrm{GL}}}
\newcommand{\Gr}{\mathsf{Gr}}
\newcommand{\hodge}{{\,\star\,}}
\newcommand{\Hom}{\mathrm{Hom}}
\newcommand{\id}{\mathrm{id}}
\newcommand{\ii}{\mathfrak{i}}
\newcommand{\II}{\mathfrak{I}}
\newcommand{\Inc}[1]{{\mathsf{Z}#1}}

\newcommand{\jj}{\mathfrak{j}}
\newcommand{\JJ}{\mathfrak{J}}
\newcommand{\kk}{\mathfrak{k}}
\newcommand{\KK}{\mathfrak{K}}
\newcommand{\N}{\mathbb{N}}
\newcommand{\old}[1]{}
\newcommand\op[1]{\mathsf{#1}}
\renewcommand{\P}{\mathbb{P}}

\newcommand{\proj}[1]{{\sf\Pi}^{#1}}
\newcommand{\Q}{{\mathsf Q}}
\newcommand{\qdet}{\mathrm{Qdet}}
\newcommand{\R}{\mathbb{R}}
\newcommand{\rC}{\mathsf{C}}
\def\Res{\mathsf{Res}}
\newcommand{\rR}{\mathsf{R}}
\renewcommand{\S}{\mathfrak{S}}
\newcommand{\sdim}{\mathop{\underline{\dim}}}
\newcommand{\SL}{{\mathrm{SL}}}

\newcommand{\Tr}{\mathrm{Tr}}
\renewcommand{\U}{\mathrm{U}}
\newcommand{\ul}{\underline}
\newcommand{\Vect}{\mathop{\rm Vect}}
\newcommand\X{{\mathsf X}}
\newcommand{\Z}{\mathbb{Z}}

\renewcommand\leq{\leqslant}
\renewcommand\geq{\geqslant}

\title{Determinantal probability measures on Grassmannians}

\author{Adrien Kassel}
\address{Adrien Kassel -- CNRS, \'Ecole Normale Sup\'erieure de Lyon}
\email{adrien.kassel@ens-lyon.fr}

\author{Thierry L\'evy}
\address{Thierry L\'evy -- LPSM, Sorbonne Universit\'e, Paris}
\email{thierry.levy@sorbonne-universite.fr}

\thanks{\textit{Acknowledgments of support.} We thank, for their hospitality and support, the Mathematisches Forschungsinstitut Oberwolfach (MFO) who granted us a \emph{Research in Pairs} fellowship, as well as the Forschungsinstitut f\"{u}r Mathematik (FIM) in Zurich, the Centre Interfacultaire Bernoulli (CIB) in Lausanne, the ETH Zurich, and the Newton Institute in Cambridge where parts of this work were completed. T.L. would like to thank the Isaac Newton Institute for Mathematical Sciences for support and hospitality during the programme {\em Scaling limits, rough paths, quantum field theory} when work on this paper was pursued. This work was supported by EPSRC grant numbers EP/K032208/1 and EP/R014604/1, by a CNRS PEPS JCJC grant, and by the ANR grant number ANR-18-CE40-0033.}

\date{\today}

\keywords{Determinantal measures, geometric probability, integral geometry, random geometry, enumerative geometry, Grassmannians, Pl\"{u}cker coordinates, graded vector spaces, matroid stratification, uniform spanning tree, quantum spanning forest} 

\subjclass[2010]{60G55, 60D05, 14M15, 81P15, 82B20}

\begin{document}

\begin{abstract}
We introduce and study a class of determinantal probability measures generalising the class of discrete determinantal point processes. These measures live on the Grassmannian of a real, complex, or quaternionic inner product space that is split into pairwise orthogonal  finite-dimensional subspaces. They are determined by a positive self-adjoint contraction of the inner product space, 
in a way that is equivariant under the action of the group of isometries that preserve the splitting.
\end{abstract}

\maketitle

\setcounter{tocdepth}{1}
{\small \tableofcontents}

%%%%%%%%%%%%%%%%%%%%%%%%%%%%%%%%%%%%
%%%%%%%%%%%%%%%%%%%%%%%%%%%%%%%%%%%%
%%%%%%%%%%%%%%%%%%%%%%%%%%%%%%%%%%%%

\section*{Introduction}

Determinantal point processes (DPP) are an extensively studied class of random locally finite subsets of  a nice measured topological space, for example of a Polish space endowed with a Borel measure. 
There is a discrete theory and a continuous theory of DPP, corresponding to the cases where this Borel measure is atomic or diffuse; these two cases are conceptually identical, but differ slightly in their presentation and methods. 

The goal of this paper is to introduce a generalisation of the discrete theory of DPP, which we call \emph{determinantal linear processes} (DLP), where instead of a random collection of points drawn from a ground set, we consider a random collection of linear subspaces drawn from blocks of a vector space orthogonally split into finite-dimensional summands. Conceptually, this amounts to studying probability distributions on the Grassmannian of that vector space instead of studying probability distributions on the power set of the ground set, and to consider the co-incidence of inclusion events in the lattice of linear subspaces of that vector space rather than in the lattice of subsets of that set. Although the family of measures we construct is indeed a generalisation of DPP, we emphasise the introduction of a geometric framework (closer to the language of geometric probability) which, already apparent in the exterior algebra description of DPP, here plays a genuine role and better reveals the geometric structure of these processes. Indeed, some of the general results we provide using this approach are new, at least to us, even in the classical case of DPP.

Before discussing the motivation, let us say a short word on the background. DPP as we know them were initially introduced by Macchi~\cite{Macchi} in order to model probabilistically the statistical observables of fermions in certain quantum optics experiments (see also~\cite{Macchi-slides}). This was a groundbreaking work connecting experimental physics with the theory of point processes. Before her work, and in parallel to it, continuous DPP had been studied, among others, by Wigner, Dyson, Mehta, Gaudin, who studied spectra of random matrices in the context of nuclear physics.  Since then, the probability community has extensively studied these processes which arise in various stochastic models (whether discrete: e.g. uniform spanning trees, or continuous: e.g. eigenvalues of certain random matrices) and constitute a class of tractable point processes exhibiting repulsive behaviour a.k.a. negative association (a notion of negative dependence has yet to be fully understood~\cite{Pemantle-negative}, although recent major progress has been made~\cite{Borcea-Branden-Liggett}). Recently, DPP have found many applications in statistics, starting with the foundational work of~\cite{Kulesza-Taskar}; see~\cite{Bardenet} for a review. Along with authoritative papers of Borodin~\cite{Borodin} and Soshnikov~\cite{Soshnikov} (see also~\cite{Shirai-Takahashi} and~\cite{Johansson}), a good entry point in the theory of discrete, or even finite, DPP is the paper of Lyons~\cite{Lyons-DPP}; see also the survey~\cite{Lyons-ICM} and references therein for a more complete overview of the numerous works on the subject. Our work draws closer to the one of Lyons, first by emphasising the point of view of the exterior algebra (which is further justified here, by our discussion of Grassmannians) and second by being motivated by the theory of random spanning trees and their variants (on graphs, simplicial complexes, and vector bundles over these).

To put it in a nutshell, our motivation came from the following modelling question: How to construct a discrete probabilistic model which could be a toy-model for a system of fermionic matter interacting with a gauge field (a similarly motivated question was addressed for bosonic fields in \cite{KL}). Building on the two classical ideas that uniform spanning trees on graphs are a toy-model for fermions, and that a collection of coupled random matrices is a toy-model for a gauge field, our goal was to build a model where a random spanning tree would interact with a collection of coupled random matrices on the edges of the graph. Pushed by considerations of symmetry under the natural linear group acting on this situation, we realized that spanning trees could fruitfully be replaced by a collection of random linear subspaces, one for each edge (see Section~\ref{example:UST-QSF} for an illustration and~\cite{KL4} for a detailed presentation of this model, which we call \emph{quantum spanning forests}). In order to formalise this definition, we had to revisit our point of view on DPP and define DLP. The goal of this paper is to present this theory of DLP. Let us stress that, since no prerequisites are assumed from the reader, this paper may also serve as a self-contained introduction to discrete DPP, from a slightly more geometrical point of view than the usual.

The paper is organised as follows. Section~\ref{sec:dpp-geom} consists in an overview of the path going from DPP to DLP. In that section, we review the definition of DPP and highlight the conceptual change of point of view consisting in replacing subsets by subspaces; this allows us to describe broadly what DLP are, to highlight their main properties, and to illustrate the theory from the point of view of uniform spanning trees and their generalisation, quantum spanning forests. Section~\ref{sec:invariant-measures} lays the groundwork about measures on real and complex Grassmannians, their incidence measures, as well as useful formulas on determinants we will use throughout. Section~\ref{sec:DLP} proceeds to construct DLP (Theorem~\ref{thm:existuniqueDSP}) both as a measure with an explicit density (Proposition \ref{prop:exDLP}) and as the result of a concrete sampling algorithm (Proposition~\ref{prop:sampling}). Section~\ref{sec:geometry-DLP} contains a presentation of the main geometrical properties of DLP; some of these properties, such as Theorem~\ref{prop:meanprojection}, seem to be new even for DPP (we also note the nice, but maybe not so surprising fact, that the support of any of these determinantal measures is determined by a matroid polytope, see Proposition~\ref{prop:lafforgue}). In that section, we also extend the construction of DLP to the infinite-dimensional setting. Section~\ref{sec:extalg} gives a presentation of DLP from the point of view of the exterior algebra. This is somewhat closer to the initial physics motivation and language, and we in particular give a quantum information theoretic interpretation of DLP in Corollary~\ref{cor:loiDLPext} (based on an algebraic reformulation in Proposition~\ref{prop:lienDLPQM} of the explicit density given in Proposition~\ref{prop:exDLP}), and a very concise rewriting of Theorem~\ref{prop:meanprojection} in Theorem~\ref{thm:wedgeA} (this formalism also allows to see the so-called $\op{l}$-ensembles appear naturally, see Section~\ref{sec:Qmeasure}). Finally, Section~\ref{sec:quat} is devoted to explaining how the theory of DLP extends to the case of quaternionic vector spaces and discusses how DLP behave upon restriction of scalars from quaternions, to complex numbers, to real numbers.\footnote{Our motivation for treating the quaternionic case is threefold; mathematical : ``As a matter of principle, one should always consider the three cases $\R$, $\C$, and $\bH$, and these are the only three finite-dimensional real division algebras''~\cite[page 9]{Brocker}; physical: the orthogonal, unitary, and symplectic groups are natural from the point of view of lattice gauge theory \cite{Levy-spin}; and concrete: we give an example of such a quaternionic process in Example \ref{ex:QQSF}.} The paper closes on a few directions of research we find interesting.

%%%%%%%%%%%%%%%%%%%%%%%%%%%%%%%%%%%%
%%%%%%%%%%%%%%%%%%%%%%%%%%%%%%%%%%%%
%%%%%%%%%%%%%%%%%%%%%%%%%%%%%%%%%%%%

\section{Overview}\label{sec:dpp-geom}

In this section, which precedes the formal introduction of definitions and statements, we present an overview of the path which leads from the classical notion of DPP to that of DLP and provide an example. 

\subsection{Finite determinantal point processes}

Let us fix an integer $d\geq 1$ and consider the set $S=\{1,\ldots,d\}$. Let $K\in M_{d}(\C)$ be a matrix. For each subset $J$ of $S$, let us denote by $K^{J}_{J}$ the sub-matrix of $K$ obtained by erasing all rows and  columns whose indices do not belong to $J$. 

We say that a random subset $\X$ of $S$ is {\em determinantal with kernel $K$} if for all subsets $J$ of $S$, the equality
\begin{equation}\label{eq:DPPJX}
\P(J\subseteq \X)=\det K^{J}_{J}
\end{equation}
holds. Probabilities of the form $\P(J\subseteq \X)$ are called {\em incidence probabilities}. 

One interesting aspect of this definition is that there is no reason why there should exist a random subset $\X$ satisfying~\eqref{eq:DPPJX}, and in general there does not. On the other hand, if there does, then it is an elementary fact that~\eqref{eq:DPPJX} characterises the distribution of $\X$ completely (see~\eqref{eq:mobiusfini}).

One of the first results of the theory is that a random subset satisfying \eqref{eq:DPPJX} exists whenever $K$ is a self-adjoint {\em contraction}, that is, a Hermitian matrix such that $0\leq K \leq 1$. The case where~$K$ is a self-adjoint projection is of particular and fundamental interest. If $K$ is a projection of rank $n$, it can be shown that the subset $\X$ has $n$ elements with probability $1$. In particular, its distribution can be described as follows: for all $n$-subset $I$ of $S$,
\begin{equation}\label{eq:DPPIX}
\P(\X=I)=\det K^{I}_{I}.
\end{equation}
The equality
\begin{equation}\label{eq:DPPPyth}
\sum_{I\subseteq S, |I|=n}\det K^{I}_{I}=1
\end{equation}
can be understood as the statement that the $n$-th elementary symmetric function of the eigenvalues of $K$ is equal to $1$, or as an instance, in the appropriate inner product space, of Pythagoras' theorem (see Proposition~\ref{prop:supportextalg} for a precise, but more general statement, and the remark right after that proposition). 

A determinantal point process with arbitrary contraction kernel can be realised as a mixture of determinantal point processes with projection kernels. In this more general case, the formula~\eqref{eq:DPPIX} is replaced by the following:
\begin{equation}\label{eq:DPPCIX}
\P(\X=I)=\det (KP^{I}+(I_{d}-K)P^{I^{c}}),
\end{equation}
where $I^{c}=S\setminus I$ is the complement of $I$, and $P^{I}$ denotes the diagonal matrix with entries equal to $1$ in the columns indexed by elements of $I$, and $0$ elsewhere.
The fact that these probabilities add up to $1$ is a consequence of the equality
\begin{equation}\label{eq:DPPCPyth}
\sum_{I\subseteq S}\det (AP^{I}+BP^{I^{c}})=\det(A+B),
\end{equation}
which is valid for any two $d\times d$ square matrices $A$ and $B$, and an expression of the multilinearity of the determinant with respect to columns (see also Proposition \ref{prop:multilin} for a proof of a more general formula). The number of points of a determinantal point process with kernel $K$ is random, with the distribution of the sum of independent Bernoulli random variables with parameters equal to the eigenvalues of $K$. 

\subsection{A geometric point of view}

The point of view which we adopt in this paper is that the elements of the set $S$ label the vectors of the canonical basis of the vector space $\C^{d}$, or more generally of an orthonormal basis $(e_{1},\ldots,e_{d})$ of an arbitrary Hermitian space $E$ of dimension~$d$. The matrix $K$ is the matrix in this basis of a self-adjoint endomorphism $\op{k}$ of $E$. Then, we interpret the random subset $\X$ as a {\em random linear subspace} of $E$, namely the random subspace
\begin{equation}\label{eq:subsetsubspace}
\Q=\Vect(e_{i} : i \in \X).
\end{equation}
This way of thinking of a subset as a linear subspace was already used, for instance, in \cite[Section 1.3]{Okounkov}.

It is thus fair to say that a determinantal point process produces, from an orthonormal basis $(e_{1},\ldots,e_{d})$ of an inner product space $E$ and a self-adjoint operator $\op{k}$ on $E$ such that $0\leq \op{k}\leq 1$, a random linear subspace of $E$ that is adapted to the basis $(e_{1},\ldots,e_{d})$ in the sense that it is a sum of lines generated by these vectors: 
\begin{equation}\label{eq:diagrammeDPP}
 \begin{minipage}{0.5\textwidth}
\begin{enumerate}[\raisebox{0.5pt}{$\scriptstyle\triangleright$}]
\item orthonormal basis $(e_{1},\ldots,e_{d})$ of $E$
\item self-adjoint operator $\op{k}$ on $E$ with $0\leq \op{k}\leq 1$
\end{enumerate}
\end{minipage}\Big\rangle \hspace{-5.2pt} \rightsquigarrow 
\begin{minipage}{0.27\textwidth}
\begin{center}
random subspace of $E$\\  adapted to $(e_{1},\ldots,e_{d})$
\end{center}
\end{minipage}
\end{equation}
The dimension of the random subspace of $E$ is itself random, with the distribution of the sum of $d$ independent Bernoulli random variables with parameters given by the eigenvalues of $\op{k}$. In the special case where $\op{k}$ is a projection, this dimension is almost surely equal to the rank of $\op{k}$.

In this construction, the orthonormal basis of $E$ is only used through the {\em splitting} of $E$ that it induces, that is, the orthogonal decomposition
\[E=\C e_{1}\oplus \ldots \oplus \C e_{d}.\]

The goal of this paper is to extend the correspondence \eqref{eq:diagrammeDPP} to the situation where the orthonormal basis, or the splitting that it induces, is replaced by an arbitrary orthonormal splitting 
\[E=E_{1}\oplus \ldots \oplus E_{s}\]
of $E$ as an orthogonal direct sum of linear subspaces. A subspace $Q$ of $E$ is said to be {\em adapted} to this splitting if it is a sum of subspaces of $E_{1},\ldots,E_{s}$. This can be written in several equivalent ways, such as
\[Q=(Q\cap E_{1})\oplus \ldots \oplus (Q\cap E_{s}) \ \text{ or } \ \dim Q = \sum_{i=1}^{s} \dim(Q\cap E_{i}).\]
The correspondence replacing \eqref{eq:diagrammeDPP} is then
\begin{equation}\label{eq:diagrammeDSP}
 \begin{minipage}{0.5\textwidth}
\begin{enumerate}[\raisebox{0.5pt}{$\scriptstyle\triangleright$}]
\item splitting $E=E_{1}\oplus \ldots \oplus E_{s}$
\item self-adjoint operator $\op{k}$ on $E$ with $0\leq \op{k}\leq 1$
\end{enumerate}
\end{minipage}\Big\rangle \hspace{-5.2pt} \rightsquigarrow 
\begin{minipage}{0.3\textwidth}
\begin{center}
random subspace of $E$\\ adapted to the splitting
\end{center}
\end{minipage}
\end{equation}
As before, the case where $\op{k}$ is a projection is special, in that the dimension of the random subspace of $E$ is almost surely equal to the rank of $\op{k}$.  

\subsection{Determinantal measures on Grassmannians}\label{sec:detmeasGr}

The basic data of our construction is thus an inner product space $E$ endowed with a splitting $E=E_{1}\oplus \ldots \oplus E_{s}$ and a kernel $\op{k}$, that is a self-adjoint endomorphism of $E$ with spectrum contained in the interval $[0,1]$. We use the letter $\sigma$ to denote the splitting, and to label objects which depend on it. 

From this data, we construct a probability measure $\mu_{\sigma,\op{k}}$ on the subset $\Gr(E,\sigma)$ of the Grassmannian of $E$ consisting of all linear subspaces of $E$ adapted to the splitting. 

A random linear subspace $\Q$ of $E$ distributed according to this measure $\mu_{\sigma,\op{k}}$ can be sampled according to the following very simple procedure:

\vspace{7mm}
\begin{equation}\label{eq:sampling}
\end{equation}
\vspace{-1.75cm}
\begin{itemize}
\item pick uniformly at random an orthonormal basis of each of the spaces $E_{1},\ldots,E_{s}$ and agregate them to form an orthonormal basis of $E$,
\item sample the determinantal point process with kernel $\op{k}$ associated with this basis, seen as a linear subspace of $E$ according to \eqref{eq:subsetsubspace}.
\end{itemize}

It is not clear from this description of the measure $\mu_{\sigma,\op{k}}$ to what extent it deserves to be called a determinantal probability measure. It turns out that it satisfies a property that is, in this new context, the closest possible analogue of \eqref{eq:DPPJX}. To formulate this property, we define on $\Gr(E,\sigma)$ the {\em incidence measure} $\Inc{\mu}_{\sigma,\op{k}}$ of $\mu_{\sigma,\op{k}}$ as a dimension-biased version of a uniform random subspace of $\Q$ (see Definition \ref{def:defZ}). In the case of a finite determinantal point process $\X$, this incidence measure is simply the counting measure
\[\E\Big[\sum_{J\subseteq \X} \delta_{J}\Big]\]
which to each singleton $\{J\}$ assigns the mass $\P(J\subseteq \X)$. As this special case suggests, $\Inc{\mu}_{\sigma,\op{k}}$ is in general a finite measure rather than a probability measure. We will prove that, just as the incidence measure of $\X$ characterises its distribution, the measure $\mu_{\sigma,\op{k}}$ can be recovered from~$\Inc\mu_{\sigma,\op{k}}$, 
 by a continuous version of the M\"obius inversion in the lattice of linear subspaces of~$E$ adapted to the splitting $\sigma$ (see Proposition \ref{prop:incmeaschar}). Moreover, the density of $\Inc{\mu}_{\sigma,\op{k}}$ with respect to a natural reference measure $\nu^{E,\sigma}$ on $\Gr(E,\sigma)$ is given by minors of $\op{k}$ (see Section \ref{sec:minors} for the notation):
\begin{equation}\label{eq:DSR}
\d (\Inc{\mu}_{\sigma,\op{k}})(R)=\det \op{k}^{R}_{R} \d\nu^{E,\sigma}(R),
\end{equation}
in close analogy to \eqref{eq:DPPJX}. 

\subsection{Properties of the determinantal measures}

In this paper, we choose to define the measure $\mu_{\sigma,\op{k}}$ by~\eqref{eq:DSR} rather than by the algorithm \eqref{eq:sampling}, which will be seen as one of its properties.

Our first task will be to prove that there exists, for every splitting $\sigma$ of $E$ and every contraction~$\op{k}$, a unique probability measure $\mu_{\sigma,\op{k}}$ on $\Gr(E,\sigma)$ satisfying \eqref{eq:DSR}. Then we will prove that this measure satisfies the following properties.

\begin{description}
\item[Sampling] The outcome of the procedure \eqref{eq:sampling} is a random subspace of $E$ distributed according to $\mu_{\sigma,\op{k}}$.
\item[Dimension] Let $\Q$ be distributed according to $\mu_{\sigma,\op{k}}$. The dimension of $\Q$ has the distribution of the sum of $d$ independent Bernoulli random variables with parameters given by the eigenvalues of $\op{k}$. More precisely, consider the random vector of integers $(D_{1},\ldots,D_{s})=(\dim(\Q\cap E_{1}),\ldots,\dim(\Q\cap E_{s}))$ and fix $s$ real numbers $t_{1},\ldots,t_{s}$. Let $T$ be the endomorphism of $E$ that acts on $E_{i}$ by multiplication by $t_{i}$ for each $i\in \{1,\ldots,s\}$. Let finally $\proj{\Q}$ denote the orthogonal projection on $\Q$. Then
\[\E\big[e^{\Tr(\proj{\Q}T)}\big]=\E\big[e^{t_{1}D_{1}+\ldots+t_{d}D_{s}}\big]=\det(\id_{E}+\op{k}(e^{T}-1)).\]
\item[Orthocomplement] The random subspace $\Q^{\perp}$ is also determinantal, distributed according to $\mu_{\sigma,1-\op{k}}$.
\item[Scaling] Pick $p\in[0,1]$. Sample $\Q$ according to $\mu_{\sigma,\op{k}}$. For each $i\in \{1,\ldots,s\}$, sample a binomial random variable $d_{i}$ with parameters $\dim(\Q\cap E_{i})$ and $p$, and choose uniformly at random a subspace ${\sf R}_{i}$ of $\Q\cap E_{i}$ of dimension $d_{i}$. The direct sum ${\sf R}_{1}\oplus \ldots \oplus {\sf R}_{s}$ is distributed according to $\mu_{\sigma,p\op{k}}$.
\item[Restriction] Choose $t\in \{1,\ldots,s\}$ and consider the subspace $F=E_{1}\oplus \ldots \oplus E_{t}$ of $E$ endowed with the obvious splitting. Then $\Q\cap F$ is a random subspace of $F$ associated to the kernel $\op{k}_{F}^{F}$, obtained from $\op{k}$ by compression on $F$ (see Section \ref{sec:minors}).
\item[Equivariance] Let $u$ be an isometry of $E$ preserving the splitting $\sigma$. Then $u(\Q)$ is distributed according to $\mu_{\sigma,u\op{k}u^{*}}$. More generally, if $u$ is an isometry of $E$, without any special relation to $\sigma$, then $u(\Q)$ is distributed according to $\mu_{u(\sigma),u\op{k}u^{*}}$. 
\item[Extension to infinite dimensional spaces] Let $E=\bigoplus_{i\geq 1} E_{i}$ be an  infinite-dimensional inner product space written as the orthogonal direct sum of finite-dimensional subspaces. Let $\op{k}$ be a self-adjoint operator on $E$ such that $0\leq \op{k}\leq 1$. There exists a unique random linear subspace $\Q$ of $E$ adapted to its (infinite) splitting such that for all $s\geq 1$, setting $E_{\leq s}=E_{1}\oplus \ldots \oplus E_{s}$, the random subspace $\Q\cap E_{\leq s}$ of $E_{\leq s}$ is determinantal with kernel~$\op{k}^{E_{\leq s}}_{E_{\leq s}}$. 
\item[Stochastic domination] If $\op{k}_{1}$ and $\op{k}_{2}$ are two kernels on $E$ such that $0\leq \op{k}_{1} \leq \op{k}_{2} \leq 1$, then the measure $\mu_{\sigma,\op{k}_{1}}$ is stochastically dominated by $\mu_{\sigma,\op{k}_{2}}$. This means that if $\Q_{1}$ and~$\Q_{2}$ are respectively distributed according to $\mu_{\sigma,\op{k}_{1}}$ and $\mu_{\sigma,\op{k}_{2}}$, then for all continuous non-decreasing functions $f$ on the partially ordered space $(\Gr(E,\sigma),\subseteq)$, one has the inequality $\E[f(\Q_{1})]\leq \E[f(\Q_{2})]$.
\item[Negative association] Let $R$ be a subspace of $E$ equal to the direct sum of some of the elements of $\sigma$. Let $f,g$ be continuous non-decreasing functions on $\Gr(E,\sigma)$. Then \[\E[f(\Q\cap R)g(\Q\cap R^{\perp})]\leq \E[f(\Q\cap R)] \E[g(\Q\cap R^{\perp})]\,.\]
\item[Uniqueness] If $\op{k}$ is the orthogonal projection on a subspace $H$ of $E$, then almost surely, \[\Q\oplus H^{\perp}=\Q^{\perp}\oplus H=E\,.\] In particular, the map $h\mapsto \langle h,\cdot\rangle$ from $H$ to the dual $\Q^*$ of $\Q$, 
%\begin{align*}
%H & \longrightarrow \Q^{*} \\
%h & \longmapsto \langle h, \cdot \rangle
%\end{align*}
is almost surely injective (hence bijective). This parallels results of Lyons and Bufetov--Qiu--Shamov stating that almost every realisation of a determinantal point process associated with a kernel that is a projection on a finite-dimensional space of functions is a uniqueness set for this space of functions (see \cite[Thm. 7.11]{Lyons-DPP} and \cite{BufetovQiuShamov} for an extension to Polish spaces).
\item[Mean projection] Let us assume again that $\op{k}$ is the orthogonal projection onto a subspace~$H$ of $E$. Then the orthogonal projection onto $H$ is equal to the expectation of the projection onto $\Q$ parallel to $H^{\perp}$, i.e. 
\[\E\big[\op{P}^{\Q} \big]=\proj{H}\,,\]
where $\proj{H}$ denotes the orthogonal projection on $H$ and $\op{P}^{\Q}$ the projection onto $\Q$ parallel to $H^{\perp}$. In fact, much more is true, and we have the following equality of endomorphisms of the exterior algebra of $E$:
\begin{equation}\label{eq:wedgeAintro}
\E\big[\ext \op{P}^{\Q} \big]=\ext\proj{H}.
\end{equation}
In other words, if we choose a basis of $E$, it is not only every entry, but every minor of the matrix of $\proj{H}$ that is the expectation of the corresponding minor of the random matrix $\op{P}^{\Q}$.
Note that, thanks to the self-adjointness of $\proj{H}$, one can replace, in \eqref{eq:wedgeAintro}, the operator $\op{P}^{\Q}$ by its adjoint, which is the projection on $H$ parallel to $\Q^{\perp}$.
\end{description}
 
\subsection{Example}\label{example:UST-QSF}

We now illustrate the above framework with a concrete example motivated by statistical physics.
In order to keep this section short, we voluntarily go briefly over the following definitions. The reader is invited to consult the given references for more details.

Let $X$ be a finite weighted simplicial complex of dimension $m\ge 1$. For each $0\le k \le m$, let $\Omega^k(X,\R)$ be the space of real $k$-forms, that is, antisymmetric real functions over oriented $k$-cells. There is a structure of chain complex induced by a collection of maps $d:\Omega^{k-1}(X,\R)\to\Omega^k(X,\R)$ such that $d\circ d=0$ (one for each $1\le k\le m$, although, following tradition, we omit the dependency of $d$ on $k$). By duality with respect to natural scalar products on these spaces, there is a collection of associated dual maps $d^*:\Omega^{k}(X,\R)\to\Omega^{k-1}(X,\R)$.

Coming from the structure of chain complex, there are some natural subspaces of $\Omega^k(X,\R)$ to consider: $\bigstar^k=\mathrm{im}(d)$, the space of exact forms, $\lozenge^k=\mathrm{ker}(d^*)$, the space of cycles; and likewise $\bigstar^{*k}=\mathrm{im}(d^*)$ and $\lozenge^{*k}=\ker(d)$. Since the complex is finite, we immediately have the orthogonal decompositions
$\Omega^k(X,\R)=\bigstar^k \oplus \lozenge^k=\lozenge^{*k}\oplus \bigstar^{*k}$.

By further considering the possibly trivial subspace of harmonic forms $H^k=\ker(d\circ d ^* + d^* \circ d)$, and noting that $\lozenge^k=H^k \oplus \bigstar^{*k}$, we arrive at the refined decomposition
\[\Omega^k(X,\R)=\bigstar^k \oplus H^k \oplus \bigstar^{*k}\,,\]
which is a discrete analog of the Hodge decomposition for differential forms in geometry, as first considered by Eckmann \cite{Eckmann}.

For any unoriented $k$-cell $\tau$, let $\omega_\tau$ be an associated $k$-form, defined up to sign, if $\{\tau_1,\tau_2\}$ denote the two orientations of that cell, by $\op{1}_{\omega_1}-\op{1}_{\omega_2}$. The line $\mathsf{L}_\tau=\R \omega_\tau$ is independent of that choice of sign, and there is a natural splitting $\sigma$ given by 
\[\Omega^k(X,\R)=\bigoplus_{\tau\in C^k(X)} \mathsf{L}_\tau \,,\]
where $C^k(X)$ denotes the collection of unoriented $k$-cells of $X$. 

The DLP associated to $H=\bigstar^k$ (respectively $H=\bigstar^{*k}$) and the above splitting $\sigma$, is a random subspace $\Q$ of $\Omega^k(X,\R)$ which is a supplement to  $H^\perp=H^k\oplus \bigstar^{*k}=\lozenge^k$ (respectively $H^\perp=\bigstar^k \oplus H^k=\lozenge^{*k}$).

Given a subset $T$ of $k$-cells, we let $Q_T=\bigoplus_{\tau\in T} \mathsf{L}_\tau$. This establishes a correspondence between subsets of $C^k(X)$ and subspaces of $\Omega^k(X,\R)$. The above DLP correspond under this mapping to the DPP introduced by Lyons \cite{Lyons-Betti} (the so-called upper and lower matroidal measures, dual of one another, defined for any chain complex) and the above properties about the geometric position of the random subspace $\Q$ with respect to $H$ are equivalent to \cite[Theorem A]{CCK-complexes}. As we have already mentioned, this last statement can be further strenghtened to a statement in the exterior algebra, and moreover it holds true for all DLP, not just the example described here, see Theorems \ref{prop:meanprojection} and \ref{thm:wedgeA} below.

In the case when $m=1$, $X$ is simply a finite graph, and $\bigstar^1=\mathrm{im}(d)$, $H^1=0$, and $\lozenge^1=\mathrm{ker}(d^*)$, where $d$ is sometimes called the discrete derivative map, and $d^*$ its adjoint, the discrete divergence map. The DLP just described, a random supplement to the space of cycles whose average is the space of stars, is simply the \emph{uniform spanning tree measure}, initially shown to be a DPP by Burton and Pemantle \cite{Burton-Pemantle}.

It is natural to generalise the above setup by twisting the chain complex with the data of linear isomorphisms for pairs of adjacent cells (i.e. one isomorphism for each edge of the barycentric subdivision of the simplicial complex, see e.g. \cite{ReVe}). In the case of a graph ($m=1$), this amounts to putting orthogonal matrices of size $N\times N$ over half-edges, for a fixed integer $N\ge 1$ we call the rank, the collection of which we denote $h$ and call a connection. There is then a space of vector-valued $1$-forms to be considered, $\Omega^1(X,\R^N)$, a twisting of the map $d$, denoted $d_h$, and corresponding subspaces $\bigstar_h^1$ and $\lozenge_h^1$, which give the decomposition $\Omega^1(X,\R^N)=\bigstar_h^1\oplus \lozenge_h^1$. The lines $\mathsf{L}_\tau$ are moreover replaced by $N$-dimensional vector spaces $\mathsf{F}_\tau$ and they induce a splitting $\Omega^1(X,\R^N)=\bigoplus_{\tau\in C^k(X)} \mathsf{F}_\tau$.

We call the DLP associated to this splitting and the orthogonal projection on $\bigstar_h^1$, a \emph{quantum spanning forest}. The case $N=1$ and $h$ non-trivial, corresponds to the cycle-rooted spanning forest model \cite{Kenyon}, and the above statement about the average projection was shown in \cite[Theorem A]{CCK-line}.

\begin{figure}[ht!]
\begin{center}
\includegraphics[width=12cm]{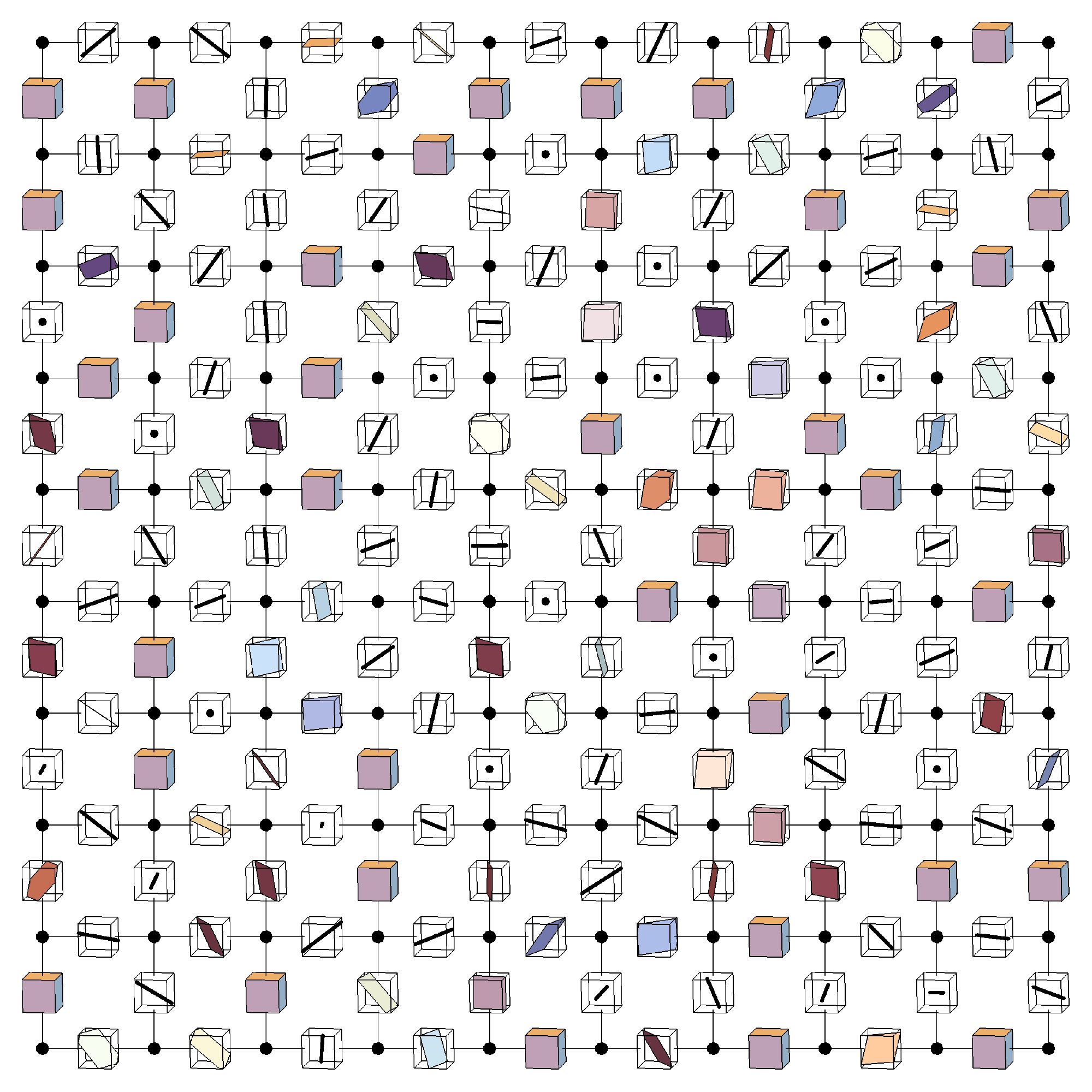}
\caption{\label{fig:exQSF}Simulation of a quantum spanning forest in the case of rank $N=3$. As explained in Section~\ref{example:UST-QSF}, this is a DLP associated to a graph and a collection of orthogonal matrices on its edges. This figure was sampled using an algorithm described in Section~\ref{sec:algosample}.}
\end{center}
\end{figure}

Similarly, for $m\ge 2$ and any $0\le k\le m$, one may consider the DLP on $\Omega^k(X,\R^N)$ induced by a splitting indexed by $C^k(X)$ and the orthogonal projection on $\bigstar_h^k$. This yields a general notion of quantum spanning forest in dimension $k$ and rank $N$.

A sample of quantum spanning forests for $N=3$ on a square grid graph is given in Figure~\ref{fig:exQSF} (in this simulation, the connection $h$ was chosen randomly under a $2$-dimensional Yang--Mills measure; see e.g.~\cite{Levy-survey} for background). A more detailed presentation of this model and its basic properties will appear in~\cite{KL4}.

%%%%%%%%%%%%%%%%%%%%%%%%%%%%%%%%%%%%
%%%%%%%%%%%%%%%%%%%%%%%%%%%%%%%%%%%%
%%%%%%%%%%%%%%%%%%%%%%%%%%%%%%%%%%%%
 
\section{Invariant measures on Grassmannians}\label{sec:invariant-measures}

In this section, we introduce split vector spaces, their Grassmannians, invariant measures on these, and their incidence measures, and prove useful formulas about the integral of determinants of operators with respect to these. This is the formal setup which will be used in Section \ref{sec:DLP} to define DLP.

\subsection{Split inner product spaces} \label{sec:splittings} In this paper, until Section \ref{sec:quat}, vector spaces are over real or complex numbers. Moreover, until further notice, they have finite dimension. Inner products are understood to be Euclidean in the real case and Hermitian in the complex case. Hermitian inner products are taken to be linear with respect to the second variable, according to the physicists' convention. We use, in the real and complex case, the notation $\U(E)$ for the group of isometries of an inner product space $E$. 

\begin{definition}\label{def:sips} An {\em orthogonal splitting}, or simply {\em splitting}, of an inner product space $E$ is a finite ordered sequence $\sigma=(E_{1},\ldots,E_{s})$ of pairwise orthogonal linear subspaces of $E$ of positive dimension such that $E=E_{1}\oplus \ldots \oplus E_{s}$. A {\em split inner product space} is a pair $(E,\sigma)$ where $E$ is an inner product space and $\sigma$ a splitting of $E$.
\end{definition}

We like to think of a splitting of an inner product space as a linearised version of a labelled partition of a finite set. 

\begin{example}\rm 
Extreme examples of splittings of a $d$-dimensional space $E$ are the {\em splittings in lines}, which are the splittings consisting of $d$ pairwise orthogonal lines and, if $d$ is positive, the {\em coarse splitting} $\sigma=(E)$. The unique splitting of the null vector space is the empty splitting $\sigma=()$. In general, the set of all possible splittings of an inner product space identifies with the set of flags, complete or partial, of this vector space.
\end{example}

Let $\Sigma(E)$ denote the set of all splittings of an inner product space $E$. The group $\U(E)$ acts on $\Sigma(E)$: for all $u\in \U(E)$ and $\sigma=(E_{1},\ldots,E_{s})\in \Sigma(E)$, we set
\[u(\sigma)=(u(E_{1}),\ldots,u(E_{s})).\]
The stabiliser of a splitting $\sigma$ is denoted by $\U(E,\sigma)$. If $\sigma=(E_{1},\ldots,E_{s})$, we have, up to an obvious identification, $\U(E,\sigma)=\U(E_{1})\times \ldots \times \U(E_{s})$. Note that an element of $\U(E,\sigma)$ cannot exchange elements of $\sigma$ of equal dimension, if any.

Let us also mention that the set $\Sigma(E)$ carries a natural partial order. Given two splittings $\tau$ and $\sigma$, we say that $\tau$ is \emph{finer} than $\sigma$, and we write $\tau\leq \sigma$, if every element of $\tau$ is contained in an element of $\sigma$.

Our main goal in this paper is to associate to some linear data on $E$, namely a linear subspace or an operator of a certain kind, a probability measure on the Grassmannian of $E$, in a way that is \emph{equivariant} under the action of a group $\U(E,\sigma)$.

\subsection{Grassmannians}\label{sec:grassmannians}

Let $E$ be a $d$-dimensional vector space. For each $n\in \{0,\ldots,d\}$, we denote by $\Gr_{n}(E)$ the variety of all $n$-dimensional linear subspaces of $E$. It is a smooth compact connected real manifold on which the group $\U(E)$ acts smoothly and transitively. We denote by $\nu_{n}^{E}$ the unique Borel measure on $\Gr_{n}(E)$ that is invariant under the action of $\U(E)$ and has total mass $\binom{d}{n}$. This measure can be realised as the average under the action of $\U(E)$ of any Borel measure on $\Gr_{n}(E)$ with total mass $\binom{n}{d}$. In particular, it can be described concretely as follows. Let $\dd u$ denote the normalised Haar measure on $\U(E)$. Let us choose an orthonormal basis $(e_{1},\ldots,e_{d})$ of $E$. For every subset $J$ of $\{1,\ldots,d\}$, let us write $E_{J}=\Vect(e_{i} : i \in J)$. Let us also denote by $|J|$ the cardinality of $J$. Then for all scalar continuous functions $f$ on $\Gr_{n}(E)$, we have
\begin{equation}\label{eq:intnun}
\int_{\Gr_{n}(E)}f(Q)\d\nu_{n}^{E}(Q)=\int_{\U(E)} \sum_{J\subset \{1,\ldots,d\}, |J|=n} f(u(E_{J}))\d u.
\end{equation}

We denote by $\Gr(E)=\bigsqcup_{n=0}^{d} \Gr_{n}(E)$ the full Grassmannian of $E$ and endow it with the measure $\nu^{E}=\sum_{n=0}^{d}\nu^{E}_{n}$, which has total mass $2^{d}$.

Let us now assume that $E$ is an inner product space and let $\sigma=(E_{1}, \ldots, E_{s})$ be a splitting of~$E$. Let us say that a linear subspace $Q$ of $E$ is {\em adapted} to $\sigma$ if one, and hence all of the following equivalent conditions are satisfied: 
\begin{align*}
&\exists (Q_{1},\ldots,Q_{s})\in \Gr(E_{1})\times \ldots \times \Gr(E_{s}), \ Q=Q_{1}\oplus \ldots \oplus Q_{s}, \\
&Q=(Q\cap E_{1})\oplus \ldots \oplus (Q\cap E_{s}),\\
&\dim Q=\dim (Q\cap E_{1})+ \ldots + \dim(Q\cap E_{s}).
\end{align*}
An adapted subspace $Q$ of $(E,\sigma)$ comes with an induced splitting $\sigma_{Q}=(Q\cap E_{1}, \ldots, Q\cap E_{s})$.
We call the $s$-tuple of integers $\sdim Q=(\dim(Q\cap E_{1}),\ldots,\dim(Q\cap E_{s}))$ the {\em split dimension} of $Q$. It will be convenient to denote by $\ul{d}$ the split dimension of $E$ itself, that is, $\ul{d}=(\dim E_{1},\ldots,\dim E_{s})$. We will compare tuples of integers componentwise, so that $\ul{m}\leq \ul{n}$ is equivalent, by definition, to $m_{1}\leq n_{1},\ldots,m_{s}\leq n_{s}$. We will also write $|\ul{n}|=n_{1}+\ldots+n_{s}$, so that an adapted subspace of $E$ with split dimension $\ul{n}$ has ordinary dimension $|\ul{n}|$.

We define the following subsets of $\Gr(E)$, which depend on an integer $n\leq d$ and an $s$-tuple of integers $\ul{n}\leq \ul{d}$:
\begin{itemize} 
\item $\Gr(E,\sigma)$, the set of all linear subspaces of $E$ adapted to $\sigma$,
\item $\Gr_{n}(E,\sigma)=\{Q\in \Gr(E,\sigma) : \dim Q=n\}$,
\item $\Gr_{\ul{n}}(E,\sigma)=\{Q\in \Gr(E,\sigma) : \sdim Q=\ul{n}\}$.
\end{itemize}
For all $\ul{n}=(n_{1},\ldots,n_{s})$, we make without further comment the identification 
\[\Gr_{\ul{n}}(E,\sigma)\simeq \Gr_{n_{1}}(E_{1})\times \ldots \times \Gr_{n_{s}}(E_{s}).\]

The space $\Gr_{\underline{n}}(E,\sigma)$ is acted on smoothly and transitively by $\U(E,\sigma)$ and the unique Borel  measure that is invariant under this action and has total mass $\binom{d_{1}}{n_{1}}\ldots \binom{d_{s}}{n_{s}}$ is the measure
\[\nu_{\underline{n}}^{E,\sigma}=\nu_{n_{1}}^{E_{1}}\otimes \ldots \otimes \nu_{n_{s}}^{E_{s}}.\]
We will use the notation $\binom{\ul{d}}{\ul{n}}=\binom{d_{1}}{n_{1}}\ldots \binom{d_{s}}{n_{s}}$. We define also the measures
\[\nu_{n}^{E,\sigma}=\sum_{\underline{n}\leq \ul{d}, \, |\underline{n}|=n} \nu_{\underline{n}}^{E,\sigma} \ \text{ and } \ \nu^{E,\sigma}=\sum_{n=0}^{d} \nu^{E,\sigma}_{n}\]
on $\Gr_{n}(E,\sigma)$ and $\Gr(E,\sigma)$ respectively. Note that these measures have total masses
\[\nu^{E,\sigma}_{n}(\Gr_{n}(E,\sigma))=\binom{d}{n} \ \text{ and } \ \nu^{E,\sigma}(\Gr(E,\sigma))=2^{d},\]
regardless of the splitting of $E$. 

The measures $\nu^{E,\sigma}_{n}$ can be described by a formula analogous to \eqref{eq:intnun}. Let us say that an orthonormal basis of $E$ is {\em adapted} to $\sigma$ if the linear subspace generated by any subset of this basis is adapted to $\sigma$. Equivalently, a basis is adapted if every vector of this basis belongs to one of the subspaces of the splitting. Let us choose such an orthonormal basis $(e_{1},\ldots,e_{d})$ of~$E$ adapted to $\sigma$. Let us denote by $\dd u$ the normalised Haar measure on $\U(E,\sigma)$. Then, for all continuous functions $f$ on $\Gr_{\ul{n}}(E,\sigma)$, we have, with the notation used in \eqref{eq:intnun},
\begin{equation}\label{eq:intnunsigma}
\int_{\Gr_{n}(E,\sigma)} f(Q)\d \nu^{E,\sigma}_{n}(Q)=\int_{\U(E,\sigma)} \sum_{\substack{J\subseteq \{1,\ldots,d\}\\ |J|=n}} f(u(E_{J}))\d u.
\end{equation}
This formula can easily be adapted to describe, instead of the measure $\nu_{n}^{E,\sigma}$, the measure $\nu_{\ul{n}}^{E,\sigma}$ for some $\ul{n}=(n_{1},\ldots,n_{s})$: it suffices to restrict the sum in the right-hand side to subsets $J$ such that $\sdim E_{J}=\ul{n}$, that is, to subsets that pick $n_{1}$ vectors in the subspace $E_{1}$, $n_{2}$ vectors in $E_{2}$, and so on.

\begin{example}\label{ex:grlines}\rm
In the case where $\sigma=(E_{1},\ldots,E_{d})$ is a splitting in lines of $E$, the map from the power set $2^{\{1,\ldots,d\}}$ to $\Gr(E,\sigma)$ which to a set $I\subset \{1,\ldots,d\}$ associates the subspace $E_{I}=\bigoplus_{i\in I} E_{i}$ is a bijection.\footnote{These subspaces are sometimes called \emph{coordinate subspaces}.} Moreover, the measure $\nu^{E,\sigma}$ is simply the counting measure on $\Gr(E,\sigma)$.
\end{example}

It is useful to realise that $\Gr(E,\sigma)$ is stable under the involution $Q\mapsto Q^{\perp}$ which sends a subspace of $E$ to its orthogonal. Moreover, this involution exchanges the subsets $\Gr_{\ul{n}}(E,\sigma)$ and $\Gr_{\ul{d}-\ul{n}}(E,\sigma)$ in a way that is equivariant under the action of $\U(E,\sigma)$, and hence is measure-preserving. Thus, if $f$ is a continuous test function on $\Gr(E,\sigma)$, then
\begin{equation}\label{eq:orthnu}
\int_{\Gr_{\ul{n}}(E,\sigma)} f(Q) \d\nu^{E,\sigma}_{\ul{n}}(Q)=\int_{\Gr_{\ul{d}-\ul{n}}(E,\sigma)} f(Q^{\perp}) \d\nu^{E,\sigma}_{\ul{d}-\ul{n}}(Q).
\end{equation}
We will use this equality, as well as the following slightly elaborated version of it. 

\begin{lemma}\label{lem:orthpi} Let $f$ be a continuous test function on $\Gr(E,\sigma)\times \Gr(E,\sigma)$. We have the equality
\begin{align*}
&\int_{\Gr(E,\sigma)} \bigg(\int_{\Gr(Q,\sigma_{Q})} f(R,Q) \d\nu^{Q,\sigma_{Q}}(R) \bigg) \d\nu^{E,\sigma}(Q)\\
&\hspace{6cm} =\int_{\Gr(E,\sigma)} \bigg(\int_{\Gr(Q,\sigma_{Q})} f(Q^{\perp},R^{\perp}) \d\nu^{Q,\sigma_{Q}}(R) \bigg) \d\nu^{E,\sigma}(Q).
\end{align*}
\end{lemma}

\begin{proof} It is sufficient to prove that for all vectors of integers $0\leq \ul{m} \leq \ul{n} \leq \ul{d}$, the following equality holds:
\begin{align}\nonumber
&\int_{\Gr_{\ul{n}}(E,\sigma)} \bigg(\int_{\Gr_{\ul{m}}(Q,\sigma_{Q})} f(R,Q) \d\nu_{\ul{m}}^{Q,\sigma_{Q}}(R) \bigg) \d\nu_{\ul{n}}^{E,\sigma}(Q)\\  \label{eq:toprovemn}
&\hspace{5cm} =\int_{\Gr_{\ul{d}-\ul{m}}(E,\sigma)} \bigg(\int_{\Gr_{\ul{d}-\ul{n}}(Q,\sigma_{Q})} f(Q^{\perp},R^{\perp}) \d\nu_{\ul{d}-\ul{n}}^{Q,\sigma_{Q}}(R) \bigg) \d\nu_{\ul{d}-\ul{m}}^{E,\sigma}(Q)
\end{align}
since the result follows by summing over $\ul{m}$ and $\ul{n}$.

The set of pairs $(R,Q)$ of subspaces of $E$ such that $Q\in \Gr_{\ul{n}}(E,\sigma)$ and $R\in \Gr_{\ul{m}}(Q,\sigma_{Q})$ is a connected component of a partial flag manifold of $(E,\sigma)$ that could, following tradition, be denoted by $\Gr_{(\ul{m},\ul{n}-\ul{m},\ul{d}-\ul{n})}(E,\sigma)$, the indices being the split dimensions of the successive quotients of the ascending chain $\{0\} \subseteq R \subseteq Q \subseteq E$. This set is acted on transitively by $\U(E,\sigma)$ and the measure described by the left-hand side of \eqref{eq:toprovemn} is its invariant measure of total mass $\binom{\ul{d}}{\ul{n}}\binom{\ul{n}}{\ul{m}}$. 

The map $(R,Q)\mapsto (Q^{\perp},R^{\perp})$ commutes to the action of $\U(E,\sigma)$ and sends $\Gr_{(\ul{m},\ul{n}-\ul{m},\ul{d}-\ul{n})}(E,\sigma)$ to $\Gr_{(\ul{d}-\ul{n},\ul{n}-\ul{m},\ul{m})}(E,\sigma)$. It pushes the invariant measure on the first space to an invariant measure of the same mass on the second space, which is thus the unique invariant measure with this mass, and is precisely the measure described by the right-hand side of \eqref{eq:toprovemn}.
\end{proof}

\subsection{Incidence measure of a random linear subspace}\label{sec:incidence} With the notation introduced so far, our aim is to construct and study a family of probability measures on $\Gr(E,\sigma)$. Such measures can be described by their incidence measures, as we explain now. We start by reviewing the more classical situation of point processes on a finite set which, according to Example \ref{ex:grlines}, is in our language the special case where $\sigma$ is a splitting in lines.

Let $\X$ be a random subset of the set $S=\{1,\ldots,d\}$. By this we mean that we are given a probability measure $\mu$ on the measurable space $(2^{S}, 2^{2^{S}})$ and that $\X$ is the identity map of the probability space $(2^{S},2^{2^{S}},\mu)$. The incidence measure of the distribution of $\X$ is the finite measure $\Inc{\mu}$ on $2^{S}$ defined by
\begin{equation}\label{eq:incidencefini}
\Inc{\mu}=\E\Big[\sum_{J\subseteq \X} \delta_{J}\Big]=\sum_{J\subseteq I\subseteq S} \mu(\{I\}) \delta_{J},
\end{equation}
so that for all $J\subseteq S$, we have
\[\P(J\subseteq \X)=\Inc{\mu}(\{J\}).\]
The letter $\Inc{}$ stands here for {\em zeta}, in reference to the zeta function of the lattice $(2^{S},\subseteq)$, of which the M\"{o}bius function is the inverse (see for example \cite{Rota}).

Point processes are often described by their incidence measure rather than by their distribution.\footnote{This incidence measure, in turn, is usually described by first pushing it forward onto the space $\bigsqcup_{n\geq 0} S^{n}$ by the map 
%\thierry{Si on ne s'en sert pas, c'est peut-être pas la peine de passer dix lignes dessus}
\begin{align*}
2^{S} & \longrightarrow {\sf Meas}\bigg(\bigsqcup_{n\geq 0} S^{n}\bigg)\\
I=\{i_{1},\ldots,i_{n}\} & \longmapsto \sum_{\sigma\in \mathfrak S_{n}} \delta_{(i_{\sigma(1)},\ldots,i_{\sigma(n)})},
\end{align*}
and expressing the density of the resulting measure with respect to the measure $\sum_{n\geq 0} \rho^{\otimes n}$, where $\rho$ is some reference measure on $S$. For each $n\geq 0$, the restriction to $S^{n}$ of the density is called the $n$-th correlation function of the process. We will not need the language of correlation functions, and will instead work with the incidence measure itself.}
It is thus a crucial fact that a measure on $2^{S}$ is uniquely characterised by its incidence measure, thanks to the following inclusion-exclusion principle, or M\"obius inversion formula: for all $I\subseteq S$, we have
\begin{equation}\label{eq:mobiusfini}
\P(\X=I)=\sum_{I\subseteq J \subseteq S} (-1)^{|J\setminus I|} \P(J\subseteq \X).
\end{equation}
Recall that here and thereafter, we denote by $|J|$ the cardinality of a set $J$.

This formula can be checked by a simple computation using the elementary fact that for all $I\subseteq \{1,\ldots,d\}$, the alternating sum $\sum_{J\subseteq I} (-1)^{J}$ is equal to $(1-1)^{|I|}$, that is to $0$, unless $I$ is empty, in which case it is equal to $1$. Let us note that the equality \eqref{eq:mobiusfini} can be written, in terms of measures, under the form
\begin{equation}\label{eq:mobiusfini2}
\mu=\sum_{J\subseteq S} \bigg(\sum_{I\subseteq J} (-1)^{|J\setminus I|} \delta_{I}\bigg) \Inc{\mu}(\{J\}).
\end{equation}
Let us now state and prove the results that extend these relations to our setting of measures on Grassmannians.

\begin{definition}\label{def:defZ} Let $(E,\sigma)$ be a split inner product space. Let $\mu$ be a Borel  measure on $\Gr(E,\sigma)$. The {\em incidence measure} of $\mu$ is the measure 
\begin{equation}\label{eq:definc}
\Inc{\mu}=\int_{\Gr(E,\sigma)} \nu^{Q,\sigma_{Q}} \d\mu(Q),
\end{equation}
where for every $Q\in \Gr(E,\sigma)$, the set $\Gr(Q,\sigma_{Q})$ is seen as a subset of $\Gr(E,\sigma)$, and the measure $\nu^{Q,\sigma_{Q}}$ as a measure on $\Gr(E,\sigma)$.
\end{definition}

This definition is manifestly analogous to \eqref{eq:incidencefini}, once the lattice $(2^{S},\subseteq)$ has been replaced by the lattice $(\Gr(E,\sigma),\subseteq)$.\footnote{Recall that a lattice is a partially ordered set with a notion of meet and join for any pair of elements. As pointed out in~\cite{Klain-Rota}, the two above-mentioned lattices share analogous properties but the main (obvious) difference is that the Grassmannian is not a distributive lattice: in general, for vector spaces $E,F,H$, we do not have $(E\oplus F)\cap H=(E\cap H)\oplus (F\cap H)$. This non-distributivity property is of crucial importance for the fact that DLP are non trivial and this argues in favour of the point of view of considering random subspaces instead of point processes, even in the case of DPP.} Let us stress that the definition of $\Inc{\mu}$ depends on the splitting. 

Let us observe that the total mass of $\Inc{\mu}$, which is equal to the integral of the function $Q\mapsto 2^{\dim Q}$ with respect to $\mu$, is equal at least to the total mass of $\mu$ and at most to $2^{d}$ times this mass. In particular, $\mu$ is a finite measure if and only if $\Inc{\mu}$ is a finite measure.

The main result at this point is the following.

\begin{proposition}\label{prop:incmeaschar} Let $(E,\sigma)$ be a split inner product space. Let $\mu$ be a finite measure on $\Gr(E,\sigma)$. Let $\Inc{\mu}$ be the incidence measure of $\mu$. Then
\begin{equation}\label{eq:mobius}
\mu=\int_{\Gr(E,\sigma)}\bigg( \int_{\Gr(Q,\sigma_{Q})} (-1)^{\dim Q-\dim R} \delta_{R} \d\nu^{Q,\sigma_{Q}}(R)\bigg) \d\Inc{\mu}(Q).
\end{equation}
\end{proposition}

This formula implies in particular that the measure $\Inc{\mu}$ characterises $\mu$ uniquely. In the case of a splitting in lines, it reduces to \eqref{eq:mobiusfini}.

\begin{proof} Since $\mu$ is finite, the measure $\Inc{\mu}$ is also finite and the integral on the right-hand side is well defined, a priori as a signed measure on $\Gr(E,\sigma)$. Let $f$ be a continuous test function on $\Gr(E,\sigma)$. Let us compute the integral of $f$ with respect to the measure defined by the right-hand side of \eqref{eq:mobius}. Using the definition of $\Inc{\mu}$, we find that this integral is equal to
\begin{equation}\label{eq:fmu}
\int_{\Gr(E,\sigma)} \bigg(\int_{\Gr(Q,\sigma_{Q})}\bigg(\int_{\Gr(R,\sigma_{R})}
(-1)^{\dim R-\dim S} f(S) \d\nu^{R,\sigma_{R}}(S)\bigg) \d\nu^{Q,\sigma_{Q}}(R) \bigg)\d\mu(Q).
\end{equation}
Let us fix $Q$ and consider the expression between the outermost pair of brackets. Let us compute this expression by applying Lemma \ref{lem:orthpi} in the split inner product space $(Q,\sigma_{Q})$, paying attention to the fact that orthogonal complements must be taken with respect to $Q$, so that, for instance, wherever $R^{\perp}$ stands in the statement of Lemma \ref{lem:orthpi}, we must write $Q\cap R^{\perp}$. We find, thanks to the equality $\dim (Q\cap S^{\perp})-\dim (Q\cap R^{\perp})=\dim R-\dim S$, that this expression is equal to 
\begin{equation}\label{eq:QQ}
\int_{\Gr(Q,\sigma_{Q})}\bigg(\int_{\Gr(R,\sigma_{R})}
(-1)^{\dim S}  \d\nu^{R,\sigma_{R}}(S)\bigg) (-1)^{\dim R} f(R^{\perp}) \d\nu^{Q,\sigma_{Q}}(R)\,.
\end{equation}
The innermost integral of \eqref{eq:QQ} is equal to $(1-1)^{\dim R}$, that is, to $1$ if $R=\{0\}$ and to $0$ otherwise. Thus, \eqref{eq:QQ} is equal to $f(Q)$, and \eqref{eq:fmu} is the integral of $f$ with respect to $\mu$.
\end{proof}

\subsection{Sub-matrices and compressions}\label{sec:minors} 
In this brief section, we introduce some notation for sub-matrices of matrices and compressions of linear maps. This notation will allow us to write the incidence measures of the determinantal subspace processes that we want to define.

Let us consider first sub-matrices. Consider a $p\times q$ matrix $M$. For all $I\subset \{1,\ldots,p\}$ and all $J\subset \{1,\ldots,q\}$, we denote by $M^{I}_{J}$ the matrix obtained from $M$ by keeping only the rows indexed by the elements of $I$ and the columns indexed by the elements of $J$. For example, $M^{\scriptscriptstyle\{i\}}_{\scriptscriptstyle\{j\}}$ is the $1\times 1$ matrix containing the entry $M_{ij}$. We also use the notation $M^{I}$ instead of $M^{I}_{\scriptscriptstyle\{1,\ldots,q\}}$ and $M_{J}$ instead of $M^{\scriptscriptstyle\{1,\ldots,p\}}_{J}$.

We find it convenient to introduce the analogue of this notation for linear maps. Let $\op{a}:F\to E$ be a linear map between two inner product spaces. For all linear subspaces $Q\subset E$ and $R \subset F$, we denote by $\op{1}_{R}:R\to F$ the inclusion map, by $\op{1}^{Q}:E\to Q$ the orthogonal projection, and we set
\[\op{a}_{R}^{Q}=\op{1}^{Q} \circ a \circ \op{1}_{R} : R\to Q.\]

We also introduce the notation $\op{a}^{Q}=\op{a}^{Q}_{F}$ and $\op{a}_{R}=\op{a}^{E}_{R}$. For example, $(\id_{F})_{R}=\op{1}_{R}$ and $(\id_{E})^{Q}=\op{1}^{Q}$. Note also that $\op{a}_{R}^{Q}=(\op{a}_{R})^{Q}=(\op{a}^{Q})_{R}$.

Let us give a few more examples of uses of this notation which will occur later on in the text. Given two linear subspaces $Q$ and $R$ of an inner product space $E$, the endomorphism 
\[\op{1}_{Q}^{R}=(\id_{E})_{Q}^{R}:Q\to R\]
is the orthogonal projection of $Q$ on $R$. In particular, 
\[\id_{Q}=\op{1}^{Q}\op{1}_{Q}:Q\to Q\] 
is the identity of $Q$. On the other hand, 
\[\proj{Q}=\op{1}_{Q} \op{1}^{Q}:E\to E\] 
is the orthogonal projection on $Q$ seen as an endomorphism of $E$. Although it is not entirely consistent with the definitions of this section, we will keep using the notation $\proj{Q}$ for this projection.

\subsection{An invariant Cauchy--Binet formula} \label{sec:CB}
Let us discuss an extension of the classical Cauchy--Binet formula which will be an important tool in our computations.

Let $E$ and $F$ be two finite-dimensional vector spaces, of respective dimensions $d$ and $n$. Let $(e_{1},\ldots,e_{d})$ and $(f_{1},\ldots,f_{n})$ be two bases of $E$ and $F$ respectively. Let $\op{a}:F\to E$ and $\op{b}:E\to F$ be two linear maps. Let $A$ and $B$ be the matrices of $\op{a}$ and $\op{b}$ respectively with respect to the chosen bases of $E$ and $F$. The classical Cauchy--Binet formula computes the determinant of the endomorphism $\op{b}\circ \op{a}$ of $F$:
\[\det(\op{b}\op{a})=\sum_{\substack{J\subset \{1,\ldots,d\} \\ |J|=n}} \det(B_{J})\det (A^{J}).\]
The individual summands in the right-hand side of the Cauchy--Binet formula depend on the choice of the basis of $E$, but of course, their sum does not. From our point of view, this classical Cauchy--Binet formula is well adapted to the splitting in lines determined by the basis that we chose on $E$. We will now write a Cauchy--Binet formula that is adapted to an arbitrary splitting of $E$. 

In the next statement, we use the notation introduced in Section \ref{sec:minors}.

\begin{proposition}[Split Cauchy--Binet formula]\label{prop:CB+} Let $E$ and $F$ be two inner product spaces. Assume that $E$ is endowed with  an orthogonal splitting $\sigma=(E_{1}, \ldots, E_{s})$. Set $n=\dim F$. Let $\op{a}:F\to E$ and $\op{b}:E\to F$ be two linear maps. Then
\begin{equation}\label{CB}
\det(\op{b}\op{a})=\int_{\Gr_{n}(E,\sigma)} \det(\op{b}_{Q} \op{a}^{Q}) \d\nu_{n}^{E,\sigma}(Q)=\int_{\Gr_{n}(E,\sigma)} \det(\op{a}^{Q} \op{b}_{Q}) \d\nu_{n}^{E,\sigma}(Q).
\end{equation}
\end{proposition}

Let us emphasise that in the case of a splitting in lines, \eqref{CB} reduces to the classical Cauchy--Binet formula. 

Note also that in general, and by contrast with the classical case, neither $\det(\op{a}^{Q})$ nor $\det(\op{b}_{Q})$ is defined in isolation. This is why we give two versions of the Cauchy--Binet formula, corresponding to the two possible orders of multiplication of the compressed linear maps. In the first integral, the determinant is that of an endomorphism of $F$, whereas in the second integral, it is the determinant of an endomorphism of $Q$.

\begin{proof} The equality between the two integrals of \eqref{CB} is a consequence of the following elementary fact: given any two vector spaces $G$ and $H$ {\em with the same dimension} and any two linear maps $\op{g}:H\to G$ and $\op{h}:G\to H$, we have $\det(\op{g}\op{h})=\det(\op{h}\op{g})$.

In order to prove the first equality, let us start by rewriting the classical Cauchy--Binet formula in a slightly different way. Set $d=\dim E$ and choose an orthonormal basis $(e_{1},\ldots,e_{d})$ of $E$ that is adapted to the splitting of $E$ (recall that this means that it is obtained by concatenating orthonormal bases of $E_{1},\ldots,E_{s}$). For all $J\subset \{1,\ldots,d\}$, let us denote by $E_{J}$ the subspace $\Vect(e_{j} : j\in J)$ of~$E$. Then the classical Cauchy--Binet formula reads
\[\det(\op{b}\op{a})=\sum_{\substack{J\subset \{1,\ldots,d\} \\ |J|=n}} \det\left(\op{b}_{E_{J}}  \op{a}^{E_{J}}\right).\]
Since this formula holds for every basis of $E$ adapted to $\sigma$, we have, for all $u\in \U(E,\sigma)$,
\[\det(\op{b}\op{a})=\sum_{\substack{J\subset \{1,\ldots,d\} \\ |J|=n}} \det\left(\op{b}_{u(E_{J})}  \op{a}^{u(E_{J})}\right).\]
Integrating with respect to the normalised Haar measure on $\U(E,\sigma)$ and using \eqref{eq:intnunsigma}, we find the announced result.
\end{proof}

\subsection{A simple formula about determinants} We will also make repeated use of an equality that, in the case of a splitting in lines, reduces to the multilinearity of the determinant. 

\begin{proposition} \label{prop:multilin}
Let $(E,\sigma)$ be a split inner product space. Let $\op{a}$ and $\op{b}$ be linear endomorphisms of $E$. Let $Q\in \Gr(E,\sigma)$ be a linear subspace of $E$ adapted to $\sigma$. The following equalities hold:
\[\int_{\Gr(Q,\sigma_{Q})} \det(\op{a}\proj{R}+\op{b}\proj{R^{\perp}}) \d\nu^{Q,\sigma_{Q}}(R)=\det(\op{a}\proj{Q}+\op{b})\]
and
\[\int_{\Gr(Q,\sigma_{Q})} \det(\proj{R}\op{a}+\proj{R^{\perp}}\op{b}) \d\nu^{Q,\sigma_{Q}}(R)=\det(\proj{Q}\op{a}+\op{b}).\]
\end{proposition}

Let us emphasise that $R^{\perp}$ denotes the orthogonal of $R$ in $E$, not in $Q$. Also, in the special case where $Q=E$, the right-hand side reduces to $\det(\op{a}+\op{b})$.

\begin{proof} Let us choose an orthonormal basis $(e_{1},\ldots,e_{n})$ of $Q$ adapted to $\sigma_{Q}$ and let us complete it into an orthonormal basis $(e_{1},\ldots,e_{d})$ of $E$. For every subset $J$ of $\{1,\ldots,n\}$, the matrix of $\op{a}\proj{E_{J}}+\op{b}\proj{E_{J}^{\perp}}$ in this basis is obtained by replacing, in the matrix of $\op{b}$, the columns indexed by an element of $J$ by the corresponding columns of $\op{a}$. Summing over all subsets of $\{1,\ldots,n\}$ and using the multilinearity of the determinant with respect to the columns, we find
\[\sum_{J\subseteq\{1,\ldots,n\}} \det(\op{a}\proj{E_{J}}+\op{b}\proj{E_{J}^{\perp}})=\det((\op{a}+\op{b})\proj{Q}+\op{b}\proj{Q^{\perp}}).\]
This equality is still true if we replace the basis of $Q$ by its image by an arbitrary element of $\U(Q,\sigma_{Q})$. Therefore, seeing, for all $J\subseteq \{1,\ldots,n\}$ and all $u\in \Gr(Q,\sigma_{Q})$, the subspace $u(E_{J})$ of~$Q$ as a subspace of $E$, we have
\[\sum_{J\subseteq\{1,\ldots,n\}} \det(\op{a}\proj{u(E_{J})}+\op{b}\proj{u(E_{J})^{\perp}})=\det((\op{a}+\op{b})\proj{Q}+\op{b}\proj{Q^{\perp}}).\]
Integrating with respect to the normalised Haar measure on $\U(Q,\sigma_{Q})$ and using \eqref{eq:intnunsigma}, we find the first equality. The second is deduced from the first by adjunction. 
\end{proof}

%%%%%%%%%%%%%%%%%%%%%%%%%%%%%%%%%%%%
%%%%%%%%%%%%%%%%%%%%%%%%%%%%%%%%%%%%
%%%%%%%%%%%%%%%%%%%%%%%%%%%%%%%%%%%%

\section{Determinantal linear processes (DLP)}\label{sec:DLP}

This section defines and shows the existence of DLP, the properties of which will be analysed in Section \ref{sec:geometry-DLP}.

\subsection{Definition, existence and uniqueness of DLP}\label{sec:defDSP}
We are now able to give the main definition of this paper.

\begin{definition}[Determinantal linear process] \label{def:defDSP} Let $(E,\sigma)$ be a split inner product space. Let~$\op{k}$ be a linear endomorphism of $E$. Let $\mu$ be a Borel probability measure on $\Gr(E,\sigma)$. We say that $\mu$ is a {\em determinantal linear process on $(E,\sigma)$ with kernel $\op{k}$} if the incidence measure $\Inc{\mu}$ is absolutely continuous with respect to $\nu^{E,\sigma}$, with density given by
\begin{equation}\label{eq:incidenceDSP}
\forall Q \in \Gr(E,\sigma), \ \frac{\d\Inc{\mu}}{\d \nu^{E,\sigma}}(Q)=\det \op{k}^{Q}_{Q}. 
\end{equation}
\end{definition}

Let us emphasise that, by contrast with the situation of determinantal point processes, the kernel $\op{k}$ is not a matrix, but a linear operator on the space $E$, on which no preferred basis is chosen. Our first main result is the following.

\begin{theorem}\label{thm:existuniqueDSP} Let $(E,\sigma)$ be a split inner product space. Let $\op{k}$ be a self-adjoint linear endomorphism of $E$ such that $0\leq \op{k} \leq 1$. There exists a unique determinantal linear process on $(E,\sigma)$ with kernel $\op{k}$.
\end{theorem}

We will denote the unique determinantal linear process on a split inner product space $(E,\sigma)$ with kernel $\op{k}$ by $\mu_{\sigma,\op{k}}$ and will use the shorthand notation DLP for {\em determinantal linear process}.

The fact that there is at most one measure $\mu$ satisfying \eqref{eq:incidenceDSP} is a direct consequence of Proposition \ref{prop:incmeaschar}. Indeed, \eqref{eq:incidenceDSP} describes the incidence measure $\Inc{\mu}$ and Proposition \ref{prop:incmeaschar} allows us to recover $\mu$ from $\Inc{\mu}$. The reason why the existence of $\mu$ is not obvious is that it is not clear that the signed measure on $\Gr(E,\sigma)$ defined by the inversion formula \eqref{eq:mobius} is a probability measure.

We will prove the existence of a determinantal linear process in quite a direct way, by exhibiting its density with respect to the uniform measure on $\Gr(E,\sigma)$. 

\begin{proposition}\label{prop:exDLP}
 Let $(E,\sigma)$ be a split inner product space. Let $\op{k}$ be a kernel on $E$. The formula
\begin{equation}\label{eq:densitymuK}
\d\mu_{\sigma,\op{k}}(Q)=\det\big(\op{k}\proj{Q}+(1-\op{k})\proj{Q^{\perp}}\big)\d\nu^{E,\sigma}(Q)
\end{equation}
defines a probability measure on $\Gr(E,\sigma)$, the incidence measure of which is given by \eqref{eq:incidenceDSP}.
\end{proposition}

Note that the density of $\mu_{\sigma,\op{k}}$, being a real number and the determinant of an operator, is also the determinant of the adjoint operator, and thus equal to $\det\big(\proj{Q}\op{k}+\proj{Q^{\perp}}(1-\op{k})\big)$.

Note also that a less symmetric, but Hermitian version of the density of $\mu_{\sigma,\op{k}}$ is given by 
\begin{equation}\label{eq:densitymuKHerm}
\d\mu_{\sigma,\op{k}}(Q)=(-1)^{\dim Q^{\perp}}\det(-\proj{Q^{\perp}}+\op{k})\d\nu^{E,\sigma}(Q).
\end{equation}
Indeed, with the notation of the proof below, the determinant in this expression is that of the matrix
\[\det\big(-\proj{Q^{\perp}}+\op{k}\big)=\det \begin{pmatrix}A & B \\ B^{*} & D-I\end{pmatrix}.\]

Let us finally emphasise that the assumption that $\op{k}$ is self-adjoint is only used in the proof of the positivity of $\mu_{\sigma,\op{k}}$. This could be useful for the theory of determinantal measures with non-symmetric kernels, which form an interesting class which we otherwise leave aside in this paper. 

\begin{proof} We need to prove three things: that the density appearing in \eqref{eq:densitymuK} is non-negative, that the integral of this density is $1$, and that the incidence measure of $\mu_{\sigma,\op{k}}$ is indeed given by \eqref{eq:incidenceDSP}.

1. Let $Q$ be a linear subspace of $E$. Let us choose an orthonormal basis  of $E$ adapted to the splitting $(Q,Q^{\perp})$ and write the matrix of $\op{k}$ in this basis as
\begin{equation}\label{eq:unitmat}
\begin{blockarray}{ccc}
  &{\scriptstyle Q} & {\scriptstyle Q^{\perp}} \\
\begin{block}{c(cc)}
    {\scriptstyle Q}  & A & B  \\
  {\scriptstyle Q^{\perp}} & B^{*} & D  \\
\end{block}
\end{blockarray}
\end{equation}
so that 
\[\det\big(\op{k}\proj{Q}+(1-\op{k})\proj{Q^{\perp}}\big)=\det \begin{pmatrix}A & -B \\ B^{*} & I-D\end{pmatrix}.\]
Let us assume for a moment that $\op{k}$ satisfies not only $0\leq \op{k} \leq 1$, but the stronger assumption $0<\op{k}<1$. Then $\op{k}$ is positive definite, and so is the principal sub-matrix $A$. In particular, $A$ is invertible. The classical trick of blockwise elimination, also called the Schur complement formula\footnote{\label{foot:Schur}
For the convenience of the reader, let us recall the one sentence proof of this formula, which we will use again later: multiplying on the left any $2\times 2$ block matrix $M=\begin{pmatrix} A & B \\ C & D \end{pmatrix}$ in which $A$ is invertible  by the matrix $\begin{pmatrix} I & 0 \\ -CA^{-1} & I \end{pmatrix}$ and taking determinants yields $\det (M)=\det (A) \det (D-CA^{-1}B)$.
}, yields
\[\det\big(\op{k}\proj{Q}+(1-\op{k})\proj{Q^{\perp}}\big)=\det(A) \det(I-D+B^{*}A^{-1}B).\]
We have already noted that $A$ is positive-definite, so that $\det(A)>0$. The matrix $I-D$, which is a principal sub-matrix of the matrix of $1-\op{k}$ in our basis of $E$, is also positive definite. Finally, $B^{*}A^{-1}B$ is self-adjoint and non-negative. The sum $(I-D)+B^{*}A^{-1}B$ is thus positive definite, and has positive determinant. Thus, 
\[\det\big(\op{k}\proj{Q}+(1-\op{k})\proj{Q^{\perp}}\big)> 0.\]
To relax the extra assumption that we made on $\op{k}$, it suffices to observe that for every kernel $\op{k}$ and all $\epsilon\in (0,\frac{1}{2})$, the operator $\op{k}_{\epsilon}=(1-2\epsilon)\op{k}+\epsilon$ is still a kernel and satisfies $\epsilon \leq \op{k}\leq 1-\epsilon$. Our argument can thus be applied to $\op{k}_{\epsilon}$ and we conclude by letting $\epsilon$ tend to $0$.

2. The total mass of the measure defined by the right-hand side of \eqref{eq:densitymuK} is easily computed thanks to Proposition \ref{prop:multilin}. We find
\[\int_{\Gr(E,\sigma)} \det\big(\op{k}\proj{Q}+(1-\op{k})\proj{Q^{\perp}}\big)\d\nu^{E,\sigma}(Q)
=\det(\op{k}+1-\op{k})=1.\]

3. The computation of the incidence measure of the probability measure $\mu_{\sigma,\op{k}}$ also relies on Proposition \ref{prop:multilin}. Let us choose a continuous test function $f$ on $\Gr(E,\sigma)$. According to the definition \eqref{eq:definc} of the incidence measure, we need to compute
\begin{equation}\label{eq:tpinc}
\int_{\Gr(E,\sigma)}\bigg(\int_{\Gr(Q,\sigma_{Q})} f(R) \d\nu^{Q,\sigma_{Q}}(R)\bigg)\det\big(\op{k}\proj{Q}+(1-\op{k})\proj{Q^{\perp}}\big)\d\nu^{E,\sigma}(Q).
\end{equation}
Using Lemma \ref{lem:orthpi}, this is equal to
\[\int_{\Gr(E,\sigma)}\bigg(\int_{\Gr(Q,\sigma_{Q})} \det\big(\op{k}\proj{R^{\perp}}+(1-\op{k})\proj{R}\big) \d\nu^{Q,\sigma_{Q}}(R)\bigg) f(Q^{\perp})\d\nu^{E,\sigma}(Q)\]
and by Proposition \ref{prop:multilin}, the integral between brackets is equal to 
\[\det((1-\op{k})\proj{Q}+\op{k})=\det(\proj{Q}+\op{k}\proj{Q^{\perp}})\]
which, writing the matrix of $\op{k}$ in a basis of $E$ adapted to the splitting $(Q,Q^{\perp})$, is easily seen to be equal to $\det \op{k}^{Q^{\perp}}_{Q^{\perp}}$. Thanks to \eqref{eq:orthnu}, we find that \eqref{eq:tpinc} is equal to
\[\int_{\Gr(E,\sigma)} f(Q) \det \op{k}_{Q}^{Q} \d\nu^{E,\sigma}(Q),\]
which concludes the proof.
\end{proof}

As everything we do in this paper, this proof applies to the case of a determinantal point process on a finite space as well. Our proof amounts to showing that the DPP with kernel $K$ has, with self-explanatory notation, the distribution
\[\forall I\subseteq \{1,\ldots,d\}, \ \ \P(\X=I)=\det(K\proj{I}+(I-K)\proj{I^{c}}).\]
This is however not the usual way in which the construction of determinantal processes is done. The classical approach consists in studying first the case where $K$ is a projector, and then showing that the general case can be realised as a mixture of projector cases. This approach is also possible, and instructive, in the more general setting that we explore in this paper. Although this is logically not necessary, we devote the next two sections to a discussion of projection DLP and of the way in which an arbitrary DLP can be obtained as a mixture of projection DLP.

\subsection{Projection DLP and angle between linear subspaces}

We want to give special attention to the case where $\op{k}$ is the orthogonal projection on a linear subspace $H$ of $E$. In this case, the density of the distribution of a DLP with kernel $\op{k}$ can be described in a nice geometric way in terms of a notion of angle between two linear subspaces, or more precisely of the square cosine of this angle. 
 
\begin{definition} Let $E$ be an inner product space. Let $F$ and $G$ be two linear subspaces of $E$. We define the {\em square of the cosine of the angle} between $F$ and $G$ as
\begin{equation}\label{eq:defcos}
\cos^{2}(F,G)=\det(\op{1}_{G}^{F}\op{1}_{F}^{G}).
\end{equation}
\end{definition}

This definition is illustrated by Figure \ref{fig:angle} below. Let us emphasise that this definition is not symmetric in $F$ and $G$: for example, if $F$ is a proper subspace of $E$, then $\cos^{2}(F,E)=1$ but $\cos^{2}(E,F)=0$.

\begin{figure}[h!]
\begin{center}
\includegraphics{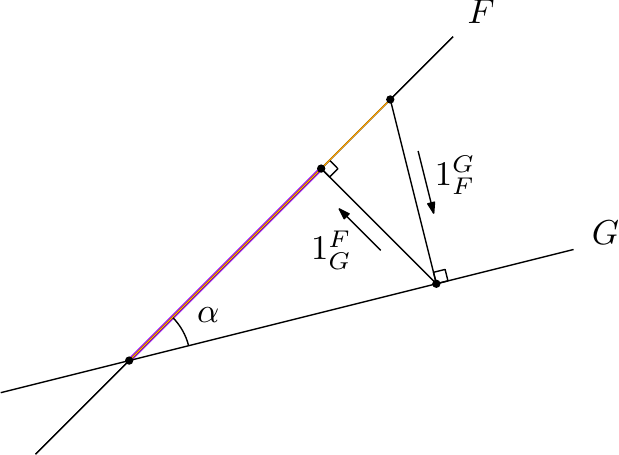}
\caption{\label{fig:angle} The square of the cosine of the angle between two lines, here $\cos^{2}\alpha=\cos^{2}(F,G)$, can be computed by two successive orthogonal projections.}
\end{center}
\end{figure}

A geometric understanding of $\cos^{2}(F,G)$ can be based on the observation that it is the product of the squares of the singular values of the orthogonal projection $\op{1}_{F}^{G}:F\to G$. Writing $m=\dim F$, these singular values are the half-lengths of the $m$ longest principal axes of the ellipsoid in $G$ that one obtains by projecting the unit ball of $F$. Denoting by $B_{F}$ this unit ball, and by ${\rm vol}_{m}$ the $m$-dimensional volume (that is, to be precise, the $m$-dimensional Hausdorff measure), we have
\[\cos^{2}(F,G)=\frac{{\rm vol}_{m}(\op{1}_{F}^{G}(B_{F}))^{2}}{{\rm vol}_{m}(B_{F})^{2}}.\]

The number $\cos^{2}(F,G)$ is also the product of the squares of the cosines of the principal angles between $F$ and $G$ as defined by Jordan \cite{Jordan}.

Let us prove some basic properties of this function.

\begin{proposition}\label{prop:cosbasic} Let $F$ and $G$ be two linear subspaces of $E$.

1. $\cos^{2}(F,G) \in [0,1]$. 

2. $\cos^{2}(F,G)=0 \Leftrightarrow F \cap G^{\perp} \neq \{0\}$. In particular, if $\dim F>\dim G$, then $\cos^{2}(F,G)=0$.

3. $\cos^{2}(F,G)=1 \Leftrightarrow F\subset G$.

4.  $\cos^{2}(F,G)=\cos^{2}(F,\op{1}_{F}^{G}(F))$.

5. If $\dim F=\dim G$, then $\cos^{2}(F,G)=\cos^{2}(G,F)$.

6. Let $u\in \U(E)$ be an isometry of $E$. Then $\cos^{2}(u(F),u(G))=\cos^{2}(F,G)$.

\end{proposition}

\begin{proof} 1. The projections $\op{1}_{F}^{G}$ and $\op{1}_{G}^{F}$ are each other's adjoint, so that $\op{1}_{G}^{F}\op{1}_{F}^{G}$ is a non-negative self-adjoint operator. Moreover, since any orthogonal projection is $1$-Lipschitz continuous, the linear map $\op{1}_{G}^{F}\op{1}_{F}^{G}$ is $1$-Lipschitz continuous on $F$, and its eigenvalues are bounded by $1$. Thus, $\cos^{2}(F,G)\in [0,1]$.

2. The kernel of $\op{1}_{G}^{F}\op{1}_{F}^{G}=(\op{1}_{F}^{G})^{*}\op{1}_{F}^{G}$ is equal to the kernel of $\op{1}_{F}^{G}$, that is, to $F\cap G^{\perp}$. The second assertion follows from the inequality $\dim(F\cap G^{\perp})\geq \dim F-\dim G$.

3. If $F\subset G$, then $\op{1}_{G}^{F}\op{1}_{F}^{G}=\id_{F}$ and $\cos^{2}(F,G)=1$. Assume conversely that $\cos^{2}(F,G)=1$. Since $\op{1}_{G}^{F}\op{1}_{F}^{G}$ is a self-adjoint operator on $F$ with eigenvalues between $0$ and $1$, the only way it can have determinant $1$ is by being the identity of $F$. Now  for all $x\in F$, the inequalities $\|\op{1}_{G}^{F}\op{1}_{F}^{G}x\|\leq \|\op{1}_{F}^{G}x\|\leq \|x\|$ combined with the equality $\op{1}_{G}^{F}\op{1}_{F}^{G}x=x$ imply $\|\op{1}_{F}^{G}x\|=\|x\|$ and finally $\op{1}_{F}^{G}x=x$, so that $F\subset G$. 

4. Let $G'=\op{1}_{F}^{G}(F)$ be the image of the orthogonal projection of $F$ on $G$. We have $\op{1}_{F}^{G}=\op{1}_{G'}^{G}\op{1}_{F}^{G'}$, so that 
$\cos^{2}(F,G)=\det(\op{1}_{G}^{F}\op{1}_{G'}^{G}\op{1}_{F}^{G'})=\det(\op{1}_{G'}^{F}\op{1}_{F}^{G'})=\cos^{2}(F,G')$.

5. If $\dim F=\dim G$, then $\cos^{2}(F,G)=\det((\op{1}_{F}^{G})^{*}\op{1}_{F}^{G})=\det(\op{1}_{F}^{G}(\op{1}_{F}^{G})^{*})=\cos^{2}(G,F)$. 

6. This equality is obvious from a structural point of view, since $\cos^{2}(F,G)$ is defined using only the inner product structure of $E$. It is nevertheless possible to compute 
\[\op{1}_{u(G)}^{u(F)}\op{1}_{u(F)}^{u(G)}=u^{u(F)}_{F}\op{1}_{G}^{F}(u^{-1})^{G}_{u(G)} u_{G}^{u(G)} \op{1}_{F}^{G}(u^{-1})_{u(F)}^{F}\]
and to draw the desired conclusion.
\end{proof}

It follows in particular from the previous proposition that the computation of $\cos^{2}(F,G)$ can always be reduced to the case where $\dim F=\dim G$. Indeed, $\cos^{2}(F,G)\neq 0$ if and only if $\op{1}_{F}^{G}$ is injective on $F$ and in this case, $\cos^{2}(F,G)=\cos^{2}(F,\op{1}_{F}^{G}(F))$.

The next proposition provides us with a matricial understanding of the number $\cos^{2}(F,G)$. 

\begin{proposition}\label{prop:matrixcos} Let $F$ and $G$ be two linear subspaces of $E$.

1. Let $(f_{1},\ldots,f_{m})$ and $(g_{1},\ldots,g_{n})$ be orthonormal bases of $F$ and $G$. Set $A=(\langle g_{i},f_{j}\rangle)_{\substack{ i=1,\ldots, n \\  j=1,\ldots, m}}$. Then 
\[\cos^{2}(F,G)=\det(A^{*}A).\]

2. Let $(f_{1},\ldots,f_{m},e_{m+1},\ldots,e_{d})$ be an orthonormal basis of $E$ obtained by completing an orthonormal basis of $F$. Let $\Pi^{G}$ be the matrix in this basis of the orthogonal projection on $G$. Let ${\Pi^{G}}^{1 \cdots m}_{1\cdots m}$ be the top left $m\times m$ sub-matrix of $\Pi^{G}$. Then
\[\cos^{2}(F,G)=\det {\Pi^{G}}^{1 \cdots m}_{1\cdots m}.\]

3. Let $\proj{G}:E\to E$ be the orthogonal projection on $G$. Then 
\[\cos^{2}(F,G)=\det ({\proj{G}})^{F}_{F}.\]
\end{proposition}

\begin{proof} The matrix $A$ is the matrix, with respect to the chosen bases of $F$ and $G$, of the map~$\op{1}_{F}^{G}$. The first assertion is thus the direct translation in matricial language of the definition of $\cos^{2}(F,G)$. The second assertion follows from the first and the observation that the top left $m\times m$ sub-matrix of $\Pi^{G}$ is equal to $A^{*}A$. The third assertion is a reformulation, without any explicit reference to bases, of the second.
\end{proof}

The link between the cosine of the angle of two linear subpaces of $E$ and the density of the distribution of a projection DLP is given by the following proposition.

\begin{proposition}\label{prop:cosdet} Let $F$ and $G$ be two linear subspaces of $E$. Then
\[\det(\proj{F}\proj{G}+\proj{F^{\perp}}\proj{G^{\perp}})=\left\{\!\begin{array}{ll} \cos^{2}(F,G) & \text{if } \dim F = \dim G,\\ 0 & \text{otherwise.}\end{array}\right.\]
\end{proposition}

\begin{proof} Let us start by proving that $\dim F$ and $\dim G$ must be equal for the determinant not to be zero. For this, let us assume that the determinant is not zero. Then $\proj{F}\proj{G}+\proj{F^{\perp}}\proj{G^{\perp}}$ is onto, so that the range of $\proj{F}\proj{G}$ is $F$ and the range of $\proj{F^{\perp}}\proj{G^{\perp}}$ is $F^{\perp}$. This implies on one hand that $\dim G\geq \dim F$ and on the other hand that $\dim G^{\perp}\geq \dim F^{\perp}$.

Let us now prove the equality when $\dim F=\dim G=n$. Let us choose an orthonormal basis $(e_{1},\ldots,e_{d})$ of $E$ such that $(e_{1},\ldots,e_{n})$ is a basis of $F$. Let us write the matrix of $\proj{G}$ in this basis and decompose it in blocks according to the splitting $E=F\oplus F^{\perp}$:
\[\proj{G}=\begin{blockarray}{ccc}
  &{\scriptstyle F} & {\scriptstyle F^{\perp}} \\
\begin{block}{c(cc)}
    {\scriptstyle F}  & A & B  \\
  {\scriptstyle F^{\perp}} & C & D  \\
\end{block}
\end{blockarray}\,\; .
\]
Then we know from Proposition \ref{prop:matrixcos} that $\cos^{2}(F,G)=\det A$. 

Let us assume first that $\det A=0$. In this case, $\cos^{2}(F,G)=0$, so that, by Proposition \ref{prop:cosbasic}, $F$ and $G^{\perp}$ have a non-trivial intersection. This forbids the range of $\proj{F^{\perp}}\proj{G^{\perp}}$ from being equal to $F^{\perp}$, and forces the determinant to be zero: the equality also holds in this case.

There remains to treat the case where $A$ is invertible. Our assumption that $\dim F=\dim G$ implies that $A$ has full rank in the matrix written above, so that there exists a $n\times (d-n)$ matrix~$V$ such that $B=AV$, $C=V^{*}A$ and $D=CV=V^{*}AV$, so that
\begin{equation}\label{eq:PPPP}
\proj{F}\proj{G}+\proj{F^{\perp}}\proj{G^{\perp}}=\begin{pmatrix} A & AV \\ -V^{*}A & 1-V^{*}AV \end{pmatrix}.
\end{equation}
An application of the Schur complement formula (see Footnote \ref{foot:Schur}) shows that the determinant of this matrix is equal to $\det(A)$, as expected.
\end{proof}

An immediate consequence of this proposition is the following expression of the distribution of a projection DLP. 

\begin{proposition} Let $(E,\sigma)$ be a split inner product space. Let $H$ be a linear subspace of $E$ of dimension $n$. Then the following equality of probability measures holds on $\Gr(E,\sigma)$:
\begin{equation}\label{eq:densityDLPproj}
\d\mu_{\sigma,\proj{H}}(Q)=\cos^{2}(Q,H) \d\nu^{E,\sigma}_{n}(Q).
\end{equation}
\end{proposition}

\subsection{General DLP as mixtures of projection DLP}\label{sec:mixture}
Just as in the classical theory of determinantal point processes, a DLP with a general kernel can be constructed as a mixture of projection DLP, and our next task is to understand which particular mixture.
For this, we will associate to each kernel $\op{k}$ on $E$ a probability measure (indeed several probability measures in general) on the full Grassmannian of $E$, such that the projection DLP associated to a random subspace of $E$ chosen according to this probability measure is a DLP with kernel $\op{k}$.

Let, as always, $(E,\sigma)$ be a split inner product space. Let $\op{k}$ be a kernel on $E$. Let us say that a splitting $\tau=(E_{1},\ldots,E_{r})$ of $E$ is {\em adapted} to $\op{k}$ if $\op{k}$ acts as a scalar on each space $E_{1},\ldots,E_{r}$: 
\[\forall j\in\{1,\ldots,r\},\, \exists \lambda_{j} : \ \op{k}_{|E_{j}}=\lambda_{j}\id_{E_{j}}.\]
For instance, the splitting of $E$ by the eigenspaces of $\op{k}$ is adapted to $\op{k}$, but any finer splitting of~$E$ is also adapted to $\op{k}$. There is more than one splitting adapted to $\op{k}$ (up to reordering) if and only if $\op{k}$ has at least one multiple eigenvalue.

Let $\tau=(E_{1},\ldots,E_{r})$ be a splitting of $E$ adapted to $\op{k}$. It turns out that the DLP with kernel~$\op{k}$ on $(E,\tau)$ is a random subspace of $E$ which will solve our problem. Fortunately, this random subspace of $E$ is easily described. Indeed, for all tuples $\ul{n}=(n_{1},\ldots,n_{r})$, the density of the measure $\mu_{\tau,\op{k}}$ is constant on $\Gr_{\ul{n}}(E,\tau)$: setting $(d_{1},\ldots,d_{r})=(\dim E_{1},\ldots,\dim E_{r})$, we have
\[\forall Q\in \Gr_{\ul{n}}(E,\tau), \ \ \det\big(\op{k}\proj{Q}+(1-\op{k})\proj{Q^{\perp}}\big)=\prod_{i=1}^{r}\lambda_{i}^{n_{i}}(1-\lambda_{i})^{d_{i}-n_{i}}.\]
Accordingly, sampling $\eta_{\op{k}}^{E,\tau}$ amounts to choosing, in each space $E_{i}$ of the splitting $\tau$, on which $\op{k}$ acts as the scalar $\lambda_{i}\in [0,1]$, a uniform linear subspace with dimension distributed according to the binomial distribution of parameters $\lambda_{i}$ and $d_{i}$.

\begin{example}\rm\label{ex:HKPV}
The measure $\mu_{\tau,\op{k}}$ is a variant of a measure that is classical in the theory of determinantal point processes and is for instance implicitly defined in \cite[Theorem 7]{HKPV}. In this paper, the authors choose first an orthogonal basis of eigenvectors and then build a random subspace of $E$ as the linear span of the subset of this basis obtained by keeping each vector with a probability equal to the corresponding eigenvalue, independently of the others. In our language, they consider the case where $\tau$ is a splitting in lines of $E$ adapted to $\op{k}$.

If $\op{k}=\frac{1}{2}\id_{E}$ for example, the authors of \cite{HKPV} consider the uniform measure on the set of the $2^{d}$ linear subspaces generated by all possible subsets of a fixed orthogonal basis of~$E$, whereas we allow for example, by taking $\tau=(E)$, a uniform linear subspace of binomial dimension $\mathcal{B}(n,\frac{1}{2})$, a measure which has full support in $\Gr(E)$. 
\end{example}

The main result of this section is the following.

\begin{proposition} Let $(E,\sigma)$ be a split inner product space. Let $\op{k}$ be a kernel on $E$. Let $\tau$ be a splitting of $E$ adapted to $\op{k}$. Then the following equality of probability measures on $\Gr(E,\sigma)$ holds:
\begin{equation}\label{eq:mixture}
\mu_{\sigma,\op{k}}=\int_{\Gr(E,\tau)} \mu_{\sigma,\proj{H}} \d\mu_{\tau,\op{k}}(H).
\end{equation}
\end{proposition}

In words: sampling a subspace $H$ of $E$ under the measure $\mu_{\tau,\op{k}}$ and then a second subspace under the measure $\mu_{\sigma,\proj{H}}$ yields a random linear subspace of $E$ distributed as a DLP on $(E,\sigma)$ with kernel $\op{k}$. Let us emphasise again that the splitting $\tau$ can be chosen arbitrarily among all splittings adapted to $\op{k}$.

\begin{proof} We prove the equality of the continuous densities of both sides of \eqref{eq:mixture} with respect to $\nu^{E,\sigma}$ at a point $Q$ of $\Gr(E,\sigma)$. On the left-hand side, this density is equal, by \eqref{eq:densitymuK}, to $\det\big(\op{k}\proj{Q}+(1-\op{k})\proj{Q^{\perp}}\big)$. On the right-hand side, it is equal to
\[\int_{\Gr(E,\tau)} \det\big(\op{k}\proj{H}+(1-\op{k})\proj{H^{\perp}}\big) \det\big(\proj{H}\proj{Q}+\proj{H^{\perp}}\proj{Q^{\perp}}\big) \d\nu^{E,\tau}(H).\]
Multiplying the two determinants inside the integral yields $\det(\op{k}\proj{H}\proj{Q}+(1-\op{k})\proj{H^{\perp}}\proj{Q^{\perp}})$. The fact that the splitting $\tau$ is adapted to $\op{k}$ implies that $\op{k}$ commutes to $\proj{H}$ for every $H\in \Gr(E,\tau)$. Thus, we are computing the integral
\[\int_{\Gr(E,\tau)} \det(\proj{H}\op{k}\proj{Q}+\proj{H^{\perp}}(1-\op{k})\proj{Q^{\perp}}) \d\nu^{E,\tau}(H)\]
which, according to Proposition \ref{prop:multilin}, is precisely equal to $\det\big(\op{k}\proj{Q}+(1-\op{k})\proj{Q^{\perp}})$. 
\end{proof} 

\subsection{Sampling of DLP}\label{sec:algosample}

In this section, we give a formal statement and proof of the sampling procedure for DLP explained in Section \ref{sec:detmeasGr}, see \eqref{eq:sampling}. In fact, we prove a slightly more general result. Let us start by establishing a property of uniform measures on Grassmannians. 

Recall the partial order that we introduced, at the very beginning of this study, on the set of all splittings of $E$ (see Section \ref{sec:splittings}): a splitting $\delta$ is finer than the splitting $\sigma$ if each element of~$\delta$ is contained in an element of $\sigma$. Recall also that we denote simply by $\dd u$ the normalised Haar measure on the group $\U(E,\sigma)$.

\begin{proposition} \label{prop:averageunif} Let $E$ be an inner product space. Let $\delta$ and $\sigma$ be two splittings of $E$. Assume that $\delta$ is finer than~$\sigma$. Then for every isometry $u\in \U(E,\sigma)$, the splitting $u(\delta)$ is still finer than~$\sigma$, and we have
\begin{equation}\label{eq:averagesplittings}
\int_{\U(E,\sigma)} \nu^{E,u(\delta)}\d u=\nu^{E,\sigma},
\end{equation}
where for every $u\in \U(E,\sigma)$, the space $\Gr(E,u(\delta))$ is seen as a subspace of $\Gr(E,\sigma)$, and $\nu^{E,u(\delta)}$ as a measure on $\Gr(E,\sigma)$.
\end{proposition}

\begin{proof} We use the concrete description of the uniform measures on Grassmannians given by \eqref{eq:intnunsigma}. Let $(e_{1},\ldots,e_{d})$ be an orthonormal basis of $E$ adapted to $\delta$. Then for every $u\in \U(E,\sigma)$, the basis $(u(e_{1}),\ldots,u(e_{d}))$ is adapted to $u(\delta)$, and to $\sigma$. Thus, the left-hand side of \eqref{eq:averagesplittings} is the measure
\[\int_{\U(E,\sigma)}\bigg(\int_{\U(E,u(\delta))} \sum_{J\subseteq \{1,\ldots,d\}} \delta_{v(u(E_{J}))} \d v\bigg) \d u.\]
For all $u\in \U(E,\sigma)$, we have the equality $\U(E,u(\delta))=u\U(E,\delta)u^{-1}$ of subgroups of $\U(E,\sigma)$. Thus, the left-hand side of \eqref{eq:averagesplittings} is equal to
\[\int_{\U(E,\delta)}\bigg(\int_{\U(E,\sigma)} \sum_{J\subseteq \{1,\ldots,d\}} \delta_{uw(E_{J})} \d u\bigg) \d w.\]
The invariance of the Haar measure by translation on the right implies that the integral between the brackets does not depend on $w$ and we are left with the expression of the right-hand side 
of~\eqref{eq:averagesplittings} given by~\eqref{eq:intnunsigma}.
\end{proof}

The result that will lead us to the sampling algorithm is the following.

\begin{proposition}\label{prop:averaging} Let $E$ be an inner product space. Let $\delta$ and $\sigma$ be two splittings of $E$. Assume that $\delta$ is finer that $\sigma$. Let $\op{k}$ be a kernel on $E$. Then the following equality of measures holds on $\Gr(E,\sigma)$:
\begin{equation}\label{eq:averagemuK}
\mu_{\sigma,\op{k}}=\int_{\U(E,\sigma)} \mu_{u(\delta),\op{k}} \d u.
\end{equation}
\end{proposition}

In this statement, as in the statement of Proposition \ref{prop:averageunif}, we see, for all $u\in \U(E,\sigma)$, the measure $\mu_{u(\delta),\op{k}}$ as a measure on $\Gr(E,\sigma)$, through the natural inclusion $\Gr(E,u(\delta))\subseteq \Gr(E,\sigma)$.

\begin{proof} Let $f$ be a continuous test function on $\Gr(E,\sigma)$. We compute the integral of $f$ with respect to the right-hand side of \eqref{eq:averagemuK}, using the definition \eqref{eq:densitymuK} of the DLP. This integral is equal to
\begin{align*}
&\int_{\U(E,\sigma)} \int_{\Gr(E,u(\delta))} f(Q) \det(\op{k}\proj{Q}+(1-\op{k})\proj{Q^{\perp}}) \d\nu^{E,u(\delta)}(Q)\\
&\hspace{3cm} = \int_{\Gr(E,\sigma)}\Big(Q\mapsto  f(Q) \det(\op{k}\proj{Q}+(1-\op{k})\proj{Q^{\perp}})\Big)  \d \int_{\U(E,\sigma)} \nu^{E,u(\delta)}\d u.
\end{align*}
According to Proposition \ref{prop:averageunif}, the measure against which the last integral is computed is nothing but $\nu^{E,\sigma}$. Hence, the integral is equal to
\[\int_{\Gr(E,\sigma)} f(Q) \det(\op{k}\proj{Q}+(1-\op{k})\proj{Q^{\perp}})\d\nu^{E,\sigma}(Q)=\int_{\Gr(E,\sigma)} f \d\mu_{\sigma,\op{k}},\]
and the result is proved.
\end{proof}

Let us note that in this proof, the particular form of the density plays no role whatsoever. This relates to the fact that the distributions of DLP with the same kernel but relative to different splittings have the same density with respect to different reference measures (recall Proposition~\ref{prop:exDLP}).

In the case where $\delta$ is a splitting in lines, we arrive at the sampling procedure \eqref{eq:sampling} explained in Section~\ref{sec:dpp-geom}.

\begin{proposition}[Sampling of DLP]\label{prop:sampling} Let $(E,\sigma)$ be a split inner product space. Let $\op{k}$ be a kernel on $E$. In each element of $\sigma$, choose uniformly at random an orthonormal basis. Agregate these bases to produce a (random) orthonormal basis $(e_{1},\ldots,e_{d})$ of $E$. Let $K$ be the matrix of $\op{k}$ in this basis. Sample the determinantal point process with kernel $K$ on $\{1,\ldots,d\}$. This produces a (random) subset $X$ of $\{1,\ldots,d\}$. Then the random linear subspace $\Vect(e_{i}:i\in X)$ of $E$ is distributed according to $\mu_{\sigma,\op{k}}$. 
\end{proposition}

For the sake of completeness, let us briefly give an algorithm for sampling a DPP, denoted by~$\X$, on $\{1,\ldots,d\}$ with kernel $K$. This recursive algorithm extends \cite[Algorithm 18]{HKPV} to the case where the kernel is not necessarily a projection:\footnote{In the literature, the preferred sampling method usually seems to be to use the fact that a DPP for a general kernel is a mixture of projection DPP, see Example \ref{ex:HKPV}, but we prefer this more direct way to proceed.}
\begin{itemize}
\item write $K=\begin{pmatrix} K_{11} & R \\ C & K'\end{pmatrix}$ as a $(1,(d-1))\times (1,(d-1))$ block matrix,
\item add $1$ to $\X$ with probability $K_{11}$,
\item if $1$ has been picked (so that $K_{11}>0$), sample a DPP $\X'$ on $\{2,\ldots,d\}$ with kernel $K'-K_{11}^{-1}CR$ and set $\X=\{1\}\cup \X'$,
\item if $1$ has not been picked (so that $K_{11}<1$), sample a DPP $\X'$ on $\{2,\ldots,d\}$ with kernel 
\[I_{d-1}-((I_{d-1}-K')-(1-K_{11})^{-1}CR)=K'+(1-K_{11})^{-1} CR\]
and set $\X=\X'$.
\end{itemize}

This algorithm therefore only requires sampling Bernoulli random variables and performing linear algebra operations. The fact that the above operations yield the right probabilities is a consequence of the Schur complement formula again (recall Footnote~\ref{foot:Schur}).

The task of picking a uniform random orthonormal basis can be done easily provided one knows how to sample a Gaussian distribution and perform a Gram--Schmidt orthonormalisation procedure, see e.g. \cite{Diaconis}.

This provides a concrete way to ask a computer to sample a DLP, and this is how Figure \ref{fig:exQSF} was sampled for instance.

%%%%%%%%%%%%%%%%%%%%%%%%%%%%%%%%%%%%
%%%%%%%%%%%%%%%%%%%%%%%%%%%%%%%%%%%%
%%%%%%%%%%%%%%%%%%%%%%%%%%%%%%%%%%%%

\section{Geometry of DLP}\label{sec:geometry-DLP}

We will now describe some geometric properties of DLP. Most of these properties will parallel classical properties of DPP. Nevertheless, the content of Section \ref{sec:Support} is more specific to DLP, and the main result of Section \ref{sec:meanproj} (Theorem~\ref{prop:meanprojection}), although true and meaningful for DPP, does not seem to be as well known in this case as it would deserve to be, and it was new to us.

 In all this section, we fix a split inner product space $(E,\sigma)$, with $\sigma=(E_{1},\ldots,E_{s})$, and a kernel $\op{k}$ on $E$. We let $\Q$ be a random linear subspace of $E$ distributed according to $\mu_{\sigma,\op{k}}$.

\subsection{Basic properties}

Let us start with a property of equivariance of DLP. 

\begin{proposition}[Equivariance] For all $u\in \U(E)$, the random linear subspace $u(\Q)$ of $E$ is a DLP on $(E,u(\sigma))$ with kernel $u\op{k}u^{-1}$.
\end{proposition}

In terms of measures, and considering the induced map $u:\Gr(E,\sigma)\to \Gr(E,u(\sigma))$, this property reads
\begin{equation}\label{eq:pushQu}
u_{*}\mu_{\sigma,\op{k}}=\mu_{u(\sigma),u\op{k}u^{-1}}.
\end{equation}

\begin{proof} The result follows at once from the fact that the image of the measure $\nu^{E,\sigma}$ by $u$ is the measure $\nu^{E,u(\sigma)}$ and the equality
\[\det\big(\op{k}\proj{u^{-1}(Q)}+(1-\op{k})\proj{u^{-1}(Q)^{\perp}}\big)=\det\big(u\op{k}u^{-1}\proj{Q}+(1-u\op{k}u^{-1})\proj{Q^{\perp}}\big),\]
which itself follows from $\proj{u^{-1}(Q)}=u^{-1}\proj{Q}u$ and $\det(u)=1$.
\end{proof}

Let us stress that this property loses its substance in the case of DPP. Indeed, one works in this case with a fixed splitting in lines $\sigma$ of $E$, and the group $\U(E,\sigma)$ reduces to the product of the isometry groups of the lines, which acts trivially on the finite set $\Gr(E,\sigma)$.

The next property has to do with the orthocomplement of $\Q$, and extends the fact that the complement of a finite DPP is still a DPP.

\begin{proposition}[Orthocomplement]The random linear subspace $\Q^{\perp}$ of $E$ is a DLP on $(E,\sigma)$ with kernel $1-\op{k}$.
\end{proposition}

\begin{proof} This follows immediately from the invariance of the measure $\nu^{E,\sigma}$under the map $Q\mapsto Q^{\perp}$ (see \eqref{eq:orthnu}) and from the expression  \eqref{eq:densitymuK} of the density of $\mu_{\sigma,\op{k}}$ with respect to $\nu^{E,\sigma}$.
\end{proof}

Let us now turn to a scaling property, which in the case of DPP is the following: sampling a DPP with kernel $K$ and then erasing each point of the resulting set independently of the others with probability $1-p$ results in a DPP with kernel $pK$. Note that erasing each point of a subset~$I$ of $\{1,\ldots,d\}$ with probability $1-p$ amounts to sampling a DPP with kernel $pP^{I}$, where $P^{I}$ is the diagonal matrix with $1$'s in the columns labelled by $I$ and $0$'s elsewhere.

The similar statement for DLP is this: if after sampling our DLP $\Q$ with kernel $\op{k}$ we keep, for each $i\in\{1,\ldots,s\}$, from $\Q\cap E_{i}$ only a uniform random subspace with binomial dimension of parameters $\dim(\Q\cap E_{i})$ and $p$, we obtain a DLP with kernel $p\op{k}$. 

In order to articulate and prove this statement, we need to understand the operation on $\Q$ that we just described in terms of DLP. Let us approach this question in the form of several simple examples. Each of the following statements can be checked using the algorithmic description of DLP given by Proposition \ref{prop:sampling}. 

\begin{example}\rm Let $p\in [0,1]$ be a real. Let us call {\em $p$-binomial subspace} of an inner product space $F$ for a uniformly distributed subspace of $F$ with random dimension distributed according to a binomial random variable with parameters $\dim F$ and~$p$.

1. A DLP in the coarse split inner space $(E,(E))$ with kernel $\op{k}=p\, \id_{E}$ is a $p$-binomial subspace of $E$. 

2. A DLP in the split inner product space $(E,\sigma)$ with kernel $p\, \id_{E}$ is the direct sum over $i\in \{1,\ldots,s\}$ of a $p$-binomial subspace of $E_{i}$.

3.  Let $Q$ be a subspace of $E$ adapted to $\sigma$, that is, an element of $\Gr(E,\sigma)$. A tempting but wrong guess would be that a DLP in $(E,\sigma)$ with kernel $p\proj{Q}$ is the sum over $i\in\{1,\ldots,s\}$ of a $p$-binomial subspace of $Q\cap E_{i}$.\footnote{Even more tempting but equally wrong would be the guess that this DLP is a $p$-binomial subspace of $Q$.} For instance, consider the case where $E$ is a Euclidean plane endowed with the coarse splitting $(E)$, and $Q$ is a line of $E$. Then a DLP with kernel $\proj{Q}$ is supported by the whole set of lines in $E$, the density of probability of a line forming an angle $\alpha$ with $Q$ being $\cos^{2}(\alpha)$.\footnote{The reader wondering how the function $\cos^{2}$, which notoriously has mean $\frac{1}{2}$, can be a density of probability, should remember that our reference measure on the set of lines of $E$ has total mass $\binom{2}{1}=2$.}

In fact, we need to use $Q$ to form a new splitting of $E$ that we denote by $\sigma\vee Q$, and which consists in the non-zero subspaces in the list
\[Q\cap E_{1},Q^{\perp}\cap E_{1}, \ldots, Q\cap E_{s},Q^{\perp}\cap E_{s}.\]
The DLP on $(E,\sigma\vee Q)$ with kernel $p\proj{Q}$ is the random space we were looking for: the direct sum over $i\in \{1,\ldots,s\}$ of a $p$-binomial subspace of $Q\cap E_{i}$.
\end{example}

Here is the formal statement.

\begin{proposition}[Scaling] \label{prop:scaling} Let $p\in[0,1]$ be a real number. Then
\[\E[\mu_{\sigma\vee \Q,p\proj{\Q}}]=\int_{\Gr(E,\sigma)} \mu_{\sigma\vee Q,p\proj{Q}} \d\mu_{\sigma,\op{k}}(Q)=\mu_{\sigma,p\op{k}}.\]
\end{proposition}

\begin{proof} Let us consider a point $R\in \Gr(E,\sigma)$. The density at $R$ of the measure $\E[\mu_{\sigma\vee Q,p\proj{\Q}}]$ with respect to $\nu^{E,\sigma}$ is equal to
\[\int_{\Gr(E,\sigma)}  \det\big(\op{k}\proj{Q}+(1-\op{k})\proj{Q^{\perp}}\big) \det\big(p\proj{Q}\proj{R}+(1-p\proj{Q})\proj{R^{\perp}}\big)\d\nu^{E,\sigma}(Q).\]
Multiplying the two determinants and reordering the terms, we find the expression
\[\int_{\Gr(E,\sigma)}  \det\big(\op{k}\proj{Q}(p\proj{R}+(1-p)\proj{R^{\perp}})+(1-\op{k})\proj{Q^{\perp}}\proj{R^{\perp}}\big)\d\nu^{E,\sigma}(Q).\]
Now we claim that $\proj{R}$ and $\proj{Q}$ commute in this expression. Indeed, the subspace $Q$ is adapted to $\sigma$, and $R$ is sampled under the measure $\mu_{\sigma\vee Q,p\proj{Q}}$. According to the discussion preceding the statement of Proposition \ref{prop:scaling}, $R$ is almost surely a subspace of $Q$. More precisely, for $R\in \Gr(E,\sigma\vee Q)$, the determinant $\det\big(p\proj{Q}\proj{R}+(1-p\proj{Q})\proj{R^{\perp}}\big)$ vanishes whenever $R\not\subseteq Q$. Thus, the value at $R$ of the density of the measure that we are computing is equal to
\[\int_{\Gr(E,\sigma)}  \det\big(\op{k}(p\proj{R}+(1-p)\proj{R^{\perp}})\proj{Q}+(1-\op{k})\proj{R^{\perp}}\proj{Q^{\perp}}\big)\d\nu^{E,\sigma}(Q).\]
Applying Proposition \ref{prop:multilin}, we find that it is equal to
\[\det\big(\op{k}(p\proj{R}+(1-p)\proj{R^{\perp}})+(1-\op{k})\proj{R^{\perp}}\big)=\det\big(p\op{k}\proj{R}+(1-p\op{k})\proj{R^{\perp}}\big),\]
and the result is proved.
\end{proof}

\subsection{Dimension and split dimension} 

\begin{proposition}[Dimension I] \label{prop:dimH} Assume that $\op{k}$ is an orthogonal projection of rank $n$. Then $\dim\Q=n$ almost surely.
\end{proposition}

\begin{proof} Indeed, the density of $\mu_{\sigma,\op{k}}$, which is given by \eqref{eq:densitymuK}, is, according to Proposition \ref{prop:cosdet}, supported by $\Gr_{n}(E,\sigma)$.
\end{proof}

We can in fact describe, in general, the Laplace transform of the split dimension of $\Q$. For that purpose, it will be useful to consider a class of very simple linear operators on $E$, namely those which act by a scalar on each of the subspaces $E_{1},\ldots,E_{s}$. We will call such operators {\em split scalar} operators. The set of split scalar operators is a subalgebra of $\End(E)$, indeed the commutant of the subalgebra $\End(E_{1})\oplus \ldots \oplus \End(E_{s})$. In particular, an operator is split scalar if and only if it commutes to $\proj{Q}$ for all $Q\in \Gr(E,\sigma)$.

We introduce $s$ indeterminates $t_{1},\ldots,t_{s}$ and consider the generic split scalar operator
\[T=t_{1}\proj{E_{1}}+\ldots+t_{s}\proj{E_{s}} \in \C[t_{1},\ldots,t_{s}]\otimes \End(E).\]

\begin{proposition}[Split dimension]\label{prop:dimQ}
Let $\op{k}$ be a kernel on $(E,\sigma)$. In the ring of formal series $\C[[t_{1},\ldots,t_{s}]]$, the following equalities hold:
\begin{equation}\label{eq:Laplacescal}
\E\big[e^{\Tr(T\proj{\Q})}\big]=\E\bigg[\exp \sum_{i=1}^{s} t_{i}\dim(\Q\cap E_{i})\bigg]=\det(\id_{E}+\op{k}(e^{T}-1)).
\end{equation}
\end{proposition}

\begin{proof}
The first equality is straightforward. The left-hand side is equal to
\begin{align*}
\E\big[\det\big(\proj{\Q}e^{T}+\proj{\Q^{\perp}}\big)\big]&=\int_{\Gr(E,\sigma)} \det\big(\op{k}\proj{Q}+(1-\op{k})\proj{Q^{\perp}}\big) \det\big(\proj{\Q}e^{T}+\proj{Q^{\perp}}\big) \d\nu^{E,\sigma}(Q)\\
&=\int_{\Gr(E,\sigma)} \det\big(\op{k}\proj{Q}e^{T}+(1-\op{k})\proj{Q^{\perp}}\big) \d\nu^{E,\sigma}(Q)
 \end{align*}
 and since $T$ commutes with $\proj{Q}$ for every $Q$ adapted to $\sigma$, this is equal to
\[\int_{\Gr(E,\sigma)} \det(\op{k}e^{T}\proj{Q}+(1-\op{k})\proj{Q^{\perp}})\; \d\nu^{E,\sigma}(Q)\]
which, by Proposition \ref{prop:multilin}, is in turn equal to $\det(1+\op{k}(e^{T}-1))$. 
\end{proof}

It is tempting to replace, in $\E\big[e^{\Tr(T\proj{\Q})}\big]$, the operator $T$ by something more general than a split scalar operator. Unfortunately, looking at the proof, we see that a crucial step is the commutation of $T$ with $\proj{Q}$ for every $Q$ adapted to $\sigma$, and this is equivalent to $T$ being scalar on each block of the splitting. This does not of course rule out the existence of a different argument that would allow one to treat a more general situation, but assuming that there is one, we were not able find it. 

\begin{corollary}[Refinement of splittings] Let $\delta$ be a splitting of $E$ finer than $\sigma$. Then the distribution of the split dimension of $\Q$ with respect to $\sigma$ is the same under $\mu_{\sigma,\op{k}}$ and $\mu_{\delta,\op{k}}$.
\end{corollary}

\begin{proof} This follows from \ref{prop:dimQ} and the fact that if $T$ is split scalar with respect to $\sigma$, it is also split scalar with respect to $\delta$.
\end{proof}

By taking for $\delta$ a splitting in lines of $E$ adapted to $\sigma$, this corollary tells us that the split dimension of $\Q$ under $\mu_{\sigma,\op{k}}$ has the same distribution as the vector of the number of points of the $\DLP(E,\delta,\op{k})$ (which is a DPP) that fall within the blocks of the partition of $\delta$ determined by $\sigma$.

\begin{corollary}[Dimension II] Let $\lambda_{1},\ldots,\lambda_{d}$ be the eigenvalues of the kernel $\op{k}$. Then the dimension of $\Q$ is distributed as the sum of $d$ independent Bernoulli random variables with parameter $\lambda_{1},\ldots,\lambda_{d}$.
\end{corollary}

\begin{proof} Specialising Proposition \ref{prop:dimQ} to $t_{1}=\ldots=t_{s}=t$, we find
\[\E\big[e^{t\dim\Q}\big]=\prod_{i=1}^{d} (\lambda_{i}e^{t} +(1-\lambda_{i})),\]
which is the expected Laplace transform.
\end{proof}

These results could have been deduced from their version for DPP, which is known to be true, using the sampling algorithm described in Proposition \ref{prop:sampling}. However, we believe that the more canonical proofs that we provide reveal more of the structure of DLP.

\subsection{Support}\label{sec:Support}

 The question of the support of the distribution of a DLP, which is a probability measure on a continuous space, is more subtle and interesting than the corresponding question for a DPP. 

We start by proving that the support of a DLP is a union of connected components of $\Gr(E,\sigma)$. Recall that these connected components are exactly the subsets $\Gr_{\ul{n}}(E,\sigma)$, for $\ul{0}\leq \ul{n}\leq \ul{d}$, where~$\ul{d}$ is the split dimension of $E$. Let us denote by $[\![\ul{0},\ul{d}]\!]$ the set of possible split dimensions of an element of $\Gr(E,\sigma)$.

In the following proposition, by the {\em support} of $\mu_{\sigma,\op{k}}$, we mean the smallest closed subset of $\Gr(E,\sigma)$ with measure $1$. We denote it by $\mathsf{Supp}(\mu_{\sigma,\op{k}})$.

\begin{proposition}\label{prop:supportstrate} Let $(E,\sigma)$ be a split inner product space. Let $\op{k}$ be a kernel on $E$. There exists a subset $D_{\sigma,\op{k}}\subseteq [\![\ul{0},\ul{d}]\!]$ such that 
\[\mathsf{Supp}(\mu_{\sigma,\op{k}})=\bigcup_{\ul{n}\in D_{\sigma,\op{k}}} \Gr_{\ul{n}}(E,\sigma).\]
\end{proposition}

We will prove this proposition by showing that for all $\ul{n}$, the density of $\mu_{\sigma,\op{k}}$ given by \eqref{eq:densitymuK} is either identically zero on $\Gr_{\ul{n}}(E,\sigma)$ or positive on a dense subset. For this, we will use the following lemma about the vanishing set of a representative function on a compact matrix group. By a compact matrix group, we mean a compact subgroup of $\GL(E)$ for some finite-dimensional vector space $E$. 

\begin{lemma}\label{lem:vanish} Let $E$ be a real or complex vector space. Let $G\subset \GL(E)$ be a compact group. Let $f:\End(E)\to \R$ be a polynomial function on the real vector space $\End(E)$. Then $f$ vanishes either identically on $G$, or on a closed subset of $G$ with empty interior.
\end{lemma}

\begin{proof} The proof relies on two facts. The first is a Taylor formula for $f$, and the second is the fact that the exponential map of $G$ is surjective.

Let $\g$ be the Lie algebra of $G$. For every $X\in \g$, the Lie derivative in the direction $X$ is the differential operator $\mathcal L_{X}$ on $G$ which acts on a smooth function, for instance $f$, by
\[(\mathcal L_{X}f)(g)=\frac{d}{dt}_{|t=0}f(ge^{tX}).\]
For every $g\in G$ and $X\in \g$, it follows from the fact that $f$ is a polynomial function that the function of one real variable $t\mapsto f(ge^{tX})$ is real analytic. Thus, we have the Taylor formula
\[f(ge^{X})=\sum_{n=0}^{\infty} \frac{1}{n!} \big((\mathcal L_{X})^{n}f\big)(g).\]

Let us now assume that the interior of the vanishing set of $f$ on $G$ is not empty, and contains an element $g\in G$. Then for all $X\in \g$, all iterated Lie derivatives of $f$ at $g$ in the direction $X$ vanish, and $f(ge^{X})=0$.

Since $G$ is compact, the exponential map $\exp : \g \to G$ is onto, and the last equality implies that $f$ vanishes identically on $G$.
\end{proof}

\begin{proof}[Proof of Proposition \ref{prop:supportstrate}] Let us choose $\ul{n}\in [\![\ul{0},\ul{d}]\!]$ and assume that  the density of $\mu_{\sigma,\op{k}}$ with respect to $\nu^{E,\sigma}$ is positive at a point $Q_{0}$ of $\Gr_{\ul{n}}(E,\sigma)$. 

Let us consider the maps $\U(E,\sigma)\stackrel{\pi}{\longrightarrow} \Gr_{\ul{n}}(E,\sigma) \stackrel{g}{\longrightarrow} \R$ defined by 
\[\pi(u)=u(Q_{0}) \ \ \text{ and } \ \ g(Q)=\det(\op{k}\proj{Q}+(1-\op{k})\proj{Q^{\perp}}).\]
Our assumption is that $g(Q_{0})>0$. On $\U(E,\sigma)$, the function $f=g\circ \pi$ can be written
\[f(u)=\det(\op{k}u\proj{Q_{0}}u^{*}+(1-\op{k})u\proj{Q_{0}^{\perp}}u^{*}).\]
In particular, it is the restriction to $\U(E,\sigma)$ of the polynomial function defined by the same formula on $\End(E_{1})\times \ldots \times \End(E_{s})$. 

Let $V$ be the vanishing set of $g$ on $\Gr_{\ul{n}}(E,\sigma)$. The vanishing set of $f$ on $\U(E,\sigma)$ is $\pi^{-1}(V)$. By our assumption and Lemma \ref{lem:vanish}, the interior of $\pi^{-1}(V)$ is empty. Since $\pi$ is continuous and onto, this implies that the interior of $V$ is empty. In particular, any non-empty open subset of $\Gr_{\ul{n}}(E,\sigma)$ contains a point, hence an open subset, where $g$ is positive, and has positive measure for $\mu_{\sigma,\op{k}}$.
\end{proof}

Now, we would like to understand better the set $D_{\op{k}}$, that is, identify those $\ul{n}$ for which $\Gr_{\ul{n}}(E,\sigma)$ is in the support of $\mu_{\sigma,\op{k}}$. The following corollary of Proposition \ref{prop:dimQ} constitutes a step in this direction.

\begin{corollary}\label{prop:massestrates} Let $(E,\sigma)$ be a split inner product space. Let $\op{k}$ be a kernel on $E$. For all $\ul{n}=(n_{1},\ldots,n_{s})\in [\![\ul{0},\ul{d}]\!]$, we have
\[\mu_{\sigma,\op{k}}(\Gr_{\ul{n}}(E,\sigma))=\text{the coefficient of } t_{1}^{n_{1}}\ldots t_{s}^{n_{s}} \text{ in } \det(T\op{k}+(1-\op{k})).\]
\end{corollary}

In a later section, we will provide an alternative description of these coefficients in terms of the Euclidean geometry of the exterior algebra of $E$, see Section \ref{sec:extalg} and Proposition \ref{prop:supportextalg}.

In the case where $\op{k}$ is an orthogonal projection, it is possible to characterise more explicitly those $\ul{n}$ which have positive probability of occurring as split dimensions of a DLP with kernel~$\op{k}$. We borrow this result from~\cite[Proposition~1.1]{Lafforgue} and rephrase it in our context to obtain the following statement.

\begin{proposition}\label{prop:lafforgue} Let $(E,\sigma)$ be a split inner product space. Let $H$ be a linear subspace of $E$. Set $n=\dim H$. Let $\ul{n}=(n_{1},\ldots,n_{s})\leq \ul{d}$  be such that $n_{1}+\ldots +n_{s}=n$. Then the following two assertions are equivalent.\\
\indent 1. $\mu_{\sigma,\proj{H}}(\Gr_{\ul{n}}(E,\sigma))>0$.\\
\indent 2. For all $T\subset \{1,\ldots,s\}$, the following inequality holds:
\[\sum_{t\in T} n_{t}\geq \dim(H\cap \bigoplus_{t\in T}E_{t}).\]
\end{proposition}

Let us indicate that the questions treated in this section can usefully be formulated in the language of matroids. For instance, the last inequality defines the rank function of a matroid on $\{1,\ldots,s\}$ and the support of the split dimension is an example of a matroid polytope. A beautiful short account of matroid theory can be found in \cite{Ardila}.

\subsection{Uniqueness} In this section and the next, we are concerned with the case where the kernel of our DLP is a projection. 

\begin{proposition}[Uniqueness] \label{prop:uniqueness} Assume that $\op{k}$ is the orthogonal projection on a linear subspace $H$ of $E$. Then almost surely, one has the equality
\[\Q\oplus H^{\perp}=E.\]
\end{proposition}

\begin{proof} Since, by Proposition \ref{prop:dimH}, $\dim \Q+\dim H^{\perp}=\dim E$ almost surely, it suffices to prove that $\Q\cap H^{\perp}=0$ almost surely. But the second assertion of Proposition \ref{prop:cosbasic} ensures that the density of $\mu_{\sigma,\op{k}}$, given by \eqref{eq:densityDLPproj}, vanishes at every $Q\in \Gr(E,\sigma)$ such that $\Q\cap H^\perp\ne0$.
\end{proof}

This property can be interpreted as a statement of uniqueness in a sense that was already explored by Lyons and Bufetov--Qiu--Shamov for DPP in \cite{Lyons-DPP,BufetovQiuShamov}: it says that almost surely, the map
\begin{align*}
H & \longrightarrow \Q^{*}\\
h & \longmapsto \langle h, \cdot \rangle
\end{align*}
is a bijection. If we think of $H$ as a space of linear functions on $E$, then $\Q$ is almost surely a uniqueness set for this space of functions: any function in this space that vanishes on $\Q$ is identically zero. 

We will come back to this property a bit later and pursue the idea that $\Q$ is a (random) substitute to $H$ that is adapted to the splitting $\sigma$. In particular, one can consider the projection on $\Q$ parallel to $H^{\perp}$, that is random and not orthogonal, as a substitute for the orthogonal projection on $H$.

\subsection{Mean projection} \label{sec:meanproj}

Let us consider again the case where the kernel $\op{k}$ is the orthogonal projection on some linear subspace $H$ of $E$, of dimension $n$. In this case, the associated DLP has almost surely dimension $n$ and is, according to Proposition \ref{prop:uniqueness}, a linear complement of $H^{\perp}$. We will study the almost surely defined projection on $\Q$ parallel to $H^{\perp}$. Let us emphasise that this is in general not an orthogonal projection. 

We will prove that, in a strong sense, the average of this projector is the orthogonal projection onto $H$. More precisely, we will prove that in any basis of $E$, the average of every minor, principal or not, of the matrix of this random non-orthogonal projection is equal to the corresponding minor of the orthogonal projection onto $H$. The fact that Theorem \ref{prop:meanprojection} is equivalent to this strong property is more apparent from its reformulation in Theorem \ref{thm:wedgeA}. 

To the best of our knowledge, Theorem \ref{prop:meanprojection} was not known in this generality even in the case of DPP, although it is non-trivial in that case too. We are grateful to  a referee for pointing out to us that Lyons, in \cite[Proposition 6.8]{Lyons-DPP},  already proved, revisiting \cite{Maurer}, that the average of the projection on $\Q$ parallel to $H^{\perp}$ is the orthogonal projection onto $H$. This result was also proved in the special case of cycle-rooted spanning forests of a graph in~\cite[Theorem A]{CCK-line} (the weaker form of that result in the case of spanning trees is presented in \cite[Proposition 7.3]{Biggs-graph} and arguably dates back to Kirchhoff, although not stated in this form; an earlier version is in \cite{Nerode-Shank}). 

\begin{theorem}[Mean projection] \label{prop:meanprojection} Assume that $\op{k}$ is the orthogonal projection on a linear subspace $H$ of $E$. Let $\op{P}^{\Q}$ denote the almost surely defined projector onto $\Q$ parallel to $H^{\perp}$. Let $\op{a}$ be a linear endomorphism of $E$. Then
\begin{equation}\label{eq:wedgeA}
\E\big[\det\big(1+\op{a}\op{P}^{\Q}\big)\big]=\det\big(1+\op{a}\proj{H}\big).
\end{equation}
\end{theorem}

Thanks to the self-adjointness of $\proj{H}$, one can replace, in \eqref{eq:wedgeA}, the operator $\op{P}^{\Q}$ by its adjoint, which is the projection on $H$ parallel to $\Q^{\perp}$.

\begin{proof} Let us write the left-hand side of \eqref{eq:wedgeA} using \eqref{eq:densitymuK}. We find
\[\E\big[\det\big(1+\op{a}\op{P}^{\Q}\big)\big]=\int_{\Gr_{n}(E,\sigma)}\det\big(1+\op{a}\op{P}^{Q}\big)\det\big(\proj{H}\proj{Q}+\proj{H^{\perp}}\proj{Q^{\perp}}\big)\d\nu^{E,\sigma}_n(Q).\]
We will now multiply the determinants. Before that, let us observe that, $H^{\perp}$ being by definition the kernel of $\op{P}^{Q}$, we have $\op{P}^{Q}\proj{H^{\perp}}=0$. 
Moreover, we claim that $\op{P}^{Q}\proj{H}\proj{Q}=\proj{Q}$. Indeed, if $v$ is a vector of $Q$, then $\proj{H}v-v$ belongs to $H^{\perp}$, so that $\op{P}^{Q}\proj{H}v=\op{P}^{Q}v$ and $\op{P}^{Q}v=v$. Thus, we have
\[\E\big[\det\big(1+\op{a}\op{P}^{\Q}\big)\big]=\int_{\Gr_{n}(E,\sigma)}\det\big((\proj{H}+\op{a})\proj{Q}+\proj{H^{\perp}}\proj{Q^{\perp}}\big)\d\nu^{E,\sigma}_n(Q).\]
We will now use a refinement of Proposition \ref{prop:multilin}: we claim that for all endomorphisms $\op{a}$ and $\op{b}$ of $E$, we have
\begin{equation}\label{eq:refmulti}
\int_{\Gr_{\ul{n}}(E,\sigma)} \det(\op{a}\proj{Q}+\op{b}\proj{Q^{\perp}}) \d\nu^{E,\sigma}_{\ul{n}}(Q)
=\text{the coefficient of } t_{1}^{n_{1}}\ldots t_{s}^{n_{s}} \text{ in } \det(T\op{a}+\op{b}).
\end{equation}
Indeed, let us choose an orthonormal basis $(e_{1},\ldots,e_{d})$ of $E$ adapted to $\sigma$. Writing matrices in this basis, we can compute the coefficient of $t_{1}^{n_{1}}\ldots t_{s}^{n_{s}}$ in $\det(T\op{a}+\op{b})$ and find that it is equal to
\[\sum_{\substack{J\subseteq \{1,\ldots,d\}\\ \sdim{E_{J}}=\ul{n}}} \det(\op{a}\proj{E_{J}}+\op{b}\proj{E_{J}^{\perp}}).\]
Averaging over all orthonormal bases of $E$ adapted to $\sigma$ as we did in the proof of Proposition~\ref{prop:multilin} yields the left-hand side of~\eqref{eq:refmulti}.

Summing \eqref{eq:refmulti} over all $\ul{n}$'s with weight $n$, we find that 
\[\E\big[\det\big(1+\op{a}\op{P}^{\Q}\big)\big]=\text{the coefficient of } t^{n} \text{ in } \det\big(t(\op{a}+\proj{H})+\proj{H^{\perp}}\big).\]
To check that this coefficient is the right-hand side of \eqref{eq:wedgeA}, let us write the matrix of $\op{a}$ in a basis adapted to the splitting $(H,H^{\perp})$ of $E$:
\[\op{a}=\begin{blockarray}{ccc}
  &{\scriptstyle H} & {\scriptstyle H^{\perp}} \\
\begin{block}{c(cc)}
    {\scriptstyle H}  & A & B  \\
  {\scriptstyle H^{\perp}} & C & D  \\
\end{block}
\end{blockarray} \,\, .
\]
Then, assuming that $A+1$ is invertible, we find by Schur's formula
\[\det \begin{pmatrix} tA + t & t B \\ t C & tD+1 \end{pmatrix}=\det(tA+t) \det(tD+1-tC(A+1)^{-1}B).
\]
The first factor in the product on the right-hand side is $t^{n} \det(A+1)$ and the second is $1+O(t)$. The coefficient of $t^{n}$ in this determinant is thus $\det(A+1)$, which is equal to $\det(1+\op{a}\proj{H})$.
\end{proof}

\subsection{Stochastic domination} Let us consider two kernels $\op{k}_{1}$ and $\op{k_{2}}$ on $(E,\sigma)$ such that $0\leq \op{k}_{1}\leq \op{k}_{2}\leq 1$. We want to establish a property of stochastic domination of the measure $\mu_{\sigma,\op{k}_{1}}$ by the measure $\mu_{\sigma,\op{k}_{2}}$. In fact, we prove the existence of a monotone coupling of these two DLP. 

\begin{proposition}[Monotone coupling]\label{prop:monotone-coupling} Let $\op{k}_{1}$ and $\op{k_{2}}$ be two kernels on $(E,\sigma)$ such that $0\leq \op{k}_{1}\leq \op{k}_{2}\leq 1$. There exists a coupling of a DLP $\Q_{1}$ with kernel $\op{k}_{1}$ and a DLP $\Q_{2}$ with kernel $\op{k}_{2}$ such that $\Q_{1}\subseteq \Q_{2}$ almost surely.
\end{proposition}

\begin{proof} We use the corresponding result for DPP, proved for instance by Lyons in \cite[Theorems 6.2 and 8.1]{Lyons-DPP} (based on a result of stochastic domination and a theorem of Strassen), and infer it for DLP using Proposition \ref{prop:sampling}. We choose a uniform orthonormal basis $(e_{1},\ldots,e_{d})$ of $E$ adapted to $\sigma$. Then, we write the matrices $K_{1}$ and $K_{2}$ of our kernels in this basis, and sample the corresponding DPP $\X_{1}$ and $\X_{2}$ in a way that $\X_{1}\subseteq \X_{2}$. Finally, we set $\Q_{1}=\Vect(e_{i}: i\in \X_{1})$ and $\Q_{2}=\Vect(e_{i}: i\in \X_{2})$.
\end{proof}

It would be more satisfactory to have a proof of the existence of this coupling, or equivalently, of the stochastic domination of $\mu_{\sigma,\op{k}_{1}}$ by $\mu_{\sigma,\op{k}_{2}}$, that does not depend on previous results of DPP, but we were not able to find such a proof. 

\subsection{Negative association}

It is known that DPP satisfy a property called negative association (see e.g. \cite[Theorems 6.2 and 8.1]{Lyons-DPP} or \cite[Theorem 3.4]{Borcea-Branden-Liggett}). DLP satisfy an analogous property.

\begin{proposition}
Let $R$ be a subspace of $E$ that is the direct sum of some of the spaces $E_{1},\ldots,E_{s}$. Let $f,g$ be continuous non-decreasing functions on $\Gr(E,\sigma)$. Then $\E[f(\Q\cap R)g(\Q\cap R^{\perp})]\leq \E[f(\Q\cap R)] \E[g(\Q\cap R^{\perp})]$.
\end{proposition}
\begin{proof}
As in the proof of Proposition \ref{prop:monotone-coupling}, we deduce this from the corresponding result for DPP along with Proposition \ref{prop:sampling}.
\end{proof}

Let us indicate an inequality that relates to this property of negative association:
given any two orthogonal elements $R_{1}$ and $R_{2}$ of $\Gr(E,\sigma)$, the inequality
\begin{equation}
\det \op{k}^{R_{1}\oplus R_{2}}_{R_{1}\oplus R_{2}}\leq \det \op{k}^{R_{1}}_{R_{1}} \det \op{k}^{R_{2}}_{R_{2}}
\end{equation}
holds, and shows informally that
\[\d\mu_{\sigma,\op{k}}(R_{1}\oplus R_{2} \subseteq \Q) \leq \d\mu_{\sigma,\op{k}}(R_{1} \subseteq \Q) \d\mu_{\sigma,\op{k}}(R_{2} \subseteq \Q).\]
This inequality is nothing more than the following property (known as the Fischer, Hadamard or Koteljanskii inequality) of which we give a short proof.

\begin{lemma} Let $M$ be a non-negative Hermitian matrix written in block form
\[M=\begin{pmatrix} A & B \\ B^{*} & D \end{pmatrix}\]
in such a way that $A$ and $D$ are square matrices. Then $\det (M) \leq \det (A) \det (D)$.
\end{lemma}

\begin{proof} Adding if necessary a small multiple of the identity to $M$, which can be removed by continuity of the determinant at the end of the proof, we can and will assume that $M$, $A$ and $D$ are non singular. Then, according to a classical trick, we can eliminate $B$ and $B^{*}$ by multiplying $M$ on the left and on the right by appropriate matrices:
\[\begin{pmatrix} I & 0 \\ -B^{*}A^{-1} & I \end{pmatrix} \begin{pmatrix} A & B \\ B^{*} & D \end{pmatrix}\begin{pmatrix} I & -A^{-1}B \\ 0 & I \end{pmatrix}=\begin{pmatrix} A & 0 \\ 0 & D-B^{*}A^{-1}B \end{pmatrix}.\]
On one hand, taking the determinant on both sides of this equality yields the classical Schur formula
\[\det(M)=\det(A)\det(D-B^{*}A^{-1}B).\]
On the other hand, the left-hand side of the equality is a positive Hermitian matrix. Therefore, 
$D-B^{*}A^{-1}B$, which is a principal sub-matrix of a positive Hermitian matrix, is also positive Hermitian. Since $B^{*}A^{-1}B$ is non-negative, we thus have
\[0\leq D-B^{*}A^{-1}B \leq D.\]
Now, if $S$ and $D$ are positive Hermitian matrices such that $0\leq S\leq D$, then $0\leq D^{-\frac{1}{2}}SD^{-\frac{1}{2}}\leq I$, so that $\det(S)\leq \det(D)$. Thus, our last inequality entails
\[\det(D-B^{*}A^{-1}B)\leq \det (D),\]
and the proof is complete.
\end{proof}

\subsection{Restriction} Let us choose an integer $t\in\{0,\ldots,s\}$ and consider the linear subspace $F=E_{1}\oplus \ldots \oplus E_{t}$ of $E$, endowed with the splitting $\tau=(E_{1},\ldots,E_{t})$. 

\begin{proposition}[Restriction] \label{prop:restriction} The random subspace $\Q\cap F$ of $F$ is a DLP on $(F,\tau)$ with kernel $\op{k}_{F}^{F}$.
\end{proposition}

\begin{proof} It is possible to prove the result by computing the density of the distribution of $\Q\cap F$ on $\Gr(F,\tau)$, using Proposition \ref{prop:multilin}. 

Another possibility is to compute the incidence measure of the distribution of $\Q\cap F$. This can be done almost without any computation, by observing that the square
\[\xymatrix{
{\mathcal M}_{1}(\Gr(E,\sigma)) \ar[r]^{\Inc{}} \ar[d]_{(\,\cdot\,\, \cap F)_{*}} &  {\mathcal M}(\Gr(E,\sigma)) \ar[d]^{{\Res}_{F}} \\
{\mathcal M}_{1}(\Gr(F,\tau)) \ar[r]^{\Inc{}} &  {\mathcal M}(\Gr(F,\tau)) 
}\]
is commutative. In this diagram, $\mathcal M_{1}$ and $\mathcal M$ denote respectively the spaces of Borel probability measures and Borel measures on the corresponding Grassmannians. The horizontal arrows send a probability measure to its incidence measure. The vertical arrow on the left sends a probability measure to its image measure by the map $Q\mapsto Q\cap F$. The vertical arrow on the right sends a measure to its restriction to the subset $\Gr(F,\tau)$ of $\Gr(E,\sigma)$. To check that this square is commutative, it suffices to observe that for every $Q\in \Gr(E,\sigma)$, one has
\[{\Res}_{F}(\Inc{\delta_{Q}})=\Inc{\delta_{Q\cap F}}=\nu^{Q\cap F, \tau_{Q\cap F}}.\]
The case of a general probability measure on $\Gr(E,\sigma)$ follows by the linearity of all maps of the diagram. 

Applying this observation to the distribution of $\Q$, we find that the incidence measure of the distribution of $\Q\cap F$ is the restriction to $\Gr(F,\tau)$ of the incidence measure of $\Q$. By definition of a DLP, this restriction has the density $R\mapsto \det \op{k}_{R}^{R}$ with respect to $\nu^{F,\tau}$. To complete the proof, we observe that for all $R\in \Gr(F,\tau)$, 
\[\det \op{k}^{R}_{R}=\det (\op{k}^{F}_{F})^{R}_{R}\]
and $\op{k}^{F}_{F}$ is a kernel on $F$.
\end{proof}

\subsection{Determinantal linear processes in infinite-dimensional spaces} Our understanding of restrictions of determinantal linear processes will allow us to extend their definition to a (mildly) infinite-dimensional setting. 

Let us start by extending the definition of a split inner product space (Definition \ref{def:sips}). We call {\em split Hilbert space} a pair $(E,\sigma)$ where $E$ is a  separable Hilbert space and $\sigma=(E_{i})_{i\in I}$ is a countable family of pairwise orthogonal finite-dimensional linear subspaces of $E$ of positive dimension such that 
\[E=\bigoplus_{i\in I} E_{i},\]
the direct sum being in the category of Hilbert spaces. Let us repeat the crucial assumption that each summand $E_i$ is \emph{finite-dimensional}.

Let $(E,\sigma)$ be a split Hilbert space. We define 
\[\Gr(E,\sigma)=\prod_{i\in I}\Gr(E_{i})\]
and endow this space with the product topology. This makes it a compact topological space, of which the connected components are the subspaces
\[\Gr_{\ul{n}}(E,\sigma)=\prod_{i\in I}\Gr_{n_{i}}(E_{i}), \ \ul{n}=(n_{i})_{i\in I}\in \N^{I}.\]

For every $J\subseteq I$, let us define the split Hilbert space $(E_{J},\sigma_{J})$, with
\[E_{J}=\bigoplus_{j\in J} E_{j} \text{ and } \sigma_{J}=(E_{j})_{j\in J}.\]
For all $J,K\subseteq I$ such that $J\subseteq K$, let us define the restriction map
\begin{align*}
\Res_{JK}:\Gr(E_{K},\sigma_{K})&\longrightarrow \Gr(E_{J},\sigma_{J}) \\
\nonumber Q & \longmapsto Q\cap E_{J}. 
\end{align*}
We will also use, for all $J\subseteq I$, the notation $\Res_{J}=\Res_{JI}:\Gr(E,\sigma)\longrightarrow \Gr(E_{J},\sigma_{J})$.

According to Kolmogorov's extension theorem, a Borel probability measure on $\Gr(E,\sigma)$ is the same thing as a consistent collection of Borel probability measures on all spaces $\Gr(E_{J},\sigma_{J})$, where $J$ is a finite subset of $I$.

The notation introduced for compressions of endomorphisms in Section \ref{sec:minors} extends to bounded linear operators and closed subspaces of a Hilbert space. With the tools that we have in hand, the construction of determinantal subspaces of a split Hilbert space is straigthforward.

\begin{theorem}\label{thm:infinite} 
Let $(E,\sigma)$ be a split Hilbert space, with $\sigma=(E_{i})_{i\in I}$. Let $\op{k}$ be a bounded self-adjoint linear operator on $E$ such that $0\leq \op{k}\leq 1$.

There exists a unique probability measure $\mu_{\sigma,\op{k}}$ on $\Gr(E,\sigma)$ such that for every finite subset $J\subset I$, the image of $\mu_{\sigma,\op{k}}$ by the projection $\Res_{J}:\Gr(E,\sigma)\to \Gr(E_{J},\sigma_{J})$ is equal to $\mu_{\sigma_{J},\op{k}^{E_{J}}_{E_{J}}}$.
\end{theorem}

\begin{proof} For every finite subset $J$ of $I$, the operator $\op{k}_{E_{J}}^{E_{J}}$ is self-adjoint on $E_{J}$, and satisfies $0\leq \op{k}_{E_{J}}^{E_{J}} \leq 1$ : it is a kernel on $E_{J}$. Thus, the measure $\mu_{\sigma_{J},\op{k}^{E_{J}}_{E_{J}}}$, which we denote by $\mu_{J}$ for simplicity, is well defined. Moreover, it follows immediately from Proposition \ref{prop:restriction} that the family of probability spaces
\[\bigg(\big(\Gr(E_{J},\sigma_{J}),\mu_{J}\big) : J\subset I, J \text{ finite}\bigg)\]
together with the restriction maps $(\Res_{JK} : J\subset K\subset I, J \text{ and } K \text{ finite})$ form a projective system. Each probability space of this system is a compact topological space equipped with a Borel probability measure. Thus, there exists an inverse limit to this system. This inverse limit is a probability measure on the Borel $\sigma$-field of 
\[\lim_{\longleftarrow} \Gr(E_{J},\sigma_{J})=\Gr(E,\sigma)\]
endowed with its product topology. By definition, this probability measure, which we denote by $\mu_{\sigma,\op{k}}$, is the unique probability measure on $\Gr(E,\sigma)$ such that for every finite subset $J$ of $I$, the image of $\mu_{\sigma,\op{k}}$ by $\Res_{J}$ is equal to $\mu_{J}$.
\end{proof}

\begin{example}
The previous theorem, in its version for DPP, is what allows to define objects such as the free and wired uniform spanning forests on infinite graphs; see e.g. \cite[Section 4]{Lyons-Betti}. Similarly, quantum spanning forests (see Section \ref{example:UST-QSF}) on infinite graphs are built in \cite{KL4} using the DLP version of the above theorem.
\end{example}

%%%%%%%%%%%%%%%%%%%%%%%%%%%%%%%%%%%%
%%%%%%%%%%%%%%%%%%%%%%%%%%%%%%%%%%%%
%%%%%%%%%%%%%%%%%%%%%%%%%%%%%%%%%%%%

\section{The point of view of the exterior algebra} \label{sec:extalg}

In this section, we will discuss a slightly more abstract point of view on DLP and see how their properties can be understood in terms of the Euclidean geometry of the exterior algebra of the ambient space. This point of view was already largely adopted by Lyons in his study of DPP  \cite{Lyons-DPP, Lyons-ICM}. Taking a more abstract point of view has the usual advantages and  disadvantages: depending on one's familiarity with the language of exterior algebra, it will obscure things, or make them more transparent. In any case, we will not assume prior knowledge of this piece of linear algebra, and offer an introduction to the notions that we use.

\subsection{The exterior algebra} \label{subsec:extalg}
Let $E$ be a vector space of dimension $d$. Properly speaking, the exterior algebra $\ext E$ of $E$ is  defined only up to isomorphism as a solution of a universal problem. From an only slightly less canonical point of view, it is a quotient of the tensor algebra of $E$. We prefer to take the distinctly less canonical but more concrete point of view that the exterior algebra of $E$ is the subspace of the tensor algebra of $E$ consisting of all fully antisymmetric tensors. 

Let $k\geq 0$ be an integer. We define the linear endomorphism $A_{k}$ of $E^{\otimes k}$ by setting, for all $v_{1},\ldots,v_{k}\in E$,
\[A_{k}(v_{1}\otimes \ldots \otimes v_{k})=\frac{1}{k!}\sum_{\sigma\in S_{k}}\epsilon(\sigma)v_{\sigma(1)}\otimes \ldots \otimes v_{\sigma(k)}.\]
The endomorphism $A_{k}$ is a projection and we define the $k$-th exterior power of $E$ as its range: 
\[\ext^{k} E = A_{k}(E^{\otimes k}).\] 
For $k>d$, the space $\ext^{k}E$ is the null vector space, and we set
\[\ext E=\bigoplus_{k=0}^{d} \ext^{k}E.\]
We endow this vector space with a structure of graded algebra by setting, for all $k,l\in \{0,\ldots,d\}$, and all $x\in \ext^{k}E$ and $y\in \ext^{l}E$,
\[x\wedge y =A_{k+l}(x\otimes y).\]
One checks that this makes $\ext E$ a graded associative algebra with unit $1\in\ext^0 E=\mathbb{K}$, where $\mathbb{K}$ is the field of scalars.

Let us assume that $E$ is endowed with an inner product. We endow the exterior algebra of $E$ with an inner product by setting, for all $k,l\in \{0,\ldots,d\}$, and all $v_{1},\ldots,v_{k}$ and $w_{1},\ldots,w_{l}$ in $E$,
\[\langle v_{1}\wedge \ldots \wedge v_{k},w_{1}\wedge \ldots \wedge w_{l}\rangle=\left\{\begin{array}{ll} \det(\langle v_{i},w_{j}\rangle)_{i,j\in \{1,\ldots,k\}} & \text{if } k=l,\\
0 & \text{if } k\neq l.\end{array}\right.\]
Let us note that this inner product is the restriction to the exterior algebra of the inner product on the full tensor algebra defined by
\[\langle v_{1}\otimes \ldots \otimes v_{k},w_{1}\otimes  \ldots \otimes w_{l}\rangle=\delta_{k,l} k! \langle v_{1},w_{1}\rangle \ldots \langle v_{k},w_{k}\rangle.\]

Let $(e_{1},\ldots,e_{d})$ be an orthonormal basis of $E$. Let us define, for all subset $I=\{i_{1}<\ldots <i_{k}\}$  of $\{1,\ldots,d\}$, the tensor $e_{I}=e_{i_{1}}\wedge \ldots \wedge e_{i_{k}}$. Then $\{e_{I}:I\subset \{1,\ldots,d\}\}$ is an orthonormal basis of $\ext E$. In particular, $\ext E$ has dimension $2^{d}$.

\subsection{Exterior algebra and splittings}\label{sec:extalgsplit} Let $(E,\sigma)$ be a split inner product space. The splitting of $E$ induces an orthogonal decomposition of the exterior algebra of $E$ as follows. 

Let us write $\sigma=(E_{1},\ldots,E_{s})$ as usual. For each $r\in \{1,\ldots,s\}$, the inclusion map $E_{r}\hookrightarrow E$ induces an injective map $\ext E_{r} \hookrightarrow \ext E$, for which we do not use any special notation. From these injective maps, we can form the map
\begin{align*}
\ext E_{1}\otimes \ldots \otimes \ext E_{r} & \longrightarrow \ext E \\
x_{1} \otimes \ldots \otimes x_{r} \hspace{12pt}& \longmapsto x_{1}\wedge \ldots \wedge x_{r}
\end{align*}
which is an isomorphism of vector spaces\footnote{It is also an isomorphism of algebras, provided one understands tensor products in the category of $\Z/2\Z$-graded algebras. We will however only need the linear isomorphism.}.

Let us write, as we did before, $\ul{d}=(d_{1},\ldots,d_{s})=(\dim E_{1},\ldots,\dim E_{s})$. Then for all $n\in \{0,\ldots,d\}$, the isomorphism above restricts to the following isomorphism of vector spaces:
\begin{equation}\label{eq:decextE}
\ext^{n}E\simeq \bigoplus_{\ul{n}\leq \ul{d}, |\ul{n}|=n}\underbrace{\ext^{n_{1}}E_{1}\otimes \ldots \otimes \ext^{n_{s}}E_{s}}_{\ext^{\ul{n}}E}.
\end{equation}
This decomposition of $\ext^{n}E$ can be understood very concretely by building an orthonormal basis $(e_{1},\ldots,e_{d})$ of $E$ as the concatenation of orthonormal bases of $E_{1},\ldots,E_{s}$, and partitioning the basis $\{e_{I}:I\subset \{1,\ldots,d\}, |I|=n\}$ of $\ext^{n}E$ according to the number of factors in each subspace $E_{1},\ldots,E_{s}$.

A simple fact that plays an important role for us is that the decomposition \eqref{eq:decextE} is an orthogonal decomposition of $\ext E$ -- indeed an orthogonal splitting of $\ext E$, although we are not going to apply to this splitting the same treatment that we apply to the original splitting $(E,\sigma)$.

We will make use of the orthogonal projections on $\ext^{\ul{n}}E$, that we denote by 
\begin{equation}\label{eq:defpiuln}
\proj{\ul{n}}=\op{1}_{\ext^{\ul{n}}E}\op{1}^{\ext^{\ul{n}}E}:\ext E\to \ext E,
\end{equation}
instead of $\proj{\ext^{\ul{n}}E}$. Similarly, we will denote by 
\begin{equation}\label{eq:defpin}
\proj{n}=\sum_{\ul{n}\leq \ul{d}, |\ul{n}|=n} \proj{\ul{n}}
\end{equation}
the orthogonal projection on $\ext^{n}E$.

\subsection{The Pl\"ucker embedding}\label{sec:Plucker} Given a linear subspace $F$ of dimension $n$ of $E$, the subset
\[\{f_{1}\wedge \ldots \wedge f_{n} : (f_{1},\ldots,f_{n}) \text{ basis of }F\}\]
of $\ext^{n}E$ is the complement of $\{0\}$ in a line which we denote by $\iota(F)$. The map
\begin{align*}
\iota : \Gr(E) & \longrightarrow \Gr_{1}\big(\ext E\big)\\
F &\longmapsto  \iota(F)
\end{align*}
is a minute variant of a classical map called the Pl\"ucker embedding\footnote{The classical Pl\"ucker embedding takes its values in the projective space of the exterior algebra rather than in $\Gr_{1}(\ext E)$.}. 

The relevance of this construction to the description of DLP is made apparent by the following result.

\begin{proposition} Let $E$ be an inner product space. Let $F$ and $G$ be two linear subspaces of $E$. 

1. If $F$ and $G$ do not have the same dimension, then $\iota(F) \perp \iota(G)$ and $\cos^{2}(\iota(F),\iota(G))=0$.

2. It $F$ and $G$ have the same dimension, then
\[\cos^{2}_{E}(F,G)=\cos^{2}_{\ext E}(\iota(F),\iota(G)),\]
where the subscript indicates the space in which the cosine is computed.

3. In all cases,
\[\cos^{2}_{\ext E}(\iota(F),\iota(G))=\Tr_{\ext E}\big(\proj{\iota(F)}\proj{\iota(G)}\big)={\det}_{E}\big(\proj{F}\proj{G}+\proj{F^{\perp}}\proj{G^{\perp}}).\]
\end{proposition}

\begin{proof} 1. If $F$ has dimension $n$ and $G$ dimension $m$, then $\iota(F)$ belongs to $\ext^{n}E$ and $\iota(G)$ to $\ext^{m}E$. If $m\neq n$, these two subspaces of $\ext E$ are orthogonal.

2. Let $(f_{1},\ldots,f_{n})$ and $(g_{1},\ldots,g_{n})$ be orthonormal bases of $F$ and $G$ respectively. Then $f_{1}\wedge \ldots \wedge f_{n}$ and $g_{1}\wedge \ldots \wedge g_{n}$ are unit vectors of $\iota(F)$ and $\iota(G)$. The square of their scalar product is the square of the cosine of the angle between the two lines $\iota(F)$ and $\iota(G)$. On the other hand, using the definition of the scalar product on $\ext E$ and Proposition \ref{prop:matrixcos}, we find
\[\cos^{2}(\iota(F),\iota(G))=\langle f_{1}\wedge \ldots \wedge f_{n},g_{1}\wedge \ldots \wedge g_{n}\rangle^{2} = \det(\langle f_{i},g_{j}\rangle)_{i,j=1,\ldots, n}^{2}=\cos^{2}(F,G).\]

3. The first equality follows from the fact that, in an inner product space, the trace of the product of the orthogonal projection on two lines is equal to the square cosine of the angle between these two lines. The second equality follows from the previous assertions and Proposition \ref{prop:cosdet}.
\end{proof}

The previous proposition implies in particular that the Pl\"ucker embedding is an injective map. Indeed, if $\iota(F)=\iota(G)$, then $F$ and $G$ have the same dimension and $\cos^{2}(F,G)=1$, which by  
the third assertion of Proposition \ref{prop:cosbasic} implies that $F=G$.

\subsection{Linear maps} The construction of the exterior algebra is functorial and every linear endomorphism of $E$ gives rise to an algebra endomorphism of $\ext E$. Concretely, any linear endomorphism $\op{a}$ of $E$ induces the algebra endomorphism $\ext \op{a}$ defined by the formula
\[\big(\ext \op{a}\big)(v_{1}\wedge \ldots \wedge v_{k})=\op{a}(v_{1})\wedge \ldots \wedge \op{a}(v_{k}).\]

We can now explain why the language of the exterior algebra is so convenient in situations where minors of matrices play a prominent role. 

Let $(e_{1},\ldots,e_{d})$ be an orthonormal basis of $E$. Let $\{e_{I}:I\subset\{1,\ldots,d\}\}$ be the orthonormal basis of $\ext E$ that it induces. Let $\op{a}$ be a linear endomorphism of $E$ and let $\ext \op{a}$ be the induced algebra endomorphism of $\ext E$. Let $A$ be the matrix of $\op{a}$ in the basis $(e_{1},\ldots,e_{d})$. Then for all $I$ and $J$ subsets of the same size of $\{1,\ldots,d\}$, one has the equality
\begin{equation}\label{eq:minors}
\det(A^{I}_{J})=\langle e_{I}, \big(\ext \op{a}\big) e_{J}\rangle.
\end{equation}
In words: with respect to an orthonormal basis of $E$ and the induced orthonormal basis of $\ext E$, the matrix coefficients of $\ext \op{a}$ are the minors of the matrix representing $\op{a}$.

Many important relations follow from this observation, for example the relation
\begin{equation}\label{eq:trdet}
\Tr_{\ext E}\big(\ext \op{a}\big)=\det(1+\op{a}),
\end{equation}
both sides being equal to the sum of all principal minors of $\op{a}$. We will state and prove a useful generalisation of this equality once we have introduced the Hodge operator (see Lemma \ref{lem:trdetab}).

\subsection{Quantum measurement interpretation} \label{sec:Qmeasure} Before explaining in more technical detail how computations about DLP can be made in the exterior algebra, we will pause briefly to discuss, rather informally and without claiming to any physical accuracy, how we think of DLP in terms inspired from quantum mechanics. 

This is certainly known for DPP, as Macchi's \cite{Macchi} introduction of DPP was motivated by quantum mechanics; see for instance a recent work of Olshanski on the subject \cite{Olshanski}. However, we did not find a reference for the way we present the processes in terms of quantum measurement.

If a Hermitian space $E$ is the state space of a certain type of quantum particle, then in absence of any information about the nature of this particle, the state space of a system formed by $n$  particles of the same type is $E^{\otimes n}$. If our particle is a \emph{fermion}, then this state space can be reduced to $\ext^{n}E$. Furthermore, the state space of a system formed by an arbitrary number of identical such particles is the full exterior algebra $\ext E$.

Let us assume, as usual, that $E$ has dimension $d$. Let us consider a system formed by $n$  particles. Since the particles cannot be distinguished, it makes no sense to say that the first particle is in the state $e_{1}$, the second in the state $e_{2}$, and so on. What can perhaps be said is that the $n$ particles are collectively in a state described by the $n$ vectors $\{e_{1},\ldots,e_{n}\}$ and it turns out that the meaningful quantity describing the system is the linear span $H=\Vect(e_{1},\ldots,e_{n})$. Equivalently, the system is described by the line $\iota(H)$ in the exterior algebra $\ext E$, or by any unit vector of this line. 

We would like to argue that our construction of DLP is related to a \emph{quantum measurement} of the state of the system as a linear subspace of $E$. Let us explain, in general, the framework of quantum measurement that we have in mind (see \cite{BLM} for details). Given a Hilbert space $\mathscr H$, and a measurable space $(G,\mathcal G)$, we call \emph{positive operator-valued measure} a map $O:\mathcal G \to \End(\mathscr H)$ which to each event $B\in \mathcal G$ associates a non-negative self-adjoint operator $O(B)$ on $H$. This map is moreover required to be $\sigma$-additive, and to satisfy $O(G)=\id_{\mathscr H}$. Let us choose such a  positive operator-valued measure $O$. Let us now choose a \emph{density of states} $\rho$ on $\mathscr H$, that is, a non-negative self-adjoint operator with unit trace. Then the map $\mu:B\mapsto \Tr(O(B)\rho)$ defines a probability measure on $\mathcal G$ and $\mu(B)$ is interpreted as the probability, when the system is in the statistical superposition of states represented by $\rho$, that the observation of $O$ yields a result within $B$.

A positive operator-valued measure $O$ is called a \emph{projector-valued measure} if for every Borel set $G$, the operator $O(G)$ is a projection. When $(G,\mathcal G)$ is the real line with its Borel $\sigma$-field, the spectral theorem sets a one-to-one correspondence between projector-valued measures and self-adjoint operators on $\mathscr H$. In general, however, the framework that we described is slightly more general than the most usual setup of quantum mechanics. 

In our case, the Hilbert space is the exterior algebra $\ext E$ and the measurable space is $\Gr(E,\sigma)$ with its Borel $\sigma$-field. The observable is simply the map which to a Borel set $B$ associates the operator
\begin{equation}\label{eq:defOs}
O_{\sigma}(B)=\int_{B} \proj{\iota(Q)}\d\nu^{E,\sigma}(Q).
\end{equation}
We will explain how to associate to every kernel $\op{k}$ on $E$ a density of states $\rho_{\op{k}}$. In fact, when $\op{k}<1$, this density of states is given by the formula
\begin{equation}\label{eq:defrhok1}
\rho_{\op{k}}=\det(1-\op{k}) \ext(\op{k}(1-\op{k})^{-1}).
\end{equation}
Then the probability measure associated to $O_{\sigma}$ and to $\rho_{\op{k}}$ is exactly the measure $\mu_{\sigma,\op{k}}$.

The non-negative operator $\op{l}=\op{k}(1-\op{k})^{-1}$ gives their name to so-called $\op{l}$-ensembles (see \cite{BorodinRains}). We have, still under the condition $\op{k}<1$, the alternative expression
\begin{equation}\label{eq:defrhok2}
\rho_{\op{k}}=\det(1+\op{l})^{-1} \ext \op{l}.
\end{equation}

Our actual definition of the density of states $\rho_{\op{k}}$, valid even if $1$ is an eigenvalue of $\op{k}$, is unfortunately less straightforward than \eqref{eq:defrhok1} and $\eqref{eq:defrhok2}$ would suggest, and we will devote the next few sections to this definition.

The content of Sections \ref{sec:Hodge} and \ref{sec:adjugate} may seem a bit technical. It is however needed to associate to each kernel a density of states (to be defined in \eqref{eq:defrhok} below) and, in Section \ref{sec:kernel} for the interpretation of DLP in terms of quantum measurements with values in the Grassmannian.

\subsection{The Hodge operator} \label{sec:Hodge} Let us resume our investigation of the exterior algebra of an inner product space $E$ of dimension $d$.

The algebra of endomorphisms of the exterior algebra of an inner product space carries an involution which will play an important role for us, and that is essentially the conjugation by the Hodge operator. Our experience is that this involution needs to be described with some care, especially in the complex case, and that is what we do in this paragraph. 

The Hodge operator is usually defined on the exterior algebra of an oriented Euclidean real vector space. In such a space $E$, the exterior product of the elements of a positively oriented orthonormal basis does not depend on the choice of this basis, and singles out a non-zero element of $\ext^{d}E$ that we denote by $\codet_{E}$. The Hodge star operator, denoted by $\hodge$, is then defined as the unique linear endomorphism of the tensor algebra of $E$ which, for each $k\in \{0,\ldots,d\}$, sends $\ext^{k}E$ into $\ext^{d-k}E$ in such a way that for all $x,y\in \ext^{k}E$, 
\begin{equation}\label{eq:Hodge1}
 x\wedge {\star}y=\langle y,x\rangle \,\codet_{E}.
 \end{equation}
Concretely, given a positively oriented orthonormal basis $(e_{1},\ldots,e_{d})$ of $E$, and with the notation introduced in Section \ref{subsec:extalg}, we have $\text{codet}_{E}=e_{\{1,\ldots,d\}}$ and the Hodge star operator sends, for each subset $I\subset \{1,\ldots,d\}$ containing $k$ elements, the vector $e_{I}$ to the vector 
\[\hodge e_{I}=(-1)^{\frac{k(k+1)}{2}+\sum_{i\in I}i} e_{I^{c}},\] 
where $I^{c}=\{1,\ldots,d\}\setminus I$. We will denote by $(-1)^{I}$ the sign which appears in this equality.

This construction of the Hodge operator depends on the orientation of the vector space; indeed choosing the other orientation of $E$ would result in the Hodge star operator being multiplied by~$-1$. The only role of the orientation is in fact to allow us to decide which of the two elements of norm $1$ in $\ext^{d}E$ we call $\codet_{E}$. In the case where $E$ is a complex inner product space however, the unit sphere of $\ext^{d}E$ is a circle. There does not seem to exist a notion of orientation in this case\footnote{One could define an orientation of a $d$-dimensional complex space as an orbit of $\SL_{d}(\C)$ on the set of bases of this space, but this does not seem to be a classical notion.} and without it we cannot pick a point on this circle in a natural way.

We go around this problem by taking a slightly more abstract point of view. Let $E$ be a real or complex inner product space. Choose $k\in \{0,\ldots,d\}$. The tensor contraction
\begin{align*}
(E^{*})^{\otimes d}\otimes E^{\otimes k} & \longrightarrow (E^{*})^{\otimes (d-k)}\\
(\phi_{1}\otimes \ldots \otimes \phi_{d}) \otimes (v_{1}\otimes \ldots \otimes v_{k}) & \longmapsto \phi_{1}(v_{1})\ldots \phi_{k}(v_{k})\, \phi_{k+1}\otimes \ldots \otimes \phi_{d}
\end{align*}
restricts to a linear mapping $\ext^{d}E^{*}\otimes \ext^{k}E\to \ext^{d-k} E^{*}$ and, by taking the direct sum over $k$, to a linear mapping
\[\kappa:\ext^{d}E^{*}\otimes \ext E\longrightarrow \ext  E^{*}.\]
Given a basis $(e_{1},\ldots,e_{d})$ of $E$ and the dual basis $(\epsilon_{1},\ldots,\epsilon_{d})$ of $E^{*}$, one has for all $I\subset \{1,\ldots,d\}$ the equality
\[\kappa(\epsilon_{\{1,\ldots,d\}} \otimes e_{I})=(-1)^{I} \epsilon_{I^{c}},\]
from which it follows that $\kappa$ is surjective, hence an isomorphism. Let us emphasise that the definition of $\kappa$ does not depend on the inner product on $E$. 

We will now make use of this inner product, in the guise of the antilinear isomorphism
\begin{align*}
c : E &\longrightarrow E^{*}\\
v&\longmapsto c(v)=\langle v,\cdot \rangle.
\end{align*}
In the real case, the map $c$ is a linear isomorphism. In the complex case, because we take the Hermitian inner product to be linear in the second variable,  $c$ is antilinear, in the sense that for all $\lambda\in \C$ and $v\in E$, one has $c(\lambda v)=\bar\lambda c(v)$.

The map $c$ extends to an antilinear isomorphism of algebras $\ext c:\ext E \to \ext E^{*}$ and we define the Hodge operator $\hodge : \ext E \to \ext^{d}E^{*}\otimes \ext E$ as :
\begin{equation}\label{eq:defHodge}
\hodge= \kappa^{-1}\circ \ext c : \ext E \stackrel{\ext c}{\longrightarrow}  \ext E^{*}\stackrel{\kappa^{-1}}{\longrightarrow} \ext^{d}E^{*}\otimes \ext E\simeq \Hom(\ext^{d} E, \ext E).
\end{equation}
Let us emphasise that the map $\hodge$ is antilinear. 

The usual Hodge operator, in the oriented Euclidean case, is obtained by composing the map~$\hodge$ thus defined by the evaluation at the element $\codet_{E}\in \ext^{d}E$. Our construction takes the orientation as a variable instead of a given parameter, and this point of view suits both the real and complex cases. The formula which replaces \eqref{eq:Hodge1} is the following: for all $k\in \{0,\ldots,d\}$, all $x,y\in \ext^{k} E$ and all $z\in \ext^{d}E$,
\begin{equation}\label{eq:Hodge2}
x\wedge (\hodge y)(z) = \langle y, x \rangle z.
\end{equation}

In coordinates, given an orthonormal basis $(e_{1},\ldots,e_{d})$ of $E$ and the dual basis $(\epsilon_{1},\ldots,\epsilon_{d})$ of~$E^{*}$, we have for all $I\subset \{1,\ldots,k\}$ the equality
\[\hodge e_{I}=(-1)^{I}\epsilon_{\{1,\ldots,d\}}\otimes e_{I^{c}}.\]
This formula shows that the Hodge operator is invertible. 

\subsection{The adjugate endomorphism} \label{sec:adjugate} We will use the Hodge operator to make the following important construction: to each linear endomorphism $f$ of $\ext E$, we associate the adjugate endomorphism $\adj(f)$ of $\ext E$ by setting
\begin{equation}\label{eq:defadj}
\adj(f) = \big(\hodge^{-1} \circ (\id_{\ext^{d}E^{*}} \otimes f) \circ \hodge\big)^{*}.
\end{equation}
This definition may look unappealing, but we shall soon see that it corresponds to the operation which to a matrix associates the transpose of its cofactor matrix. At least, we see on this definition that the adjugation map  is linear on $\End(\ext E)$. We see also that it is an antimorphism of algebra, in the sense that for all $f,g\in \End(\ext E)$, we have
\begin{equation}\label{eq:antimorph}
\adj(f\circ g)=\adj(g)\circ  \adj(f).
\end{equation}

Let us compute the adjugate endomorphism in coordinates. Let again $(e_{1},\ldots,e_{d})$ be an orthonormal basis of $E$. For all subsets $I,J\subset \{1,\ldots,d\}$, it follows from unfolding the definitions that
\begin{equation}\label{eq:adjugatecoord}
\langle e_{I},\adj(f) e_{J}\rangle=(-1)^{I}(-1)^{J}\langle e_{J^{c}},fe_{I^{c}}\rangle.
\end{equation}

From this relation, we deduce for instance that 
\[\adj(\adj(f))=\ext (-1)^{d+1}\circ  f \circ \ext (-1)^{d+1}.\]

The following observation will be useful for us.

\begin{lemma}\label{lem:tildeconj} Let $f$ be a linear endomorphism of $\ext E$ and $r\in \U(\ext E)$ an isometry of $\ext E$. Then $r \circ \adj(f) \circ r^{*}=\adj(r\circ f\circ r^{*})$.
\end{lemma}

This property reflects the fact that the construction of the adjugate endomorphism uses only the inner product on $E$.

\begin{proof} Let us compute $r\adj(f) r^{-1}$:
\[r \adj(f) r^{-1}=r(\hodge^{-1} (\id_{\ext^{d}E^{*}}\otimes f)\hodge)^{*}r^{-1}=(r\hodge^{-1} (\id_{\ext^{d}E^{*}}\otimes f)\hodge r^{-1})^{*}.\]
From \eqref{eq:Hodge2}, we deduce that $\hodge r^{-1}=\det(r) (\id_{\ext^{d}E^{*}}\otimes r^{-1})\hodge$. Taking the inverse of both sides, we find $r\hodge^{-1}=\hodge^{-1} (\id_{\ext^{d}E^{*}}\otimes r) \det(r^{-1})$. Thus,
\[r \adj(f) r^{-1}=(\hodge^{-1} (\id_{\ext^{d}E^{*}}\otimes rfr^{-1}) \hodge)^{*}=\adj(rfr^{-1}),\]
the desired equality.
\end{proof}

Let us make a few remarks about the relation between adjugation and inversion. Let us consider an endomorphism $\op{a}$ of $E$. We want to compute the adjugate of $\ext \op{a}$. 

\begin{proposition} \label{prop:adjugate} Let $E$ be an inner product space. Let $\op{a}$ be an endomorphism of $E$. Then
\[\adj\big(\ext \op{a}\big) \circ \ext \op{a}=\det (\op{a}) \id_{\ext E}.\]
In particular, if $\op{a}$ is invertible, then
\[\adj\big(\ext \op{a}\big)=\det (\op{a}) \ext \op{a}^{-1}.\]
\end{proposition}

\begin{proof} The endomorphisms $\ext \op{a}$ and $\adj\big(\ext \op{a}\big)$ of $\ext E$ preserve the degree. Hence, the first equality follows from the fact that for all $k\in \{0,\ldots,d\}$, all $x,y\in \ext^{k} E$ and all $z\in \ext^{d}E$, we have
\begin{align*}
\langle y, \adj\big(\ext \op{a}\big)\circ \ext u (x)\rangle z&= \langle \hodge^{-1} (\id_{\ext^{d}E^{*}} \otimes \ext \op{a})\hodge y, \ext u (x)\rangle z\\
&=\ext \op{a}(x) \wedge \big((\id_{\ext^{d}E^{*}}\otimes \ext \op{a}) (\hodge y)\big)(z)\\
&=\ext \op{a}(x) \wedge \ext \op{a} ((\hodge y) (z))\\
&=\det (\op{a}) x\wedge (\hodge y)(z)\\
&=\langle y, \det (\op{a}) x\rangle z.
\end{align*}
The second equality follows immediately from the first and the fact that, if $\op{a}$ is invertible, $\ext(\op{a}^{-1})=(\ext \op{a})^{-1}$.
\end{proof}

We can now prove the following useful formula.

\begin{lemma}\label{lem:trdetab} Let $\op{a}$ and $\op{b}$ be two endomorphisms of $E$. Then
\[\Tr_{\ext E}\Big(\adj\big(\ext \op{a}\big)\circ  \ext \op{b} \Big)=\det(\op{a}+\op{b}).\]
\end{lemma}

\begin{proof} The endomorphism $\op{b}$ being fixed, both sides of the equality to prove are polynomial functions of $\op{a}\in \End(E)$. It is thus enough to prove that the equality holds when $\op{a}$ is invertible. In this case, using Proposition \ref{prop:adjugate} and \eqref{eq:trdet}, we find
\begin{align*}
\Tr_{\ext E}\Big(\adj\big(\ext \op{a}\big)\circ \ext \op{b} \Big)&=(\det \op{a}) \Tr_{\ext E}\big(\ext \op{a}^{-1} \ext \op{b})\\
&=(\det \op{a})\Tr_{\ext E}\big(\ext (\op{a}^{-1}\op{b}))\\
&=\det (\op{a}) \det(1+\op{a}^{-1}\op{b}),
\end{align*}
and the result follows.
\end{proof}

\subsection{DLP as positive operator-valued measures}\label{sec:kernel} To each kernel on $E$, we associate a density matrix on the exterior algebra of $E$, that is, a non-negative self-adjoint operator of trace~$1$.

\begin{proposition}\label{prop:KrhoK} Let $\op{k}$ be a kernel on an inner product space $E$. The operator
\begin{equation}\label{eq:defrhok}
\rho_{\op{k}}=\adj\big(\ext(1-\op{k})\big)\circ \ext \op{k}=\ext \op{k}\circ \adj\big(\ext(1-\op{k})\big)
\end{equation}
is self-adjoint and non-negative on $\ext E$, and it satisfies $\Tr_{\ext E}(\rho_{\op{k}})=1$.

Moreover, if $\op{k}$ is the orthogonal projection on a linear subspace $H$ of $E$, then $\rho_{\op{k}}$ is the orthogonal projection on the line $\iota(H)$:
\begin{equation}\label{eq:densityproj}
\rho_{\proj{H}}=\proj{\iota(H)}.
\end{equation}
\end{proposition}

\begin{proof} Let $(e_{1},\ldots,e_{d})$ be an orthonormal basis of $E$ formed with eigenvectors of $\op{k}$, such that for all $i\in \{1,\ldots,d\}$ we have $\op{k}e_{i}=\lambda_{i}e_{i}$, with $\lambda_{i}\in [0,1]$. Then, for all $I\subset \{1,\ldots,d\}$, we find, using \eqref{eq:adjugatecoord},
\[\adj\big(\ext(1-\op{k})\big) \circ \ext \op{k} (e_{I})=\prod_{i\in I} \lambda_{i} \prod_{i\notin I} (1-\lambda_{i}) e_{I}=\ext \op{k}\circ \adj\big(\ext(1-\op{k})\big) (e_{I}).\]
All the claimed properties of $\rho_{\op{k}}$ follow readily from these equalities. The fact that $\rho_{\op{k}}$ has trace~$1$ can also be deduced in a basis-free way from Lemma \ref{lem:trdetab}.
\end{proof}

We can now prove a result which relates the measure $\mu_{\sigma,\op{k}}$ and the density of states $\rho_{\op{k}}$.

\begin{proposition}\label{prop:lienDLPQM} Let $(E,\sigma)$ be a split inner product space. Let $\op{k}$ be a kernel on $E$. Let $Q$ be an element of $\Gr(E,\sigma)$. Then
\begin{equation}
\det(\op{k}\proj{Q}+(1-\op{k})\proj{Q^{\perp}})=\Tr_{\ext E}(\proj{\iota(Q)}\rho_{\op{k}} )\,.
\end{equation}
More generally, let $\op{l}$ be another kernel on $E$. Then
\begin{equation}
\det(\op{k}\op{l}+(1-\op{k})(1-\op{l}))=\Tr_{\ext E}(\rho_{\op{k}}\rho_{\op{l}})\,.
\end{equation}
\end{proposition}

\begin{proof} Thanks to the last assertion of Proposition \ref{prop:KrhoK}, the first equality is a consequence of the second, in the special case where $\op{l}=\proj{Q}$. The second equality, in turn, follows from the definition \eqref{eq:defrhok} of $\rho_{\op{k}}$ and $\rho_{\op{l}}$, from \eqref{eq:antimorph} and from Lemma \ref{lem:trdetab}. 
\end{proof}

\begin{corollary}\label{cor:loiDLPext} Let $\Q$ be a determinantal linear process of $(E,\sigma)$ with kernel $\op{k}$. Let $O_{\sigma}$ be the positive operator-valued measure defined by \eqref{eq:defOs}. Then for every Borel subset $B$ of $\Gr(E,\sigma)$, one has the equality
\begin{equation}\label{eq:DLP-POVM}
\P(\Q\in B)=\Tr_{\ext E}(O_{\sigma}(B)\rho_{\op{k}})\,.
\end{equation}
\end{corollary}

\begin{proof} Indeed, by definition of determinantal linear processes and by Proposition \ref{prop:lienDLPQM}, both sides of the equality to prove are equal to 
\[\int_{B} \det(\op{k}\proj{Q}+(1-\op{k})\proj{Q^{\perp}}) \d\nu^{E,\sigma}(Q)\,.\]
\end{proof}

\subsection{Some results revisited}

Having expressed the distribution of determinantal linear processes as we did in Corollary \ref{cor:loiDLPext}, we can reformulate  in a more concise way some of their properties, in particular those concerning the support of the distribution. Recall that we denote by $\proj{\ul{n}}$ the orthogonal projection on the subspace $\ext^{\ul{n}}E$ of the exterior algebra of $E$, see Section~\ref{sec:extalgsplit}.

\begin{proposition} \label{prop:supportextalg} Let $(E,\sigma)$ be a split inner product space. Let $\op{k}$ be a kernel on $E$. For all $\ul{n}\in [\![\ul{0},\ul{d}]\!]$, we have
\[\mu_{\sigma,\op{k}}(\Gr_{\ul{n}}(E,\sigma))=\Tr_{\ext E}(\proj{\ul{n}}\rho_{\op{k}}).\]
In particular, if $\op{k}$ is the orthogonal projection on a linear subspace $H$ of $E$, and if $\omega_{H}$ is a unit vector of the line $\iota(H)$, then for all $\ul{n}\in [\![\ul{0},\ul{d}]\!]$,
\[\mu_{\sigma,\op{k}}(\Gr_{\ul{n}}(E,\sigma))=\| \proj{\ul{n}} \omega_{H}\|^{2}.\] 
\end{proposition}

Note that Pythagoras' theorem, stating that $\sum_{\ul{n}\in [\![\ul{0},\ul{d}]\!]} \| \proj{\ul{n}} \omega_{H}\|^{2}=\| \omega_{H}\|^{2}=1$, provides a proof of \eqref{eq:DPPPyth} in the DPP case.

\begin{proof} In view of Corollary \ref{cor:loiDLPext}, the first assertion will follow from the equality
\[O_{\sigma}(\Gr_{\ul{n}}(E,\sigma))=\proj{\ul{n}}.\]
To prove it, we need to compute the integral
\begin{equation} \label{eq:intpiQ}
\int_{\Gr_{\ul{n}}(E,\sigma)} \proj{\iota(Q)} \d\nu^{E,\sigma}(Q).
\end{equation}
For this, we used a refined version of \eqref{eq:intnunsigma}: we choose an orthonormal basis $(e_{1},\ldots,e_{d})$ of $E$ adapted to $\sigma$ and consider, for each subset $J$ of $\{1,\ldots,d\}$ such that $(e_{j} : j\in J)$ contains $n_{1}$ vectors in $E_{1}$, $n_{2}$ vectors in $E_{2}$ and so on, the orthogonal projection on the line $\iota(\Vect(e_{j}:j\in J))$. The sum over all such subsets $J$ of these projections is the sum of the orthogonal projections on a set of lines forming a splitting in lines of the subspace $\ext^{\ul{n}}E$ of $\ext E$. It is thus equal to $\proj{\ul{n}}$. Averaging over the action of the group $\U(E,\sigma)$ on the set of orthonormal bases of $E$ adapted to~$\sigma$, we find that the integral \eqref{eq:intpiQ} is equal to $\Pi^{\ul{n}}$.

The second assertion follows immediately from the first and the fact that, in this case, $\rho_{\op{k}}$ is the orthogonal projection on $\iota(H)$.
\end{proof}

Another result that is more easily expressed in terms of the exterior algebra is Theorem \ref{prop:meanprojection}.

\begin{theorem}\label{thm:wedgeA}
With the notation of Theorem \ref{prop:meanprojection}, we have the equality
\[\E\Big[\ext \op{P}^{\Q}\Big]=\ext\proj{H}.\]
\end{theorem}

The part of this statement concerning the part of degree $1$ of the exterior algebra, namely the fact that $\E\big[\op{P}^{\Q}\big]=\proj{H}$, appeared in the literature before (see the paragraph before Theorem \ref{prop:meanprojection}), but to the best of our knowledge, the full statement is new.

Thanks to the self-adjointness of $\proj{H}$, one can replace, in the equality, the operator $\op{P}^{\Q}$ by its adjoint, which is the projection on $H$ parallel to $\Q^{\perp}$.

\begin{proof} Indeed, the statement of Theorem \ref{prop:meanprojection} can be written, using Proposition \ref{prop:lienDLPQM}, as
\[\E\Big[\Tr_{\ext E}\big(\ext \op{a} \ext \op{P}^{\Q}\big)\Big]=\Tr_{\ext E}\big(\ext \op{a}\ext\proj{H}\big)\]
and the fact that this holds for every endomorphism $\op{a}$ of $E$ implies the announced equality.
\end{proof}

%%%%%%%%%%%%%%%%%%%%%%%%%%%%%%%%%%%%
%%%%%%%%%%%%%%%%%%%%%%%%%%%%%%%%%%%%
%%%%%%%%%%%%%%%%%%%%%%%%%%%%%%%%%%%%

\section{Changing coefficient field and the quaternion case} \label{sec:quat}

In this section, we explain how the previous construction adapts to the case of quaternionic vector spaces (Definition \ref{def:dpmg-quaternion} and Theorem \ref{thm:qDLP}). Some care is needed in defining an appropriate notion of determinant which we review below in Section \ref{sec:nc-det}. Before doing so, we start by studying how changing the coefficient field from complex numbers to real numbers changes the DLP (Proposition \ref{prop:complex-reals}). Then we treat the analogous change from the quaternionic to the complex case (Proposition~\ref{prop:quaternions-complex}). We also detail some properties of quaternionic DLP in Proposition \ref{prop:LaplaceQ} and Theorem \ref{thm:qA}.

\subsection{From complex to real coefficient field}\label{sec:complex-real}

We have so far defined and studied determinantal probability measures on Grassmannians of real and complex vector spaces. Although the specification of the underlying coefficient field seems to make little difference\footnote{However, there are more determinantal probability measures in the complex case than in the real case. This is related to the algebraic question of the equivalence of kernels: {\em When do kernels $\op{k}$ and $\op{k'}$ define the same probability measures $\mu_{\sigma,\op{k}}=\mu_{\sigma,\op{k'}}$?}, which boils down to a question about determinantal varieties.} in the classical case of determinantal point processes, it does here. 

Indeed, if $E$ is a complex vector space split in lines, any kernel induces a discrete determinantal probability measure on the complex Grassmannian of $E$. However, when~$E$ is viewed as a real vector space $\hat{E}$, the induced splitting consists of real $2$-dimensional blocks, and unless the kernel has eigenspaces consisting exactly of sums of these blocks (i.e. unless the eigenspace decomposition of the kernel is coarser than the given splitting of $E$), the induced determinantal probability measure will have a continuous support in the even-dimensional real Grassmannian of $\hat{E}$. 

Let us more generally introduce the following notation. Given a split complex vector space $(E,\sigma)$, let $\hat{E}$ be the real vector space obtained by restriction of scalars, let $\hat{\sigma}$ be the corresponding splitting of $\hat{E}$, and if $\op{k}$ is a kernel on $E$, let $\hat{\op{k}}$ be the corresponding kernel on $\hat{E}$. This correspondence is canonical. 

The point is that the measure $\mu_{\hat{\sigma},\hat{\op{k}}}$ is \emph{not} the pushforward of $\mu_{\sigma,\op{k}}$ under the natural inclusion map $\Gr(E)\subset\Gr(\hat{E})$. 
In general, it seems that these two measures are not easily related. However, one can compare some amount of information contained in both probability measures as shown by the following proposition. 

\begin{proposition}\label{prop:complex-reals}
Let $E$ be a complex vector space given with a splitting $\sigma=(E_1, \ldots, E_s)$ and a kernel $\op{k}$. Let $\hat{E}$, $\hat{\sigma}$, and $\hat{\op{k}}$ be the data induced by restricting the coefficient field to real numbers. Let $\rC^{(1)}$ and $\rC^{(2)}$ be two independent random complex linear subspaces of~$E$ sampled from $\mu_{\sigma,\op{k}}$ and let $\rR$ be a random real linear subspace of $\hat{E}$ sampled from $\mu_{\hat{\sigma},\hat{\op{k}}}$. Then we have the equality in distribution
\[\left(\dim_{\R} \rR_i\right)_{1\leq i \leq s}\stackrel{\text{\rm(law)}}{=}\left(\dim_\C \rC^{(1)}_i+\dim_\C \rC^{(2)}_i\right)_{1\leq i \leq s}\,.\]
\end{proposition}

The proof of Proposition \ref{prop:complex-reals} is based on Lemma~\ref{lem:MMhat} below which allows to compare determinants of complex matrices seen as complex or real matrices. In order to state the latter proposition, we introduce some notation. We will make use of a choice of basis in order to represent matrices. Let us introduce the real matrix \[\II=\begin{pmatrix}0 &-1 \\1 &0 \end{pmatrix}\,.\] 
The algebra of complex numbers may be realised as the subalgebra of $M_2(\R)$ consisting in matrices of the form $a I_2+ b \II$ with $a,b$ real numbers, where $I_2$ is the identity matrix.

For any $d\ge 1$ and all matrices $M$ in $M_d(\C)$, let $\hat{M}$ denote the matrix in $M_{2d}(\R)$ obtained by replacing each entry by the $2\times 2$ real matrix given by the above identification.

\begin{lemma}\label{lem:MMhat} Let $M$ be a complex matrix. Then $\det \hat M=|\det M|^{2}$.
\end{lemma}

\begin{proof} For a diagonal matrix, the result holds by inspection. Since the map $M\mapsto \hat M$ is a morphism of algebras from $M_{d}(\C)$ to $M_{2d}(\R)$, the result extends to diagonalisable matrices. Finally, it extends to $M_{d}(\C)$ by an argument of density.
\end{proof}

In particular, if $\op{k}$ is a Hermitian endomorphism of $E$, then 
$\det(\hat{\op{k}})=\det(\op{k})^2$.
We are now ready to prove Proposition \ref{prop:complex-reals}.

\begin{proof}[Proof of Proposition \ref{prop:complex-reals}]
We compute the Laplace transform of the left-hand side using Proposition~\ref{prop:dimQ} and observe that it can be rewritten, thanks to Lemma~\ref{lem:MMhat}, as the square of a determinant. This second determinant is the Laplace transform of a summand of the right-hand side by Proposition~\ref{prop:dimQ} again, which concludes the proof.
\end{proof}

\subsection{Quaternions}

For self-containedness, we review some basics about quaternions.

In the following, we let $\bH=\{a+b\ii+c\jj+d\kk\,\vert\, a,b,c,d\in\R\}$ be the four-dimensional real division algebra of \emph{real quaternions}, generated by three elements $\ii,\jj,\kk$ subject to the relations 
$$\ii^2=\jj^2=\kk^2=\ii\jj\kk=-1\,.$$

We will abuse terminology slightly and call \emph{quaternionic vector space} any right-module on~$\bH$. Because $\bH$ is a division algebra, there is a well-defined notion of basis and hence of dimension.

The coefficient $a$ of $q=a+b\ii+c\jj+d\kk$ is called the real part of $q$, denoted $\Re(q)$. Furthermore, let us define the anti-involution by conjugation $q=a+b\ii+c\jj+d\kk\mapsto \overline{q}=a-b\ii-c\jj-d\kk$; in particular, we have $q+\overline{q}=2\Re(q)$. 

On the quaternionic vector space $\bH^{d}$, there is the standard inner product
\[\langle (q_{1},\ldots,q_{d}),(r_{1},\ldots,r_{d}) \rangle=\sum_{i=1}^{d} \overline{q_{i}}r_{i}\,.\]
On a quaternionic vector space $E$, an inner product is a map $\langle \cdot, \cdot \rangle : E\times E \to \bH$ which is additive in each variable, quaternion-sesquilinear in the sense that for all $v,w\in E$ and all $q,r\in \bH$,
\[\langle vq,wr\rangle=\overline{q}\langle v,w\rangle r.\]
It is moreover required to be such that $\langle v,v\rangle$ is a non-negative real number for every $v\in E$, and~$0$ only if $v=0$.

The group of isometries of such a quaternionic inner product space is called the group of \emph{symplectic transformations}, and we keep the notation $\U(E)$ for it.

Splittings of quaternionic inner product spaces are defined exactly as in the real and complex case, as well as Grassmannians and split Grassmannians. If $\sigma$ is a splitting of $E$, the group of isometries preserving $\sigma$ is denoted by $\Gr(E,\sigma)$ and each connected component of $\Gr(E,\sigma)$ is acted on transitively by $\U(E,\sigma)$. This gives rise to the invariant measure $\nu^{E,\sigma}$. 

The relation between probability measures on $\Gr(E,\sigma)$ and their incidence measures relies on the lattice structure of $(\Gr(E,\sigma),\subseteq)$ and is thus the same as in the real or complex case. 

We will construct a probability measure on $\Gr(E,\sigma)$ which is equivariant under the action of~$\U(E,\sigma)$, and which is determinantal in the sense of~\eqref{eq:incidenceDSP}. 

This probability measure is determined by a kernel, that is, a self-adjoint endomorphism of our quaternionic inner product space, the (right-)eigenvalues of which are reals between $0$ and $1$. 

The main difficulty in generalising \eqref{eq:incidenceDSP} to the quaternionic case is to clarify which notion of determinant has to be used.

\subsection{$\tau$-determinants and quaternion determinant}\label{sec:nc-det}

\subsubsection{$\tau$-determinants}\label{sec:tau-det}

When dealing with matrices in a non-necessarily commutative ring $R$, a simple combinatorial approach to determinants suggests the following candidate~\cite{KL2}. Consider a commutative ring $A$ and a trace map $\tau:R\to A$, i.e. an additive map such that for all $r,r'\in R$, we have $\tau(rr')=\tau(r'r)$. For all integers $d\geq 1$ and all $d\times d$ matrices $M\in M_{d}(R)$, we define the \emph{$\tau$-determinant} of $M$ as the element of $A$ given by the formula
\begin{equation}\label{eq:deftau}
{\det}_{\tau}(M)=\sum_{\sigma\in \S_{d}} \epsilon(\sigma) \prod_{\substack{c \text{ cycle of } \sigma \\ c=(i_{1}\ldots i_{r})}} \tau(M_{i_{1}i_{2}}M_{i_{2}i_{3}}\ldots M_{i_{r}i_{1}})\,,
\end{equation}
where $\S_{d}$ denotes the symmetric group on $d$ elements.

The $\tau$-determinant is obviously $\Z$-linear in columns of the matrix, and if a column vanishes, it vanishes. But in general, this $\tau$-determinant does not have further especially nice properties. In particular, if a matrix $M\in M_d(R)$ is singular, in the sense that there exists a non-zero vector $X\in R^d$ such that $MX=0$ or $^{t}XM=0$, we need not have ${\det}_{\tau}(M)=0$. In particular, this $\Z$-multilinear operation need not be alternating. Moreover, it need not be multiplicative. 

\subsubsection{Quaternion determinant}

However, in the case $R=\bH$, $A=\R$ and $\tau=\Re$ (the real part map), this is the so-called $Q$-determinant, introduced by Moore~\cite{Moore}\footnote{The work of Moore was published only posthumously, so we invite the reader to consult the very interesting \cite{Dyson}, or see \cite[Section 5.1]{Mehta} for a more detailed textbook treatment}, which, when restricted to the space of quaternion Hermitian matrices, does have very nice properties. For \ae sthetical reasons, rather than writing $\qdet$ as done by some authors, we will denote it $\det_\Re$ in what follows.

It seems that no definition of determinants on a noncommutative ring enables it to keep all properties we are accustomed to from usual determinants. Somewhat magically, on the set of quaternion Hermitian matrices, most definitions of noncommutative determinants do coincide (up to normalisation) and we can use this to transfer properties more apparent from other definitions~\cite{Aslaksen,GRW}. This implies the following important property, which we highlight as we will use it throughout in the rest of this section.

\begin{proposition}\label{prop:qdet-op} Let $E$ be a quaternionic inner product space. Let $\op{k}$ be a self-adjoint endomorphism of $E$. Let $K$ be the matrix of $\op{k}$ in an orthonormal basis of $E$. Then $\det_{\Re}(K)$ does not depend on the chosen basis. 
\end{proposition}

We denote the common value of the determinants in the above proposition by $\det_{\Re}(\op{k})$. 

\begin{proof}
If $K$ and $K'$ are the matrices of $\op{k}$ in two orthonormal bases, then there is a quaternion unitary matrix $U$ satisfying $K=UK'U^{-1}$. 
As explained before the statement of Proposition~\ref{prop:qdet-op}, one of the many definitions of noncommutative determinants is actually multiplicative, and coincides on quaternion Hermitian matrices with the Q-determinant $\det_\Re$. Therefore, $\det_\Re(K)=\det_\Re(K')$.\end{proof}

Similarly to the complex-to-real mapping discussed in Section \ref{sec:complex-real} above, we introduce both a canonical and a non-canonical correspondence when restricting scalars from quaternions to complex numbers. First of all, if $E$ is a quaternionic vector space, $\sigma$ a splitting, and $\op{f}$ an endomorphism of $E$, we denote by $\hat{E}$, $\hat{\sigma}$, and $\hat{\op{f}}$ the corresponding data obtained by restriction of scalars to complex numbers. If $\op{f}$ is self-adjoint, then so is $\hat{\op{f}}$.

Moreover, we introduce a map $M\mapsto \hat{M}$ from $d\times d$ quaternionic matrices to $2d\times 2d$ complex matrices, given an identification $\bH\simeq\C^2$.
A way of doing this by choosing bases, is to use the following classical representation of $\ii$, $\jj$, and $\kk$ by the complex matrices:
\begin{eqnarray*}
\II=\begin{pmatrix}0 &-1 \\1 &0 \end{pmatrix},\quad \JJ=\begin{pmatrix}0 & i\\i & 0\end{pmatrix},\, \mbox{and}\quad \KK=\begin{pmatrix} -i&0 \\0 & i\end{pmatrix}\,.
\end{eqnarray*}

\begin{proposition}\label{prop:Qdet-square}
Let $K$ be a quaternion Hermitian matrix. We have
\begin{equation}\label{eq:Qdet-pfaff}
\det(\hat{K})={\det}_{\Re}(K)^{2}.
\end{equation}
\end{proposition}
\begin{proof}
Although the identity is purely algebraic and could in principle be proved combinatorially, we resort to a spectral theorem, also valid in the quaternion case (see~\cite[Theorem~E.11]{Guionnet}, \cite[Theorem 3.3]{Farenick}, or~\cite{Helgason} for a statement in a very general setup). Write $K=UDU^{-1}$ with $D$ real-valued and diagonal and $U$ quaternion Hermitian. Then $\hat{K}=\hat{U}\hat{D}\hat{U}^{-1}$ and clearly $\det(\hat{K})=\det(\hat{D})={\det}_{\Re}(D)^2$ since $D$ is real-valued. Now we conclude by using the fact mentioned earlier in this section (Proposition \ref{prop:qdet-op}), that ${\det}_{\Re}(D)={\det}_{\Re}(K)$ .
\end{proof}

\subsection{Determinantal linear processes (DLP) on quaternionic Grassmannians}

Now that we have a working definition for the determinant of a quaternion Hermitian endomorphism at hand, we are ready to give the definition of determinantal probability measures on quaternionic Grassmannians, and show their existence.\footnote{In view of Dyson's results -- see Proposition \ref{prop:Dyson} below -- it makes little doubt that quaternionic DPP were known to Dyson and Mehta in the context of the Gaussian symplectic ensemble. Let us however mention that they were yet again formally introduced in \cite{Kargin} using the consequence of \eqref{eq:Qdet-pfaff} which gives the known relation between the Q-determinant of a quaternion Hermitian matrix and the Pfaffian of an associated complex antisymmetric matrix. Our existence proof bypasses the use of Pfaffians by working directly with the Q-determinant. Although quaternionic DPP are a special instance of Pfaffian processes, they enjoy strong additional symmetry properties which we think justify that we treat them separately, in as close a determinantal way as possible. Yet Dyson, whose pioneering ideas we rely on, passed the provocative and somewhat discouraging sentence: ``It is unlikely that the applications of quaternion-determinants to physics will ever be important.'' See~\cite{Dyson} for the full commentary.}

\begin{definition}[Quaternionic DLP] \label{def:dpmg-quaternion} Let $E$ be a quaternionic inner product space given with a splitting $\sigma$ and a kernel $\op{k}$. We say that a probability measure $\mu_{\sigma,\op{k}}$ on $\Gr(E,\sigma)$ is determinantal with kernel~$\op{k}$ if 
\begin{equation}
\d (Z\mu_{\sigma,\op{k}})(R)={\det}_{\Re}(\op{k}^{R}_{R}) \d\nu^{E,\sigma}(R).
\end{equation} 
\end{definition}

Our goal is to prove the following theorem.

\begin{theorem}\label{thm:qDLP} Let $E$ be a quaternionic inner product space given with a splitting $\sigma$ and a kernel~$\op{k}$. 
There exists a unique DLP on $\Gr(E,\sigma)$ with kernel $\op{k}$.
\end{theorem}

As in the real and complex case, uniqueness follows immediately from the M\"obius inversion formula (Proposition \ref{prop:incmeaschar}).

To prove existence, we propose an alternative construction to the more intrinsic and geometric approach presented in Section~\ref{sec:DLP}. 
This construction is more concrete in that it builds on the notion of quaternionic DPP. In order to construct these DPP, we follow an approach due to Dyson and based on formulas originally proved for Q-determinants, and which, as we will see, hold for $\tau$-determinants in general.\footnote{
For a quaternion Hermitian matrix $K$, Moore observed that $\det_\Re(K)$ coincides with 
\begin{equation}\label{eq:Moore}
\mathrm{Mdet}(K)=\sum_{\sigma\in \S_{d}} \epsilon(\sigma) \prod_{\substack{c \text{ cycle of } \sigma \\ c=(i_{1}\ldots i_{r})}} K_{i_{1}i_{2}}K_{i_{2}i_{3}}\ldots K_{i_{r}i_{1}}\,,
\end{equation}
where the elements in the products are ordered consistently according to the cycle structure of the permutation~$\sigma$, for example by letting each cycle start with its minimal index, called its root, and ordering cycle roots in decreasing order. The equality $\det_\Re(K)=\mathrm{Mdet}(K)$ follows from the fact that $q+\overline{q}=2\Re(q)$ for any $q\in\bH$, and the fact that permutations in the sum can be grouped according to their \emph{unoriented} cycle structure.

The results of Section \ref{sec:facts-taudet} can more generally be proved for the $R$-valued Mdet, provided the diagonal entries~$M_{ii}$ are in the center of $R$.}

\subsubsection{Some facts about general $\tau$-determinants}\label{sec:facts-taudet}
Let $R$ and $\tau$ be as in Section \ref{sec:tau-det} above.
In the following, for any matrix $M=(M_{ij})_{1\le i,j\le d}\in M_d(R)$, any $k\ge 1$, and any ordered multiset $I=(i_1,\ldots,i_k)\in\{1,\ldots,d\}^k$, let $M_I^I$ denote the matrix $(M_{i_ai_b})_{1\le a,b\le k}$. By convention, we also set $M_I^I=1$ when $I=\varnothing$. Moreover, let 
\[\Tr_\tau(M)=\sum_{i=1}^d\tau(M_{ii})\]
denote the \emph{$\tau$-trace} of a matrix $M\in M_d(R)$.

The following statement and proof is the exact analogue of Dyson's result for Q-determinants (see~\cite[Theorem~5.1.4]{Mehta}).\footnote{Its statement and proof are reminiscent of the work of Diaconis and Evans on immanants \cite[Theorem 2.1]{DiaconisEvans}.}

\begin{proposition}\label{prop:Dyson}
Let $M\in M_d(R)$ be such that $M^2=M$. Let $k\ge 1$ and $I=(i_1,\ldots,i_{k-1})\in \{1,\ldots,d\}^{k-1}$ be an ordered multiset. Then
\begin{equation}\label{eq:dyson}
\sum_{i_k=1}^{d}{\det}_{\tau}\left(M_{I\cup(i_k)}^{I\cup(i_k)}\right)=\left(\Tr_\tau(M)-k+1\right){\det}_{\tau}(M_{I}^I)\,.
\end{equation}
\end{proposition} 
\begin{proof}
Note that for $k=1$, this is simply the definition of the $\tau$-trace. Let us hence assume~$k\ge 2$.

We look at the left-hand side of Equation~\eqref{eq:dyson}. 
We consider the set of permutations in~$\S_k$ and partition it according to whether $k$ is a fixed point or not. 

If $\sigma$ is a permutation which fixes $k$, we can factor $\tau(M_{i_ki_k})$ in each of the terms ${\det}_{\tau}\left(M_{I\cup(i_{k})}^{I\cup(i_{k})}\right)$. Summing over $i_k\in\{1,\ldots,d\}$ yields a total contribution, from all permutations fixing $k$, of~$\Tr_\tau(M)$ times the $\tau$-determinant of $M_{I}^I$.

Now consider those permutations which do not fix $k$. By removing $k$ from its cycle, we define a $(k-1)$-to-$1$ correspondence to all permutations on $(k-1)$ elements.
Consider such a permutation~$\sigma$ and the cycle $c$ containing $k$. It is locally of the form $\cdots\to a\to k\to b\to\cdots$. Let $\tilde{\sigma}$ and $\tilde{c}$ be the permutation and cycles obtained by the above map. The signature of~$\tilde{\sigma}$ is the opposite of that of $\sigma$.

Since $M^2=M$, we have $\sum_{i_k=1}^d M_{i_ai_k}M_{i_ki_b}=M_{i_ai_b}$, so that, by $\Z$-linearity of $\tau$, 
\begin{equation}\label{eq:linearity}
\sum_{i_k=1}^d\varepsilon(\sigma)\tau(\cdots M_{i_ai_k}M_{i_ki_b}\cdots)=-\varepsilon(\tilde{\sigma})\tau(\cdots M_{i_ai_b}\cdots)\,.
\end{equation}

By grouping all permutations not fixing $k$ in the left-hand side of~\eqref{eq:dyson} according to their image~$\tilde{\sigma}$ and using the simplification~\eqref{eq:linearity}, we therefore get an overall contribution of \[(-1)(k-1){\det}_{\tau}(M_{I}^I)\]
from all the permutations not fixing $k$. 

Summing the two contributions yields the equality~\eqref{eq:dyson}.
\end{proof}

As a consequence of Proposition \ref{prop:Dyson}, \old{let us observe that }if $\mathrm{rk}_\tau(M)$ denotes the smallest integer such that all principal minors of $M$ (possibly with multiple indices) of that size are zero (if this integer exists), then necessarily $\Tr_\tau(M)=\mathrm{rk}_\tau(M)$.

Furthermore, by applying Proposition~\ref{prop:Dyson} inductively on $k$ running from $n$ down to $1$ we obtain the following. 

\begin{corollary}Let $M\in M_d(R)$ be such that $M^2=M$. Then for all $n\ge 1$, we have
\label{coro:dyson}
\[\sum_{I\in \{1,\ldots, d\}^n}{\det}_{\tau}(M_I^I)=\Tr_\tau(M)(\Tr_\tau(M)-1)\cdots(\Tr_\tau(M)-n+1)\,.\]
\end{corollary}

Before specialising to $\tau=\Re$, let us give a simple formula for the characteristic polynomial of a matrix with respect to the $\tau$-determinant. This will be handy to compute the Laplace transform when we specialise to $\tau=\Re$.

\begin{proposition}\label{prop:tau-characteristic}
Let $M\in M_d(R)$ be any matrix, and let $X$ denote the diagonal matrix with entries $x_1,\ldots, x_d$ in the ring $R$. Then
\[{\det}_{\tau}(X+M)=\sum_{I\subseteq \{1,\ldots, d\}}\left(\prod_{i\in \{1,\ldots,d\}\setminus I}\tau(x_i)\right){\det}_{\tau}\left(M_{I}^I\right)\,.\]
\end{proposition}
\begin{proof}
We use $\Z$-multilinearity to write
\[{\det}_{\tau}(X+M)=\sum_{I\subseteq\{1,\ldots,d\}}{\det}_{\tau}(M\overset{I}{\!\shuffle\!}X)\]
where $M\overset{I}{\!\shuffle\!}X$ denotes the matrix obtained from $M$ by replacing each column not indexed in $I$ by the corresponding column of $X$. We then use the fact that $\tau(0)=0$, so that when one of the columns is zero except for its diagonal coefficient, the only permutations which contribute are the ones for which the index of the column is a fixed point.
\end{proof}

\subsubsection{Construction of quaternionic DPP}

We now specialise the previous statements to $\det_\Re$ for quaternionic Hermitian matrices and show that they imply, in the case of a quaternionic orthogonal projection, a well-defined notion of determinantal point process. 

Let $E$ be a quaternionic vector space split in lines $\sigma=(E_1, \ldots, E_s)$ and let $H$ be a subspace of~$E$. Consider an arbitrary choice of unit vector in each line and let $K=\proj{H}=K^2$ be the quaternion Hermitian matrix of the orthogonal projection on $H$ in the corresponding orthonormal basis. 

Without resorting to quaternion linear algebra, we can compare $K$ to its complex version $\hat{K}$ and deduce elementary properties on the rank and the sign of the principal minors.
We thus find that $\Tr_\Re(K)=n$ where $n=\dim_{\bH}(H)$, and that all principal minors of $K$ are nonnegative.

Furthermore, note that if $I$ is a multiset containing an index twice, which we assume without loss of generality to be $i_1$, then, letting $X=(1,-1,0,\ldots,0)$, we have $K_I^IX=0$. Hence $\hat{K_I^I}\hat{X}=0$ and this implies that $\det(K_I^I)=0$. Hence ${\det}_{\Re}(K_I^I)=0$ by Proposition~\ref{prop:Qdet-square}. 

This means that we can rewrite Corollary~\ref{coro:dyson} as
\begin{equation}\label{eq:QDPP}
\sum_{1\le i_1<\ldots<i_n\le d}{\det}_{\Re}\left(K_{\{i_1,\ldots,i_n\}}^{\{i_1,\ldots,i_n\}}\right)=1\,,
\end{equation}
which we read as the definition of a probability measure on unordered $n$-subsets of $\{1,\ldots, d\}$, just as we did with \eqref{eq:DPPPyth}. 

Now we can check that the incidence probabilities are indeed given by principal minors of~$K$ using Proposition~\ref{prop:Dyson}. Thus, \eqref{eq:QDPP} constructs a determinantal probability measure $\mu_{\sigma, \proj{H}}$ on $\Gr(E,\sigma)$ in the sense of Definition~\ref{def:dpmg-quaternion}.

In order to extend the construction from projection kernels to general ones, but still working with a splitting in lines, we use the well-known fact (see e.g. \cite[Section 8]{Lyons-DPP}) that any kernel is the compression of a projection kernel. This works in the quaternion case just as well. If $\op{k}$ is a kernel on $E$, the operator on $E\oplus E$ defined in block form by 
\[\proj{}=\begin{pmatrix} \op{k} & \sqrt{\op{k}(1-\op{k})} \\\sqrt{\op{k}(1-\op{k})}  & 1-\op{k}  \end{pmatrix}\]
is a projection operator. The fact, proved just above, that $\proj{}$ defines a DPP implies, by compression, that $\op{k}$ defines one as well.

We have now proven the existence of quaternionic DLP for a general kernel $\op{k}$ and splittings in lines $\sigma$: we denote these measures by $\mu_{\sigma,\op{k}}$.

Let us note that, by inclusion-exclusion and Proposition \ref{prop:tau-characteristic}, the density of the measure $\mu_{\sigma,\op{k}}$ with respect to the counting measure $\nu^{E,\sigma}$ is given by
\begin{equation}\label{eq:densityQDPP}
Q\mapsto (-1)^{\dim Q^{\perp}}{\det}_{\Re}(-\proj{Q^{\perp}}+\op{k}).
\end{equation}

\subsubsection{Construction of quaternionic DLP}

We are now ready to give a practical construction of determinantal probability measures on a quaternionic Grassmannian (recall Definition~\ref{def:dpmg-quaternion}) for any splitting. 

Let $E$ be a split quaternionic inner product space with splitting $\sigma=(E_1, \ldots, E_s)$ and let~$\op{k}$ be a kernel. We consider the random subspace obtained by sampling a quaternionic DPP in a random uniform orthonormal basis and with kernel the matrix $K$ of $\op{k}$ in that basis. 

By linearity of the map which to a probability measure on $\Gr(E,\sigma)$ associates its incidence measure, the incidence measure of this random linear subspace of $(E,\sigma)$ is $R\mapsto {\det}_{\Re}(\op{k}^{R}_{R})$ with respect to  $\nu^{E,\sigma}$.

This concludes the proof of Theorem \ref{thm:qDLP}. We denote by $\mu_{\sigma,\op{k}}$ the distribution of the unique DLP on $(E,\sigma)$ with kernel $\op{k}$.

\begin{example}\label{ex:QQSF}
A quantum spanning forest on a graph with connection $h$, whose holonomies take values in the symplectic group $Sp(N)$ for some $N\ge 1$, is a quaternionic DLP associated to a projection kernel on the space of twisted exact forms $\bigstar^1_h$ (in the notation of Section \ref{example:UST-QSF}); see~\cite{KL4}.

In the case $N=1$, where the holonomy group is $Sp(1)=SU(2)$, and studied in \cite{Kenyon, Kassel-Doctorat, KL4}, this DLP is in bijection with the set of edges of a random cycle-rooted spanning forest on a graph sampled with probability proportional to the product of its edge-weights times a certain function of the holonomy of its cycles (the product over the cycles of $2$ minus the trace of their holonomy).
\end{example}

By Proposition~\ref{prop:averaging} and \eqref{eq:densityQDPP}, the measure $\mu_{\sigma,\op{k}}$ has density 
\begin{equation}\label{eq:densityQDLP}
Q\mapsto (-1)^{\dim Q^{\perp}}{\det}_{\Re}(-\proj{Q^{\perp}}+\op{k})
\end{equation}
with respect to $\nu^{E,\sigma}$ (the quaternionic analogue of~\eqref{eq:densitymuKHerm}). Indeed, the content of Proposition~\ref{prop:averaging} does not depend on the precise density of the measures and holds for the quaternion determinant just as well.

A number of properties detailed in Section \ref{sec:geometry-DLP} are true for quaternionic DLP. We leave it to the interested reader to check which proofs do carry through in the quaternionic case.

Nevertheless, let us state one simple result about the Laplace transform of the split dimension of a quaternionic DLP. This is the quaternionic analogue of Proposition \ref{prop:dimQ}.

\begin{proposition}\label{prop:LaplaceQ}
Let $(t_1,\ldots,t_s)\in\R^s$ and write $e^T-1$ for the  scaling operator acting by multiplication by $e^{t_i}-1$ in block~$E_i$. If $\Q$ follows distribution $\mu_{\sigma,\op{k}}$, then
\begin{equation}
\E\left[e^{\sum_{i=1}^s t_i \dim_{\bH}(\Q\cap E_i)}\right]={\det}_{\Re}(1+(e^T-1)\op{k})\,.
\end{equation}
\end{proposition}

A further, more subtle, result we emphasise is the following.

\begin{theorem}[Uniqueness and mean projection]\label{thm:qA}
Let $(E,\sigma)$ be a quaternionic inner product space and $H$ a subspace of $E$. Let $\Q$ be a random subspace following the distribution $\mu_{\sigma,\proj{H}}$ in $\Gr(E,\sigma)$. Then, almost surely, $E=\Q\oplus H^\perp$, and moreover
\[\E\big[\op{P}^{\Q}+(\op{P}^{\Q})^*\big]=2\proj{H}\,,\]
where $\op{P}^{\Q}$ is the projection on $\Q$ parallel to $H^\perp$.
\end{theorem}

\begin{proof}
The proof of the first statement is not hard to adapt from the real and complex cases and we leave it to the interested reader. 
For the second statement, we follow a slightly different strategy than the one of the proof of Theorem~\ref{prop:meanprojection}. 

First note that,  by \eqref{eq:densityQDLP} and \eqref{eq:Qdet-pfaff}, the density of the measure $\mu_{\sigma,\proj{H}}$ with respect to $\nu^{E,\sigma}$ is \[Q\mapsto \sqrt{\vert{\det}(-\widehat{\proj{Q^{\perp}}}+\widehat{\proj{H}})\vert}\,,\] which, by the remark containing \eqref{eq:densitymuKHerm}, may be rewritten as \[Q\mapsto\sqrt{\det\big(\widehat{\proj{H}}\widehat{\proj{Q}}+\widehat{\proj{H^{\perp}}}\widehat{\proj{Q^{\perp}}}\big)}\,.\]
Note that the determinant of which we take the square root is indeed positive (cf. the beginning of the proof of Proposition \ref{prop:exDLP}).

Moreover, note that if $\hat{F}$ denotes the complex subspace obtained from the quaternionic subspace $F$ by restriction of scalars, then $\widehat{\proj{F}}=\proj{\hat{F}}$. From this remark and Proposition \ref{prop:cosdet}, it follows that 
\[\det\big(\widehat{\proj{H}}\widehat{\proj{Q}}+\widehat{\proj{H^{\perp}}}\widehat{\proj{Q^{\perp}}}\big)=\cos^2(\hat{Q},\hat{H})=\det(\op{1}_{\hat{H}}^{\hat{Q}}\op{1}_{\hat{Q}}^{\hat{H}})=\det(\widehat{\op{1}_H^Q}\widehat{\op{1}_Q^H})\,,\]
where we used the notation for compressions of operators introduced in Section \ref{sec:minors}. The density of $\mu_{\sigma, \proj{H}}$ may thus be rewritten as 
\[Q\mapsto\sqrt{\det(\widehat{\op{1}_H^Q}\widehat{\op{1}_Q^H})}\,.\]
Still using the notations of Section \ref{sec:minors}, note that for any $Q$ in direct sum with $H$, we have 
\[\op{P}^Q=\op{1}_Q\op{1}_H^Q(\op{1}_Q^H\op{1}_H^Q)^{-1}\op{1}^H\,,\]
as indeed it can be checked that this is a projection whose kernel is $H^\perp$ and whose image is~$Q$.

Now let $\op{a}$ be any endomorphism of $E$. If $\varepsilon$ is a complex number of small enough modulus, the complex number $\det\big(\hat{1}+\varepsilon \hat{\op{a}}\widehat{\op{P}^{Q}}\big)$ has positive real part for any $Q\in \Gr(E,\sigma)$, and we will consider its principal square root. 
Note that, as a consequence of the identity $\det(I+AB)=\det(I+BA)$ for rectangular matrices of compatible dimensions $A,B$, we have
\[\det\big(1+\varepsilon \op{a}\op{P}^{Q}\big)=\det\big(\op{1}_H^H+\op{1}^H \varepsilon \op{a}\op{1}_Q\op{1}_H^Q(\op{1}_Q^H\op{1}_H^Q)^{-1}\big)\,.\]

Letting $n$ be the quaternion dimension of~$H$, we thus obtain
\[\E\left[\sqrt{\det\big(\hat{1}+\varepsilon\hat{\op{a}}\widehat{\op{P}^{\Q}}\big)}\,\right]=\int_{\Gr_{n}(E,\sigma)}\sqrt{\det\big(\widehat{\op{1}_H^H}+\widehat{\op{1}^H} \varepsilon \widehat{\op{a}}\widehat{\op{1}_Q}\widehat{\op{1}_H^Q}(\widehat{\op{1}_Q^H}\widehat{\op{1}_H^Q})^{-1}\big)}\sqrt{\det(\widehat{\op{1}_H^Q}\widehat{\op{1}_Q^H})}\d\nu^{E,\sigma}_n(Q).\]
We now multiply the (usual, not quaternion) determinants, and find
\[\E\left[\sqrt{\det\big(\hat{1}+\varepsilon\hat{\op{a}}\widehat{\op{P}^{\Q}}\big)}\,\right]=\int_{\Gr_{n}(E,\sigma)}\sqrt{\det\big(\widehat{\op{1}^H}(1+\varepsilon \hat{\op{a}})\widehat{\op{1}_Q}\widehat{\op{1}^Q}\widehat{\op{1}_H}\big)}\d\nu^{E,\sigma}_n(Q)\,.\]
Let us now assume that $\op{a}$ is self-adjoint and $\varepsilon$ is real and small enough for the above determinants to be positive. In that case, letting $\op{b}$ be an endomorphism of $E$ such that $1+\varepsilon \op{a}=\op{b}^* \op{b}$, using multiplicativity of the usual determinant, noting that $\op{1}^H\op{b}^* \op{1}_Q \op{1}^Q \op{b}\op{1}_H$ is self-adjoint, and using~\eqref{eq:Qdet-pfaff}, we have shown that 
\begin{equation}\label{eq:bb}
\E\left[\sqrt{\det\big(\hat{1}+\varepsilon\hat{\op{a}}\widehat{\op{P}^{\Q}}\big)}\,\right]=\int_{\Gr_{n}(E,\sigma)}{\det}_{\Re}\big(\op{1}^H\op{b}^* \op{1}_Q \op{1}^Q \op{b}\op{1}_H\big) \d\nu^{E,\sigma}_n(Q)\,.
\end{equation}
By a quaternion analogue of the invariant Cauchy--Binet formula \eqref{CB} (this formula is proved using the quaternion Hermitian matrix analogue of the Cauchy--Binet formula -- which appears in various places in the literature, see e.g. \cite[Proposition A.3 (g)]{Sokal} -- and averaging over the action of $\U(E,\sigma)$ on the choice of orthonormal basis like in the proof of Proposition \ref{prop:CB+}), the right-hand side of \eqref{eq:bb} is now equal to ${\det}_{\Re}(\op{1}^H(1+\varepsilon \op{a}) \op{1}_H)$. 

Now, using \eqref{eq:Qdet-pfaff} again to convert to usual determinants, and using the equality
$\det(\op{1}_H^H + \op{1}^H \varepsilon \op{a}\op{1}_H)=\old{\det(1 +  \varepsilon \op{a}\op{1}^H_H)=}\det(1+ \varepsilon \op{a}\proj{H})$ (which is again an instance of the equality $\det(I+AB)=\det(I+BA)$), we find that for any self-adjoint $\op{a}$, and real small enough $\varepsilon$, we have 
\begin{equation}
\E\left[\sqrt{\det(\hat{1}+\varepsilon \hat{\op{a}}\widehat{\op{P}^{\Q}})}\,\right]=\sqrt{\det\big(\hat{1}+ \varepsilon \hat{\op{a}}\widehat{\proj{H}\big)}}\,.
\end{equation}
By differentiating with respect to $\varepsilon$ at $0$, we finally arrive at the following statement. For any self-adjoint $\op{a}$, we have $\E\big[\Tr(\hat{\op{a}}\widehat{\op{P}^\Q})\big]=\Tr(\hat{\op{a}}\widehat{\proj{H}})$, which implies $\E\big[\Tr(\op{a}\op{P}^\Q)\big]=\Tr(\op{a}\proj{H})$. Since~$\proj{H}$ is self-adjoint itself, this implies the announced equality.
\end{proof}

It would be nice to have a refinement of Theorem~\ref{thm:qA} as in Theorems~\ref{prop:meanprojection} or~\ref{thm:wedgeA}. For the moment we could not find one, by lack of a notion of quaternion determinant for non self-adjoint operators, and by lack of a notion of exterior algebras for quaternionic vector spaces.

\subsection{From quaternion to complex coefficient field}

We conclude this section by coming back to the discussion of Section~\ref{sec:complex-real}. We have the following analogue of Proposition~\ref{prop:complex-reals} in the quaternion-to-complex case.

\begin{proposition}\label{prop:quaternions-complex}
Let $E$ be a quaternionic inner product space, given with a splitting $\sigma=(E_1, \ldots, E_s)$ and a kernel $\op{k}$. Let $\hat{E}$, $\hat{\sigma}$, and $\hat{\op{k}}$ be the data induced by restricting the coefficient field to complex numbers. Let $\Q^{(1)}$ and $\Q^{(2)}$ be two independent random quaternionic linear subspaces of~$E$ sampled from $\mu_{\sigma,\op{k}}$. Let $\rC$ be a random complex linear subspace of $\hat{E}$ sampled from~$\mu_{\hat{\sigma},\hat{\op{k}}}$. Then we have the equality in distribution
\[\left(\dim_{\C} \rC_i\right)_{1\leq i \leq s}\stackrel{\text{\rm(law)}}{=}\left(\dim_\bH \Q^{(1)}_i+\dim_\bH \Q^{(2)}_i\right)_{1\leq i \leq s}\,.\]
\end{proposition}
\begin{proof}
We compute the Laplace transform of the left-hand side using Proposition \ref{prop:dimQ}, use~\eqref{eq:Qdet-pfaff} and Proposition \ref{prop:LaplaceQ} to recognize that it is equal to the square of the Laplace transform of each of the two summands of the right-hand side.
\end{proof}

%%%%%%%%%%%%%%%%%%%%%%%%%%%%%%%%%%%%
%%%%%%%%%%%%%%%%%%%%%%%%%%%%%%%%%%%%
%%%%%%%%%%%%%%%%%%%%%%%%%%%%%%%%%%%%

\section*{Concluding remarks}

The main example and motivation for our introducing the theory of DLP are quantum spanning forests (see Section~\ref{example:UST-QSF} and Figure~\ref{fig:exQSF}). It would be interesting to find other meaningful examples of DLP. Are there such examples coming from representation theory, related to the theory of Schur processes~\cite{Okounkov}?

Is it possible to delve deeper into the connection with matroid theory provided by Proposition~\ref{prop:lafforgue} and the geometry of the Grassmannian, maybe in the spirit of the recent work on probabilistic Schubert calculus~\cite{Burgisser}?

DPP have been widely used in statistical learning theory starting with the work of~\cite{Kulesza-Taskar}; see also~\cite{Bardenet}. The main focus of study has been on signals which can be represented as point processes, such as 2D visual images. Could one use DLP to study signals of a more vectorial nature, such as sound recordings for instance? See e.g. \cite{Lim} for an approach to subspace learning, and references therein for examples of subspace-valued data.

Coming back to Macchi's foundational work~\cite{Macchi} who introduced DPP as a way to model systems of fermions in quantum optics, could one use DLP to model fermionic observables in a real physical system with internal symmetry or coupled to a gauge field?

As briefly mentioned in the opening lines of this paper, there is a rich continuous theory of DPP. In view of the fact that continuous DPP can be thought of as scaling limits of discrete DPP, is there a continuous theory of DLP? A convenient way of thinking about point processes in continuous spaces is as a random measure. One might imagine developing a framework for a process which would be a random measure, whose values instead of being real would be sections of the Grassmannian bundle of a given vector bundle. Are there such determinantal measures? Could Corollary~\ref{cor:loiDLPext} provide an approach for defining DLP with a kernel given by a trace-class operator? What would be an analog of Theorems~\ref{prop:meanprojection} and~\ref{thm:wedgeA} in that context?

%%%%%%%%%%%%%%%%
%       BIBLIOGRAPHY           %
%%%%%%%%%%%%%%%%

\def\@rst #1 #2other{#1}
\renewcommand\MR[1]{\relax\ifhmode\unskip\spacefactor3000 \space\fi
  \MRhref{\expandafter\@rst #1 other}{#1}}
\renewcommand{\MRhref}[2]{\href{http://www.ams.org/mathscinet-getitem?mr=#1}{MR#1}}
\newcommand{\arXiv}[1]{\href{http://arxiv.org/abs/#1}{arXiv:#1}}

\bibliographystyle{hmralphaabbrv}
\bibliography{DPP-biblio}

\end{document}